\documentclass[11pt]{amsart}
\usepackage[dvipsnames,usenames]{color}
\usepackage{hyperref}
\usepackage{graphicx}
\usepackage{epsfig}
\usepackage[latin1]{inputenc}
\usepackage{amsmath}
\usepackage{amsfonts}
\usepackage{amssymb}
\usepackage{amsthm}
\usepackage{amscd}
\usepackage{verbatim}
\usepackage{caption}
\usepackage{subcaption}
\usepackage{pinlabel}
\usepackage{stmaryrd}
\usepackage{enumerate, enumitem}
\usepackage{todonotes}
\usepackage{bm}
\usepackage{thmtools}
\usepackage{thm-restate}
\usepackage{lipsum}
\usepackage{setspace}

\usepackage{tikz}
\usetikzlibrary{arrows}
\usetikzlibrary{decorations.pathreplacing}
\usepackage{verbatim}
\usetikzlibrary{cd}

\newtheorem{theorem}{Theorem}[section]

\newtheorem{lemma}[theorem]{Lemma}

\theoremstyle{definition}
\newtheorem{definition}[theorem]{Definition}

\theoremstyle{remark}
\newtheorem{remark}[theorem]{Remark}
\newtheorem{example}[theorem]{Example}

\def\F{\mathbb{F}}
\def\N{\mathbb{N}}

\def\Z{\mathbb{Z}}

\def\Zbang{\Z^!}
\def\lebang{<^!}
\def\gebang{>^!}
\def\leqbang{\leq^!}
\def\geqbang{\geq^!}

\def\cC{\mathcal{C}}

\def\CFKi{\CFK^{\infty}}

\def\Im{\operatorname{Im}}

\def \gr {\operatorname{gr}}

\def\d{\partial}
\def\varep{\varepsilon}
\def\co{\colon}
\def\from{\colon}

\def\id{\textup{id}}
\def\Inv{\mathfrak{I}}

\def\im{\operatorname{im}}

\def\CF {\mathit{CF}}
\def\HF {\mathit{HF}}
\newcommand\HFhat{\widehat{\HF}}

\newcommand \CFm {\CF^-}
\newcommand \HFm {\HF^-}

\def\CFK{\mathit{CFK}}

\def\ff{\mathbb{F}}

\def\inv{\iota}

\def\mfIhat{\widehat{\mathfrak{I}}}
\newcommand{\real}{\tilde{\mathfrak{I}}} %The local equivalence group of realizable AICs.  Definitely would like to change both the macro and the symbol for this.

\def\mfI{\mathfrak{I}}
\def\cCFK{\mathcal{CFK}}
\def\iotabar{\overline{\iota}}

\def\cA{\mathcal{A}}

\def\rC{\overline{\cC}}
\def\omegahat{\widehat{\omega}}
\newcommand{\shift}{\mathrm{sh}} % The shift endomorphism
\newcommand{\omegahom}{H_\omega} %omega-homology
\author[I. Dai]{Irving Dai}
\thanks{The first author was partially supported by NSF grant DGE-1148900.}
\address {Department of Mathematics, Princeton University, Princeton, NJ 08544}
\email{idai@math.princeton.edu}

\author[J. Hom]{Jennifer Hom}
\thanks{The second author was partially supported by NSF grant DMS-1552285 and a Sloan Research Fellowship.}
\address {School of Mathematics, Georgia Institute of Technology, Atlanta, GA 30332}
\email{hom@math.gatech.edu}

\author[M. Stoffregen]{Matthew Stoffregen}
\thanks{The third author was partially supported by NSF grant DMS-1702532.}
\address {Department of Mathematics, Massachusetts Institute of Technology, Cambridge, MA 02142}
\email{mstoff@mit.edu}

\author[L. Truong]{Linh Truong}
\thanks{The fourth author was partially supported by NSF grant DMS-1606451.}
\address {Department of Mathematics, Columbia University, New York, NY 10027}
\email{ltruong@math.columbia.edu}

\numberwithin{equation}{section}

\title{An infinite-rank summand of the homology cobordism group}

\begin{document}

\begin{abstract} 
We show that the three-dimensional homology cobordism group admits an infinite-rank summand. It was previously known that the homology cobordism group contains a $\Z^\infty$-subgroup \cite{Furuta, FintushelStern} and a $\Z$-summand \cite{Froyshov}.
Our proof proceeds by introducing an algebraic variant of the involutive Heegaard Floer package of Hendricks-Manolescu and Hendricks-Manolescu-Zemke. This is inspired by an analogous argument in the setting of knot concordance due to the second author.
\end{abstract}

\maketitle

\section{Introduction}

The integral homology cobordism group $\Theta^3_\Z$ has occupied a central place in the development of smooth four-manifold topology. Historically, the first known result concerning the structure of $\Theta^3_\Z$ was the existence of the Rokhlin homomorphism $\mu \co \Theta^3_\Z \to \Z/2\Z$, which at the time was even conjectured to be an isomorphism. However, in the 1980s, techniques from gauge theory were utilized to show that $\Theta^3_\Z$ is infinite \cite{FintushelSternPseudo} and in fact contains a $\Z^\infty$-subgroup \cite{Furuta, FintushelStern}. In \cite{Froyshov}, Fr{\o}yshov used Yang-Mills theory to define a surjective homomorphism from $\Theta^3_{\Z}$ to $\Z$, proving that $\Theta^3_\Z$ has a $\Z$-summand. Analogous homomorphisms have been constructed using Seiberg-Witten theory \cite{FroyshovMonopole, KMmonopoles} and Heegaard Floer homology \cite{OSabsgr}.\footnote{The Seiberg-Witten and Heegaard Floer $d$-invariants are known to agree (after appropriate normalization); see \cite{CGH, KLT, Taubes, HuangRamos, Cristofaro-Gardiner}. It is still open whether they contain the same information as the Fr\o yshov $h$-invariant.}

More recently, Manolescu \cite{ManolescuTriangulation} used Pin(2)-equivariant Seiberg-Witten Floer theory to show that if $\mu(Y) = 1$, then $Y$ is not of order two in $\Theta^3_\Z$. By work of Galewski-Stern \cite{GalewskiStern} and Matumoto \cite{Matumoto}, this led to a disproof of the triangulation conjecture in high dimensions. See \cite{ManolescuHomcobordsurvey} for a survey on the triangulation conjecture and the homology cobordism group.

Many questions involving the structure of $\Theta^3_\Z$ remain open. Chief among these is the problem of whether $\Theta^3_\Z$ contains any torsion, or whether modulo torsion it is free abelian. We give a partial answer to the latter in the following theorem:

\begin{theorem}\label{thm:main} 
The homology cobordism group $\Theta^3_\Z$ contains a direct summand isomorphic to $\Z^\infty$.
\end{theorem}

The proof of Theorem \ref{thm:main} relies on the machinery of involutive Heegaard Floer homology, defined by Hendricks and Manolescu in \cite{HM:involutive}. This is a modification of the usual Heegaard Floer homology of Ozsv\'ath and Szab\'o which takes into account the additional data of a homotopy involution $\iota$ on the Heegaard Floer complex $\CFm(Y)$. We refer to the pair $(\CFm(Y), \inv)$ as an \textit{$\inv$-complex}. In \cite{HMZ:involutive-sum}, Hendricks, Manolescu, and Zemke showed that up to an algebraic equivalence relation called \textit{local equivalence}, the pair $(\CFm(Y), \iota)$ is an invariant of the homology cobordism class of $Y$. Motivated by this, they defined a group $\Inv$ consisting of the set of all possible (abstract) $\inv$-complexes modulo the relation of local equivalence, with operation induced by tensor product. In \cite{HMZ:involutive-sum}, it was shown that the map
\[
h: Y \mapsto h(Y) = (\CFm(Y)[-2], \inv)
\]
sending $Y$ to the local equivalence class of its (grading-shifted) $\inv$-complex induces a homomorphism from $\Theta^3_\Z$ to $\mfI$.

The construction of $\mfI$ has antecedents in work of the third author as well as work of the second author. In \cite{Stoffregen}, the third author defined a similar group $\mathfrak{CLE}$ in the Seiberg-Witten Floer setting, by considering Pin(2)-equivariant Seiberg-Witten Floer complexes modulo an algebraic equivalence relation called chain local equivalence. In the setting of knot concordance, the second author \cite{Homconcordance} considered the group $\cCFK$, generated by certain bifiltered chain complexes modulo an equivalence relation called $\varep$-equivalence, and showed that the map sending a knot $K$ to $\CFKi(K)$ induces a homomorphism from the knot concordance group to $\cCFK$. See Zemke \cite[Theorem 1.5]{Zemkeconnsuminv} for an analogous construction in the case of involutive knot Floer homology. 

%Note that the group $\cCFK$ is a quotient of $\mfI_K$, and the projection factors through the group obtained from $\mfI_K$ by forgetting the involutive aspect.

A key feature of $\cCFK$ (which the groups $\mfI$ and $\mathfrak{CLE}$ lack) is the presence of a total order that is compatible with the group structure. There is a natural partial order on $\mfI$, which is defined by setting $(C_1, \iota_1) \leq (C_2, \iota_2)$ whenever there exists an $\iota$-complex morphism
\begin{equation} \label{eq:pomorph}
	f \co (C_1, \iota_1) \to (C_2, \iota_2)
\end{equation}
that induces an isomorphism on $U^{-1}H_*$. However, since there is 2-torsion in $\mfI$, this partial order is not a total order.

In this paper, our strategy will be to define an algebraic modification $\smash{\mfIhat}$ of $\Inv$, consisting of equivalence classes of what we call \textit{almost $\iota$-complexes}. This group comes with its own homomorphism 
\[
\widehat{h}: \Theta^3_{\Z} \rightarrow \mfIhat
\]
factoring through $h: \Theta^3_{\Z} \rightarrow \Inv$. The set of $\iota$-complexes is a proper subset of the set of almost $\iota$-complexes, but the equivalence relation on $\smash{\mfIhat}$ is coarser than the equivalence relation on $\Inv$. The key difference between the two is that an almost $\iota$-complex is a pair $(C, \iotabar)$, where $\iotabar$ is no longer required to be a chain map or a homotopy involution. Instead, we require that these two conditions hold modulo the ideal generated by $U$. While this might initially sound rather unfavorable, we will show that the partial order on $\mfI$ translates to a total order on $\smash{\mfIhat}$.

\begin{remark}\label{rmk:CFK}
The total order on $\cCFK$, while originally defined in terms of $\varep$, can equivalently be defined in terms of morphisms, as in \eqref{eq:pomorph}. Indeed, following Zemke \cite{Zemkequasistab}, one may recast $\CFKi(K)$ as a chain complex over the ring $\F[U, V]$. If one then quotients by the ideal generated by $UV$ (i.e., works over $\F[U, V]/UV$), then the analogue to \eqref{eq:pomorph} induces the total order on $\cCFK$. Indeed, most of the constructions in this paper have analogues in $\cCFK$, which will be explored in upcoming work \cite{DHST}. 
\end{remark}

Our main technical result will be to give an explicit parameterization of $\smash{\mfIhat}$ as an ordered set. Each element of $\smash{\mfIhat}$ will be described by a finite tuple of nonzero integers, with the order being given by the lexicographic order. Using this description, for each $n \in \N$, we define a map
\[ \phi_n \co \mfIhat \to \Z, \]
where (roughly speaking) $\phi_n(X)$ counts the number of times $n$ appears in the tuple describing $X$. (Intuitively, this corresponds to the signed count of $\F[U]/U^n$ summands in the homology of $X$; see Sections~\ref{sec:4} and \ref{sec:7} for details.) In Section~\ref{sec:7}, we show that the $\phi_n$ are homomorphisms. It follows that each composition
\[
f_n = \phi_n \circ \widehat{h} \co \Theta^3_{\Z} \rightarrow \Z
\]
is also a homomorphism, and thus that the family 
\[
\vec{{}f} = \{f_n\}_{n \in \N} \co  \Theta^3_{\Z} \rightarrow \Z^{\infty}
\]
provides a homomorphism from $\Theta^3_{\Z}$ into $\Z^\infty$.

\begin{proof}[Proof of Theorem \ref{thm:main}]
For each $i > 0$, consider the Brieskorn sphere 
\[
Y_i = \Sigma(2i+1, 4i +1, 4i+3).
\]
The proof proceeds by observing that $f_j(Y_i) = \delta_{ij}$, where $\delta_{ij}$ denotes the Kronecker delta. (See Example~\ref{ex:xicomplex} and Theorem~\ref{thm:hhatzinfty} for details.) Hence
\[
\vec{f{}} = \{f_n\}_{n \in \N} \co  \Theta^3_{\Z} \rightarrow \Z^{\infty}
\]
is a surjective homomorphism. We thus see that the $Y_i$ generate an infinite-rank free summand of $\Theta^3_\Z$, as desired.
\end{proof}

While we give an explicit parameterization of $\smash{\mfIhat}$ as a set, a complete description of the group structure on $\smash{\mfIhat}$ remains open. In Section~\ref{sec:8}, we establish some partial results to this effect and discuss the relationship between $\smash{\mfIhat}$ and $\Inv$. We also describe some possible applications to other questions involving the span of Seifert fibered spaces in $\Theta^3_{\Z}$.

\begin{remark}
The homomorphisms $f_n$ are in fact spin homology cobordism invariants. Let $\Theta^3_{\Z_2}$ be the group of oriented $\Z_2$-homology spheres modulo $\Z_2$-homology cobordism. Then a similar argument as above shows that $\smash{\vec{f{}}}$ is a surjective homomorphism from $\Theta^3_{\Z_2}$ to $\Z^{\infty}$, proving that $\Theta^3_{\Z_2}$ also admits an infinite-rank summand.
\end{remark}

\subsection*{Organization} In Section~\ref{sec:2}, we give a brief overview of the formalism of involutive Heegaard Floer homology and the construction of $\Inv$. In Section~\ref{sec:3}, we define the modified group $\smash{\mfIhat}$ and prove that it admits a total order. Section~\ref{sec:4} is then devoted to constructing a particularly important family of almost $\inv$-complexes, which we use to give an explicit parameterization of $\smash{\mfIhat}$ over the course of Sections~\ref{sec:5} and \ref{sec:6}. In Section~\ref{sec:7}, we show that the $\phi_n$ are homomorphisms. Finally, in Section~\ref{sec:8}, we give some sample calculations and discuss the relationship between $\Inv$ and $\smash{\mfIhat}$.

\subsection*{Acknowledgements} We would like to thank Kristen Hendricks, Tye Lidman, and Ciprian Manolescu for helpful conversations. The first author would like to thank his advisor, Zolt\'an Szab\'o, for his continued support and guidance. We are also grateful for the Topologie workshop at Oberwolfach in July 2018, during which part of this work was completed.

%%%%%%%%%%%%%%%%%%%%%%%%%%%%%%%%%%%%%%%%%%%%%%%%%%%%%%%%%%%%%%%%%%%%%%%%%%%%%%%%%%%%%%%%%%%%%%%%%%%%%%%%%%%%%%%%%%%%%%%%%%%%%%%%%%%%%%%%%%%%%%%%%%%%%%%%%%%%%%%%%%%%%%%%%%%%%%%%%%%%%%%%%%%%%%%%%%%%%%%%%%%%%%%%%%%%%%%%%%%%%%%%%%%%%%%%%%%%%%%%%%%%%%%%%%%%%%%%%%%%%%%%%%%%%%%%%%%%%%%%%%%%%%%%%%%%%%%%%%%%%%%%%%%%%%%%%%%%%%%%%%%%%%%%%%%%%%%%%%%%%%%%%%%%%%%%%%%%%%%%%%%%%%%%%%%%%%%%%%%%%%%%%%%%%%%%%%%%%%%%%%%%%%%%%%%%%%%%%%%%%%%%%%%%%%%%

\section{Involutive Heegaard Floer Homology}\label{sec:2}
In this section, we briefly review the involutive Heegaard Floer package defined by Hendricks-Manolescu \cite{HM:involutive} and Hendricks-Manolescu-Zemke \cite{HMZ:involutive-sum}. For the present application, we will only need to understand the algebraic formalism of the output; we refer the interested reader to \cite{HM:involutive} instead for the topological details of the construction. 

\subsection{Local equivalence and $\mfI$}
We begin with the definition of involutive Heegaard Floer homology.

\begin{definition}\cite[Definition 8.1]{HMZ:involutive-sum}\label{def:iotacomplex}
An {\em $\inv$-complex} $\cC = (C, \inv)$ consists of the following data:
\begin{itemize}
\item A free, finitely-generated, $\Z$-graded chain complex $C$ over $\ff[U]$, with
\[
U^{-1}H_*(C) \cong \ff[U, U^{-1}].
\]
Here, $U$ has degree $-2$ and $U^{-1}H_*(C)$ is supported in even gradings.
\item A grading-preserving, $U$-equivariant chain homomorphism $\inv: C \to C$ such that $\inv^2$ is chain homotopic to the identity; i.e.,
\[
\inv^2 \simeq \id.
\]
\end{itemize}
Two $\inv$-complexes are \textit{homotopy equivalent} if there exist (grading-preserving) homotopy equivalences between them, which are themselves $\inv$-equivariant up to homotopy. Throughout this paper, we write
\[
\omega = 1 + \inv.
\] 
\end{definition}

\begin{remark}\label{rem:gradingconvention}
In future sections, we will also impose the condition that the maximally graded $U$-nontorsion element in $H_*(C)$ has degree zero. In more topological terms, this means that we always shift gradings so as to make the $d$-invariant zero.
\end{remark}

Given an integer homology sphere $Y$, Hendricks and Manolescu define a homotopy involution $\inv$ on the Heegaard Floer complex $\CFm(Y)$ by using the involution on the Heegaard diagram interchanging the $\alpha$- and $\beta$-curves. They then show that the map sending
\[
Y \mapsto (\CFm(Y), \inv)
\] 
is well-defined up to homotopy equivalence of $\inv$-complexes. The involutive Heegaard Floer homology of $Y$ is defined to be the homology of the mapping cone of $\CFm(Y)$ with respect to the map $\omega = 1 + \inv$.\footnote{Involutive Heegaard Floer homology is defined for all 3-manifolds, but in this paper we will only consider integer homology spheres.} In this paper, we will always work with the $\inv$-complexes themselves, rather than their involutive homology.

\begin{definition}\cite[Definition 8.5]{HMZ:involutive-sum}
Two $\inv$-complexes $(C, \inv)$ and $(C', \inv')$ are said to be {\em locally equivalent} if there exist (grading-preserving) chain maps
\[
f : C \to C', \ \ g : C' \to C
\]
between them such that 
\[
f \circ \inv \simeq \inv' \circ f,  \ \ \ g \circ \inv' \simeq \inv \circ g,
\]
and $f$ and $g$ induce isomorphisms on homology after inverting the action of $U$. We refer to $f$ as a \textit{local map} from $(C, \inv)$ to $(C', \inv')$, and similarly refer to $g$ as a local map in the other direction. Note that local equivalence is a strictly weaker notion than the relation of homotopy equivalence between two $\inv$-complexes. 
\end{definition}

In \cite{HMZ:involutive-sum}, it is shown that if $Y_1$ and $Y_2$ are homology cobordant, then their respective $\inv$-complexes are locally equivalent. It is further shown in \cite[Section 8]{HMZ:involutive-sum} that the set of $\inv$-complexes up to local equivalence forms a group, with the group operation being given by tensor product. We call this group the \textit{(involutive Floer) local equivalence group} and denote it by $\Inv$:

\begin{definition}\cite[Proposition 8.8]{HMZ:involutive-sum}
Let $\Inv$ be the set of $\inv$-complexes up to local equivalence. This has a multiplication given by tensor product, which sends (the local equivalence classes of) two $\inv$-complexes $(C_1, \inv_1)$ and $(C_2, \inv_2)$ to (the local equivalence class of) their tensor product complex $(C_1 \otimes C_2, \inv_1 \otimes \inv_2)$. The identity element of $\Inv$ is given by the trivial complex $\cC(0)$ consisting of a single $\ff[U]$-tower starting in grading zero, together with the identity map on this complex. Inverses in $\Inv$ are given by dualizing. See \cite[Section 8]{HMZ:involutive-sum} for details. 
\end{definition}

Let $h$ be the map sending $Y$ to the local equivalence class of its (grading-shifted) $\inv$-complex:
\[
h(Y) = (\CFm(Y), \iota)[-2].\footnote{The grading shift in this definition is so that $Y = S^3$ is mapped to $\cC(0)$.}
\]
According to \cite[Theorem 1.1]{HMZ:involutive-sum}, $h$ takes connected sums to tensor products, and hence effects a homomorphism
\[
h: \Theta_{\Z}^3 \rightarrow \mfI.
\]
See \cite[Theorem 1.8]{HMZ:involutive-sum}.

\begin{remark}
Restricting to the set of $\inv$-complexes satisfying the normalization convention of Remark~\ref{rem:gradingconvention} gives a subgroup of $\mfI$, which we denote by $\mfI_0$. It is easily checked that
\[
\mfI = \mfI_0 \oplus \Z,
\]
with the extra factor of $\Z$ recording the value of the $d$-invariant. In this paper, we abuse notation slightly and refer to $\mfI_0$ as $\mfI$. Similarly, we let $h$ be the usual $h$-map, except followed by projection onto $\mfI_0$. The reader may equivalently think of this as always considering $\inv$-complexes up to an overall grading shift.
\end{remark}

%%%%%%%%%%%%%%%%%%%%%%%%%%%%%%%%%%%%%%%%%%%%%%%%%%%%%%

\subsection{Properties and examples} We now review several background results about $\mfI$ in order to give some motivation for subsequent sections. 

\begin{example}\label{ex:xicomplex}
For any positive integer $i$, let $X_i$ be the $\inv$-complex displayed on the left in Figure~\ref{fig:iotacomplexes}. This is generated (over $\ff[U]$) by three generators $T_0$, $T_1$, and $T_2$, with:
\[
\partial T_0 = \partial T_1 = 0, \ \ \ \partial T_2 = U^i T_1. 
\]
We specify the action of $\inv$ by defining $\omega = 1 + \inv$:
\[
\omega T_0 = T_1, \ \ \ \omega T_1 = \omega T_2 = 0.
\] 
The gradings of $T_0$, $T_1$, and $T_2$ are given by $0$, $0$, and $-2i + 1$, respectively. The dual complex $-X_i$ may be computed via the procedure of \cite[Section 8.3]{HMZ:involutive-sum}; we denote its generators similarly by $T_0$, $T_1$, and $T_2$. These satisfy:
\[
\partial T_0 = \partial T_2 = 0, \ \ \ \partial T_1 = U^i T_2 
\]
and
\[
\omega T_0 = \omega T_2 = 0, \ \ \ \omega T_1 = T_0.
\] 
The gradings of $T_0$, $T_1$, and $T_2$ are given by $0$, $0$, and $2i - 1$, respectively. The complex $-X_i$ is displayed on the right in Figure~\ref{fig:iotacomplexes}.
\end{example}

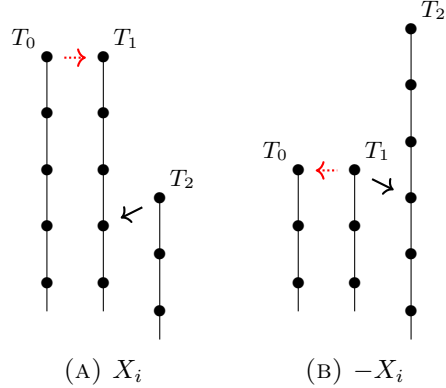
\begin{figure}[h!]
    \centering
     %spacing
    \begin{subfigure}[b]{0.23\textwidth}
    \centering
        \begin{tikzpicture}[scale=.75, yscale=.5]
	\begin{scope}[thin, black]
		\draw [-] (-1, 2) -- (-1, -7); %x
		\draw [-] (0, 2) -- (0, -7); %z
		\draw [-] (1, -3) -- (1, -8); %y
		\node[anchor=south east] (x) at (-1, 2) {\footnotesize $T_0$};
		\node[anchor=south west] (z) at (0, 2) {\footnotesize $T_1$};
		\node[anchor=south west] (y) at (1, -3) {\footnotesize $T_2$};
		\foreach \x in {2, 0, -2,-4,-6}
    			\draw (-1,\x ) -- (-1,\x ) node (\x x) {$\bullet$};
		\foreach \z in {2, 0, -2,-4,-6}
    			\draw (0,\z ) -- (0,\z ) node (\z z) {$\bullet$};
		 \foreach \y in {-3,-5,-7}
    			\draw (1,\y ) -- (1,\y ) node (\y y) {$\bullet$};
		\draw [thick, ->] (-3y) -- (-4z);
		\draw [thick, densely dotted, red, ->] (2x) -- (2z);
	\end{scope}	
\end{tikzpicture}
        \caption{$X_i$}
    \end{subfigure}
     %spacing
    \begin{subfigure}[b]{0.23\textwidth}
    \centering
        \begin{tikzpicture}[scale=.75, yscale=.5]
	\begin{scope}[thin, black]
		\draw [-] (-1, -2) -- (-1, -7); %x
		\draw [-] (0, -2) -- (0, -7); %y
		\draw [-] (1, 3) -- (1, -8); %z
		\node[anchor=south east] (x) at (-1, -2) {\footnotesize $T_0$};
		\node[anchor=south west] (y) at (0, -2) {\footnotesize $T_1$};
		\node[anchor=south west] (z) at (1, 3) {\footnotesize $T_2$};
		\foreach \x in {-2,-4,-6}
    			\draw (-1,\x ) -- (-1,\x ) node (\x x) {$\bullet$};
		\foreach \y in {-2,-4,-6}
    			\draw (0,\y ) -- (0,\y ) node (\y y) {$\bullet$};
		 \foreach \z in {3, 1, -1, -3,-5,-7}
    			\draw (1,\z ) -- (1,\z ) node (\z z) {$\bullet$};
		\draw [thick, ->] (-2y) -- (-3z) ;
		\draw [thick, densely dotted, red, <-] (-2x) -- (-2y);
	\end{scope}	
\end{tikzpicture}
        \caption{$- X_i$}
    \end{subfigure}
    \caption{Depiction of $X_i$ (left) and its dual (right). Here, $i = 3$. Vertical lines represent multiplication by $U$. Red dotted horizontal arrows represent the action of $\omega = 1+\iota$ and  black diagonal arrows represent the differential $\partial$; both are $U$-equivariant.
}\label{fig:iotacomplexes}
\end{figure}

It is shown in \cite[Theorem 1.8]{Stoffregenconnectedsum} (see also \cite[Theorem 1.7]{DaiManolescu}) that the $X_i$ are realized by the Brieskorn spheres
\[
h(\Sigma(2i+1, 4i+1, 4i+3)) = X_i.\footnote{Again, this equality should be interpreted as holding up to overall grading shift.} 
\]
Following work by the first author in the Pin(2)-case \cite{Stoffregenconnectedsum}, it was established in \cite[Theorem 1.7]{DaiManolescu} that the $X_i$ are linearly independent in $\mfI$. Hence their span constitutes a $\Z^\infty$-subgroup of $\mfI$. This was further studied by the first and third authors in \cite{DaiStoffregen}, in which it was shown that
\[
h(\Theta_{\text{SF}}) = h(\Theta_{\text{AR}}) = \text{span}\{X_i\}_{i \in \N} \cong \Z^{\infty}.
\]
Here, $\Theta_{\text{SF}}$ and $\Theta_{\text{AR}}$ are the subgroups of $\Theta^3_{\Z}$ generated by Seifert fibered spaces and almost-rational (AR) plumbed manifolds, respectively. It is currently unknown whether there are any integer homology spheres $Y$ for which $h(Y) \notin \text{span}\{X_i\}_{i \in \N}$.\footnote{Such an example would show that $\Theta^3_{\Z}$ is not generated by Seifert fibered spaces.} For more computations involving involutive Floer homology, see e.g.\ \cite{HendricksHomLidman}, \cite{Dai}.

There are thus several disadvantages to working with $\mfI$. First, it turns out that the set of $\inv$-complexes up to local equivalence is very large; apparently much larger than the span of the $X_i$. In particular, the authors do not in general know how to give an explicit parameterization of the different elements of $\mfI$. In addition, while the group operation on $\text{span}\{X_i\}_{i \in \N}$ is evidently straightforward, there are certainly more exotic elements of $\mfI$, as illustrated by the following 2-torsion example:

\begin{example}\label{ex:selfdual}
Let $X$ be the complex displayed in Figure~\ref{fig:selfdual}. This is generated (over $\ff[U]$) by five generators $T_{-2}, T_{-1}, T_0, T_1, \text{ and } T_2$, with
\[
\partial T_{-2} = \partial T_0 = \partial T_1 = 0, \ \ \ \partial T_{-1} = U^2T_{-2}, \ \ \ \partial T_2 = U^2 T_1
\]
and
\[
\omega T_{-2} = \omega T_1 = \omega T_2 = 0, \ \ \ \omega T_{-1} = UT_0, \ \ \ \omega T_0 = T_1.
\]
Note that $\omega^2$ is not zero, but is zero up to homotopy. It is straightforward to check that $X$ is self-dual, with the duality map sending $T_i$ to $T_{-i}$. Example~\ref{ex:selfdual} fits into a larger family of self-dual complexes, which we leave to the reader to define.
\end{example}

\begin{figure}[h!]
\center
\includegraphics[scale=0.85]{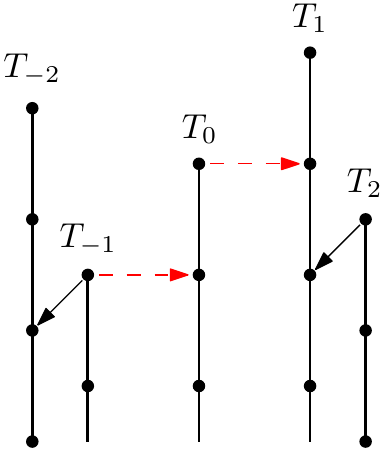}
\caption{A self-dual $\inv$-complex. Red dotted horizontal arrows represent the action of $\omega = 1+\iota$ and  black diagonal arrows represent the differential $\partial$; both are $U$-equivariant. The central generator $T_0$ has grading zero.}
\label{fig:selfdual}
\end{figure}

It is easy to check by hand that Example~\ref{ex:selfdual} is not locally equivalent to the identity element $\cC(0)$. This obstructs the existence of a total order on $\mfI$, since $X$ would then have to be either strictly greater than $\cC(0)$ or strictly less than $\cC(0)$; this is a contradiction, since $X$ is self-dual. Again, however, we stress that we do not know whether $X$ occurs as the $\inv$-complex of any actual 3-manifold.

Using $\mfI$ to initiate a general study of the homology cobordism group is thus difficult without either (a) obtaining a better understanding of the elements of $\mfI$, or (b) proving that the image of $h$ lies in some more manageable subgroup. In this paper, we circumvent this problem by introducing a slightly modified group $\smash{\mfIhat}$, whose elements we are able to parameterize explicitly. As discussed in the introduction, the key feature of $\smash{\mfIhat}$ which will allow us to establish this parameterization will be the presence of a total order.

%%%%%%%%%%%%%%%%%%%%%%%%%%%%%%%%%%%%%%%%%%%%%%%%%%%%%%%%%%%%%%%%%%%%%%%%%%%%%%%%%%%%%%%%%%%%%%%%%%%%%%%%%%%%%%%%%%%%%%%%%%%%%%%%%%%%%%%%%%%%%%%%%%%%%%%%%%%%%%%%%%%%%%%%%%%%%%%%%%%%%%%%%%%%%%%%%%%%%%%%%%%%%%%%%%%%%%%%%%%%%%%%%%%%%%%%%%%%%%%%%%%%%%%%%%%%%%%%%%%%%%%%%%%%%%%%%%%%%%%%%%%%%%%%%%%%%%%%%%%%%%%%%%%%%%%%%%%%%%%%%%%%%%%%%%%%%%%%%%%%%%%%%%%%%%%%%%%%%%%%%%%%%%%%%%%%%%%%%%%%

\section{Almost Local Equivalence}\label{sec:3}
In this section, we define a modification $\smash{\mfIhat}$ of $\mfI$, which we call the \textit{almost local equivalence group}. This will come with its own homomorphism from $\Theta^3_{\Z}$, factoring through $\smash{\mfI}$:
\[
\widehat{h} \co \Theta^3_{\Z} \rightarrow \mfI \rightarrow \mfIhat.
\] 
The group $\mfIhat$ is constructed in analogy to $\mfI$ by considering the set of \textit{almost $\inv$-complexes} up to the appropriate equivalence relation. Our first order of business will be to define these terms.

\subsection{Almost $\iota$-complexes}
\begin{definition}
Let $C_1$ and $C_2$ be free, finitely-generated, $\Z$-graded chain complexes over $\F[U]$, with $\deg(U)=-2$. Two grading-preserving $\ff[U]$-module homomorphisms
\[ f, g \co C_1 \rightarrow C_2\]
are \emph{homotopic mod $U$}, denoted $f \simeq g \mod U$, if there exists an $\ff[U]$-module homomorphism $H \co C_1 \rightarrow C_2$ such that $H$ increases grading by one and
\[ f + g + H \circ \d + \d \circ H \in \im U.\]
\end{definition}

\begin{definition}\label{def:aic}
An \emph{almost $\iota$-complex} $\cC = (C, \iotabar)$ consists of the following data:
\begin{itemize}
	\item A free, finitely-generated, $\Z$-graded chain complex $C$ over $\F[U]$, with
	\[
	U^{-1}H_*(C) \cong \F[U, U^{-1}].
	\] 
	Here, $U$ has degree $-2$ and $U^{-1}H_*(C)$ is supported in even gradings.
	\item A grading-preserving, $\ff[U]$-module homomorphism $\iotabar \co C \rightarrow C$ such that
	\[ \iotabar \circ \d + \d \circ \iotabar \in \im U \qquad \text{ and } \qquad \iotabar^2 \simeq \id \mod U. \]
\end{itemize}
Abusing notation slightly, we write
\[
\omega = 1 + \iotabar.
\]
Note that $\iotabar$ is \textit{not} a chain map; instead, it only commutes with $\partial$ modulo $U$.\footnote{Hence $\iotabar$ does not induce an action on the homology $H_*(C)$.} We likewise relax the condition that $\iotabar$ be a homotopy involution (from Definition~\ref{def:iotacomplex}) by only requiring that this hold mod $U$.
\end{definition}

\begin{remark}\label{rem:gradingconventiontwo}
As in Remark~\ref{rem:gradingconvention}, we also impose the convention that the maximally graded $U$-nontorsion element in $H_*(C)$ has grading zero. The enterprising reader will find it easy to re-phrase the results of the paper without this restriction, which is introduced largely for convenience.
\end{remark}

We similarly need a notion of maps between almost $\inv$-complexes:

\begin{definition}\label{def:aim}
An \emph{almost $\iota$-morphism} from $\cC_1 = (C_1, \iotabar_1)$ to $\cC_2 = (C_2, \iotabar_2)$ is a grading-preserving, $U$-equivariant chain map
\[ f \co C_1 \rightarrow C_2 \]
such that
\[ f \circ \iotabar \simeq \iotabar \circ f \mod U. \]
\end{definition}

Roughly speaking, the passage from $\inv$-complexes to almost $\inv$-complexes may be thought of as replacing all functional equalities involving $\inv$ with congruences modulo~$U$. 

We have the obvious definition of local equivalence between two almost $\inv$-complexes:

\begin{definition}\label{def:localmap}
Two almost $\inv$-complexes $(C, \inv)$ and $(C', \inv')$ are said to be {\em locally equivalent} if there exist almost $\inv$-morphisms
\[
f : C \to C', \ \ g : C' \to C
\]
between them which induce isomorphisms on homology after inverting the action of $U$. We refer to $f$ as a \textit{local map} from $(C, \inv)$ to $(C', \inv')$, and similarly refer to $g$ as a local map in the other direction. 
\end{definition}

We now show that the set of almost $\inv$-complexes up to local equivalence forms a group. The reader familiar with the construction of $\mfI$ may wish to skip to the next subsection; the current series of algebraic verifications mirrors those in \cite[Section 8]{HMZ:involutive-sum} essentially by replacing equalities everywhere with congruences modulo~$U$.

We begin by recording the proposed identity, inverse, and group operation:

\begin{definition}\label{def:identity}
We define $\cC(0)$ to be the almost $\inv$-complex consisting of a single $\ff[U]$-tower starting in grading zero, together with the identity map on this complex. 
\end{definition}

\begin{definition}\label{def:aicdual}
The \textit{dual} of an almost $\iota$-complex $\cC = (C,\iotabar)$ is defined to be $\cC^\vee = (C^\vee,\iotabar^\vee)$, where $C^\vee$ is the dual in the category of $\F[U]$-chain complexes, and $\iotabar^\vee$ is given by 
\[
(\iotabar^\vee y)(x)=y (\iotabar x)
\]
for $y\in C^\vee$ and $x\in C$. 
\end{definition}

\begin{definition}\label{defn:product}
The \textit{tensor product} of two almost $\iota$-complexes $\cC_1 = (C_1, \iotabar_1)$ and $\cC_2 = (C_2, \iotabar_2)$ is defined to be
\[ \cC_1 \otimes \cC_2 = (C_1 \otimes_{\F[U]} C_2, \iotabar_1 \otimes \iotabar_2). \]
\end{definition}

We now check that these are well-defined operations on the set of almost $\inv$-complexes modulo local equivalence:
\begin{lemma}
The dual of an almost $\inv$-complex is an almost $\inv$-complex. Moreover, the dual of an almost $\inv$-morphism is an almost $\inv$-morphism, which is local if the original map is local.
\end{lemma}
\begin{proof}
Straightforward.
\end{proof}

It follows from this that:

\begin{lemma}
The dual of (the local equivalence class of) an almost $\inv$-complex is well-defined up to local equivalence.
\end{lemma}
\begin{proof}
If $f$ and $g$ are local maps going between $\cC$ and $\cC'$, then $f^\vee$ and $g^\vee$ are local maps between their duals. 
\end{proof}

For tensor products, we have the slightly more complicated lemma:

\begin{lemma}\label{lem:tensormapnatural}
Let $f$ and $f'$ be grading-preserving $\ff[U]$-module homomorphisms from $\cC_1$ to $\cC_1'$ such that
\[
\partial f + f \partial \equiv \partial f' + f' \partial \equiv 0 \bmod U,
\]
and likewise let $g$ and $g'$ be grading-preserving $\ff[U]$-module homomorphisms from $\cC_2$ to $\cC_2'$ such that 
\[
\partial g + g \partial \equiv \partial g' + g' \partial \equiv 0 \bmod U.
\]
If $f \simeq f' \bmod U$ and $g \simeq g' \bmod U$, then $f \otimes g \simeq f' \otimes g' \bmod U$.
\end{lemma}
\begin{proof}
By considering $f \otimes g \simeq f' \otimes g \bmod U$ and $f' \otimes g \simeq f' \otimes g' \bmod U$ separately, it suffices to establish the case when one of the maps is held constant. Without loss of generality, assume $g = g'$. Let
\[
f + f' \equiv \partial s + s\partial \mod U.
\]
We define our homotopy $H \co \cC_1 \otimes \cC_2 \rightarrow \cC_1' \otimes \cC_2'$ by
\[
H = s \otimes g.
\]
Then 
\begin{align*}
\partial H(a \otimes b) + H\partial(a \otimes b) &= \partial(sa \otimes gb) + H(\partial a \otimes b + a \otimes \partial b) \\
&= \partial sa \otimes gb + sa \otimes \partial gb + s\partial a \otimes gb + sa \otimes g\partial b \\
&\equiv (\partial s + s \partial) a \otimes gb \mod U \\
&\equiv (f \otimes g)(a \otimes b) + (f' \otimes g)(a \otimes b) \mod U,
\end{align*}
where in the second-to-last line we have used the fact that $\partial g + g \partial \equiv 0 \bmod U$.
\end{proof}

\begin{lemma}\label{lem:tensornatural}
The tensor product of two almost $\iota$-complexes is an almost $\iota$-complex. Moreover, the tensor product of two almost $\inv$-morphisms is an almost $\inv$-morphism, which is local if the two factors are local.
\end{lemma}

\begin{proof}
First we check that $(\iotabar_1 \otimes \iotabar_2) \d + \d (\iotabar_1 \otimes \iotabar_2)  \in \im U $. Let $a \in \cC_1$ and $b \in \cC_2$. 
\begin{eqnarray*}
&&(\iotabar_1 \otimes \iotabar_2) \d (a \otimes b)+ \d(\iotabar_1 \otimes \iotabar_2)  (a \otimes b) 
\\
&=& (\iotabar_1 \otimes \iotabar_2) ( \d a \otimes b + a \otimes \d b) +  \d ( \iotabar_1 a \otimes \iotabar_2 b)
\\
&=&  \iotabar_1 \d a \otimes  \iotabar_2b +  \iotabar_1 a \otimes  \iotabar_2\d b +  \d  \iotabar_1 a \otimes \iotabar_2 b + \iotabar_1 a \otimes \d \iotabar_2 b
\\
&=& (\iotabar_1 \d + \d \iotabar_1)a \otimes  \iotabar_2 b
+\iotabar_1 a\otimes (\iotabar_2 \d + \d \iotabar_2)b.
\end{eqnarray*}
Since $\cC_1$ and $\cC_2$ are almost $\iota$-complexes, we are done. To see that $(\iotabar_1 \otimes \iotabar_2)^2 \simeq \id \bmod U$, note that the left-hand side is just $\iotabar_1^2 \otimes \iotabar_2^2$. Since 
\begin{align*}
&\iotabar_1^2 \simeq \id_1 \mod U \text{ and} \\
&\iotabar_2^2 \simeq \id_2 \mod U,
\end{align*}
by Lemma~\ref{lem:tensormapnatural} we have that $(\iotabar_1 \otimes \iotabar_2)^2 \simeq \id_1 \otimes \id_2 = \id \mod U$, as desired. By the K\"unneth formula, $U^{-1}H_*(\cC_1 \otimes \cC_2) \cong \F[U, U^{-1}]$. This shows that $\cC_1 \otimes \cC_2$ is an almost $\inv$-complex. Now let $f \co \cC_1 \rightarrow \cC_1'$ and $g \co \cC_2 \rightarrow \cC_2'$ be two almost $\inv$-morphisms. Then 
\begin{align*}
&f \iotabar_1 \simeq \iotabar_1' f \mod U \text{ and} \\ 
&g \iotabar_2 \simeq \iotabar_2' g \hspace{1.5pt}\mod U. 
\end{align*}
We wish to show that $(f \otimes g)(\iotabar_1 \otimes \iotabar_2) \simeq (\iotabar_1' \otimes \iotabar_2')(f \otimes g) \bmod U$. This follows immediately from Lemma~\ref{lem:tensormapnatural}, observing that the maps $f \iotabar_1$, $\iotabar_1' f$, $g \iotabar_2$, and $\iotabar_2' g$ all commute with $\d$ modulo $U$. The claim about locality is a consequence of the K\"unneth formula.
\end{proof}

It follows from this that:

\begin{lemma}\label{lem:aictensor}
The tensor product of two almost $\inv$-complexes is well-defined up to local equivalence.
\end{lemma}
\begin{proof}
If $f \co \cC_1 \rightarrow \cC_1'$ and $g \co \cC_2 \rightarrow \cC_2'$ are local maps, then $f \otimes g$ is a local map from $\cC_1 \otimes \cC_2$ to $\cC_1' \otimes \cC_2'$.
\end{proof}

It is evident that multiplication by $\cC(0)$ returns the same almost $\inv$-complex. We are thus left with verifying:

\begin{lemma}\label{lem:aicdual}
For any almost $\inv$-complex $\cC$, the tensor product $\cC \otimes \cC^\vee$ is locally equivalent to $\cC(0)$.
\end{lemma}
\begin{proof}
There is a canonical map $\cC \otimes \cC^\vee\to \F[U]$ given by the contraction map $c(x \otimes y) = y(x)$. Inspection shows that in sufficiently negative degrees, the induced map $H_*(\cC \otimes \cC^\vee) \to \F[U]$ is an isomorphism. We check that $c$ is an almost $\iota$-map. For this, we need to produce a chain homotopy $H \co \cC\otimes \cC^\vee \to \F[U]$ such that
\[
\iotabar c(x \otimes y) + c(\iotabar x \otimes \iotabar^\vee y) \equiv (\d H + H\d)(x \otimes y) \mod U.
\]
Now, the action of $\iotabar$ on $\cC(0)$ is the identity, and $\d$ on $\cC(0)$ is zero. Hence the above congruence reduces to
\[
y(x)+y(\iotabar^2 x) \equiv H\partial(x\otimes y) \mod U.
\]
Let $\iotabar^2+ \id \equiv \d s+s \d$ $\bmod$ $U$ on $\cC$. Substituting this in, we obtain
\[
y(\d s x)+y(s \d x) \equiv H \d (x\otimes y) \mod U .
\]
By construction of the contraction map, $y(\d sx)=(\d y)(s x)$. It is then straightforward to check that setting $H(a \otimes b) = b(sa)$ provides the desired homotopy. Hence $c$ constitutes a local map from $\cC \otimes \cC^\vee$ to $\F[U]$. Applying duality yields a local map
\[
\F[U]^\vee \to \cC\otimes \cC^\vee.
\]
Hence $\cC^\vee$ is an inverse of $\cC$. (One may also perform this calculation in bases, of course. See \cite[Proposition 8.8]{HMZ:involutive-sum}.)
\end{proof}

Noting that the operation of tensor product is clearly associative (as well as commutative), we now have everything we need to define the group $\smash{\mfIhat}$:

\begin{definition}
Let $\mfIhat$ be the set of almost $\inv$-complexes up to local equivalence. This has a group operation given by tensor product, where the identity element is given by $\cC(0)$ and inverses are given by dualizing. We call $\smash{\mfIhat}$ the \textit{almost local equivalence group}. 
\end{definition}

Clearly, every $\inv$-complex is automatically an almost $\inv$-complex, and a local map between two $\inv$-complexes is automatically an almost $\inv$-map. Hence we have a well-defined forgetful map from $\mfI$ to $\smash{\mfIhat}$:

\begin{lemma}\label{lem:map-of-local-groups}
The forgetful map $\mfI \to \mfIhat$ is a group homomorphism.
\end{lemma}
\begin{proof}
The almost $\iota$-complex associated to a tensor product of $\iota$-complexes is the tensor product of the underlying almost $\iota$-complexes.
\end{proof}

Composing the forgetful homomorphism with the map $h$ from Section~\ref{sec:2}, we finally obtain our desired homomorphism:
\[
\widehat{h}: \Theta^3_{\Z} \rightarrow \mfI \rightarrow \mfIhat.
\]

%%%%%%%%%%%%%%%%%%%%%%%%%%%%%%%%%%%%%%%%%%%%%%%%%%%%%%%%%%%%%%%%%%%%%%

\subsection{Preliminary examples} We now give some simple examples to help illustrate the difference between $\smash{\mfIhat}$ and $\mfI$.

\begin{example}
As remarked previously, every $\inv$-complex is automatically an almost $\inv$-complex. Hence, the complexes $X_i$ and $-X_i$ introduced in Example~\ref{ex:xicomplex} are certainly almost $\inv$-complexes. In Figure~\ref{fig:almostiotacomplexes}, we have depicted two similar almost $\inv$-complexes which are \textit{not} $\inv$-complexes. The first of these, which we denote by $\cC(-, -3)$ \footnote{The reason for this notation will become clear in Section~\ref{sec:4}.}, has three generators $T_0, T_1$, and $T_2$, with
\[
\d T_0 = \d T_2 = 0, \ \ \ \d T_1 = U^3 T_2
\]
and
\[
\omega T_0 = T_1, \ \ \ \omega T_1 = \omega T_2 = 0.
\]
Note that $\omega$ does \textit{not} commute with $\d$, since $\omega(\d T_0) = 0$ but $\d \omega(T_0) = U^3 T_2$. (However, $\omega$ trivially commutes with $\d$ modulo $U$.) The second complex, which we denote by $\cC(+, 3)$, also has three generators. These satisfy
\[
\d T_0 = \d T_1 = 0, \ \ \ \d T_2 = U^3 T_2
\]
and
\[
\omega T_0 = \omega T_2 = 0, \ \ \ \omega T_1 = T_0.
\]
The reader should compare with Figure~\ref{fig:iotacomplexes}.
\end{example}

\begin{figure}[h!]
    \centering
    \begin{subfigure}[b]{0.23\textwidth}
    \centering
        \begin{tikzpicture}[scale=.75, yscale=.5]
	\begin{scope}[thin, black]
		\draw [-] (-1, -2) -- (-1, -7); %x
		\draw [-] (0, -2) -- (0, -7); %y
		\draw [-] (1, 3) -- (1, -8); %z
		\node[anchor=south east] (x) at (-1, -2) {\footnotesize $T_0$};
		\node[anchor=south west] (y) at (0, -2) {\footnotesize $T_1$};
		\node[anchor=south west] (z) at (1, 3) {\footnotesize $T_2$};
		\foreach \x in {-2,-4,-6}
    			\draw (-1,\x ) -- (-1,\x ) node (\x x) {$\bullet$};
		\foreach \y in {-2,-4,-6}
    			\draw (0,\y ) -- (0,\y ) node (\y y) {$\bullet$};
		 \foreach \z in {3, 1, -1, -3,-5,-7}
    			\draw (1,\z ) -- (1,\z ) node (\z z) {$\bullet$};
		\draw [thick, ->] (-2y) -- (-3z) ;
		\draw [thick, densely dotted, red, ->] (-2x) -- (-2y);
	\end{scope}	
\end{tikzpicture}
        \caption{$\cC(-, -3)$}
    \end{subfigure}
    \begin{subfigure}[b]{0.23\textwidth}
    \centering
        \begin{tikzpicture}[scale=.75, yscale=.5]
	\begin{scope}[thin, black]
		\draw [-] (-1, 2) -- (-1, -7); %x
		\draw [-] (0, 2) -- (0, -7); %z
		\draw [-] (1, -3) -- (1, -8); %y
		\node[anchor=south east] (x) at (-1, 2) {\footnotesize $T_0$};
		\node[anchor=south west] (z) at (0, 2) {\footnotesize $T_1$};
		\node[anchor=south west] (y) at (1, -3) {\footnotesize $T_2$};
		\foreach \x in {2, 0, -2,-4,-6}
    			\draw (-1,\x ) -- (-1,\x ) node (\x x) {$\bullet$};
		\foreach \z in {2, 0, -2,-4,-6}
    			\draw (0,\z ) -- (0,\z ) node (\z z) {$\bullet$};
		 \foreach \y in {-3,-5,-7}
    			\draw (1,\y ) -- (1,\y ) node (\y y) {$\bullet$};
		\draw [thick, ->] (-3z) -- (-4y);
		\draw [thick, densely dotted, red, ->] (2z) -- (2x);
	\end{scope}	
	\end{tikzpicture}
        \caption{$\cC(+, 3)$}
       
    \end{subfigure}
    \caption{Some simple almost $\inv$-complexes. Vertical lines represent multiplication by $U$. Red dotted horizontal arrows represent the action of $\omega = 1+\iota$ and  black diagonal arrows represent the differential $\partial$; both are $U$-equivariant.
}\label{fig:almostiotacomplexes}
\end{figure}
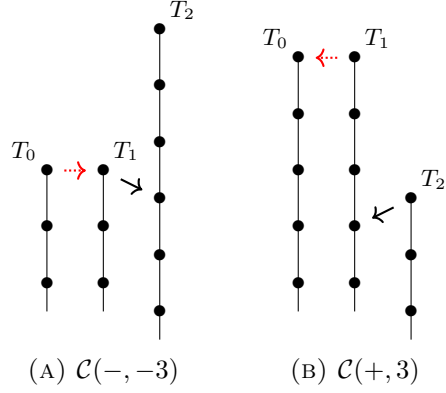

It turns out that even up to local equivalence, there exist almost $\inv$-complexes which do not come from $\inv$-complexes, showing that the forgetful homomorphism $\mfI \rightarrow \mfIhat$ is not surjective. The problem of determining which almost $\inv$-complexes are realized by genuine $\inv$-complexes appears to be quite difficult; see Section~\ref{sec:8} for comments.

Conversely, the relation of local equivalence in $\mfIhat$ is coarser than the relation of local equivalence in $\mfI$. Hence the forgetful homomorphism is not injective. Psychologically, the most important example of this is the case of Example~\ref{ex:selfdual}:

\begin{example}\label{ex:selfdualzero}
Consider the self-dual $\inv$-complex $X$ from Example~\ref{ex:selfdual}. We claim that considered as an almost $\inv$-complex, this is locally equivalent to $\cC(0)$. Indeed, we have a local map $f \co \cC(0) \rightarrow X$ which takes the generator of $\cC(0)$ to $T_0$. We define a local map $g \co X \rightarrow \cC(0)$ by sending $T_0$ to the generator of $\cC(0)$, and all of the other generators to zero. Neither of these maps commute with $\omega$, even up to homotopy. However, they trivially commute with $\omega$ modulo $U$, so $f$ and $g$ are local maps in the setting of almost $\inv$-complexes.
\end{example}

We will further discuss the relation between $\mfI$ and $\mfIhat$ in Section~\ref{sec:8}.

%%%%%%%%%%%%%%%%%%%%%%%%%%%%%%%%%%%%%%%%%%%%%%%%%%%%%%%%%%%%%%%%%%%%%%

\subsection{Total order}
We now turn to the most important characteristic of the almost local equivalence group: the presence of a total order. 

\begin{definition}\label{def:partialorder}
We define partial orders on $\mfI$ and $\mfIhat$ as follows:
\begin{itemize}
\item Let $\cC_1$ and $\cC_2$ be two $\inv$-complexes. We say that $\cC_1 \leq \cC_2$ if there exists a local map (in the sense of $\inv$-complexes) $f \co \cC_1 \rightarrow \cC_2$.
\item Let $\cC_1$ and $\cC_2$ be two almost $\inv$-complexes. We say that $\cC_1 \leq \cC_2$ if there exists a local map (in the sense of almost $\inv$-complexes) $f \co \cC_1 \rightarrow \cC_2$.
\end{itemize}
\end{definition}

In light of Lemma~\ref{lem:tensornatural} (as well as the analogous statement for $\inv$-complexes), it is clear that the above partial order respects the group structure in both cases. Moreover, the forgetful homomorphism $\mfI \rightarrow \mfIhat$ is evidently a homomorphism of partially ordered groups. Although Example~\ref{ex:selfdual} prevents the partial order on $\mfI$ from being a total order, the fact that the forgetful homomorphism sends $X$ to $\cC(0)$ means that we can still try to establish a total order on $\smash{\mfIhat}$.

It will be convenient for us to first establish some technical results aimed at simplifying Definitions~\ref{def:aic} and \ref{def:aim}. We begin by showing that if we have two homotopy equivalent chain complexes over $\ff[U]$, then an almost $\inv$-action on one can be transferred to an almost $\inv$-action on the other:

\begin{lemma}\label{lem:iotatransfer}
Let $\cC_1 = (C_1, \iotabar_1)$ be an almost $\inv$-complex. Let $C_2$ be another chain complex over $\ff[U]$, and suppose we have homotopy inverses
\[
f\co C_1 \rightarrow C_2, \ \ \ g \co C_2 \rightarrow C_1
\]
in the usual sense of chain complexes over $\ff[U]$. Then we may make $C_2$ into an almost $\inv$-complex by defining
\[
\iotabar_2 = f \circ \iotabar_1 \circ g.
\]
Moreover, $\cC_2 = (C_2, \iotabar_2)$ is locally equivalent to $\cC_1$.
\end{lemma}
\begin{proof}
Showing that $\partial \circ \iotabar_2 + \iotabar_2 \circ \partial \in \im U$ is straightforward. Thus, let $gf + \id = \partial H + H \partial$. A quick computation yields
\begin{align*}
\iotabar_2^2 &= (f  \iotabar_1 g) (f \iotabar_1 g) \\
&= f \iotabar_1 (\id + \partial H + H\partial) \iotabar _1 g \\
&\equiv f \iotabar_1^2 g + \partial (f\iotabar_1 H \iotabar_1g) + (f\iotabar_1 H \iotabar_1g)\partial \mod U,
\end{align*}
where in the last line we have used the fact that $\partial$ commutes with $\iotabar_1$ mod $U$. Since $\iotabar_1^2 \simeq \id \bmod U$ and $fg \simeq \id$, it easily follows that $\iotabar_2^2 \simeq \id \bmod U$. Hence $\cC_2$ is an almost $\inv$-complex.

We now show that $f$ and $g$ are local maps between $\cC_1$ and $\cC_2$. Let $H$ be as above. Then
\begin{align*}
f\iotabar_1 + \iotabar_2 f &= f\iotabar_1 + f\iotabar_1 gf \\
&= f\iotabar_1 (\partial H + H\partial) \\
&\equiv \partial(f \iotabar_1H) + (f \iotabar_1 H)\partial \mod U,
\end{align*}
where in the last line we have used the fact that $\partial$ commutes with $\iotabar_1$ mod $U$. This shows that that $f\iotabar_1 \simeq \iotabar_2 f \bmod{U}$, so $f$ is an almost $\inv$-morphism. Since $f$ is a quasi-isomorphism, it is clearly local. The proof for $g$ is similar. 
\end{proof}

This yields the following convenient structure result:

\begin{lemma}\label{lem:reduced}
Let $\cC$ be any almost $\iota$-complex. Then $\cC$ is locally equivalent to an almost $\iota$-complex $\cC'$ with $\partial\equiv 0 \bmod{U}$.
\end{lemma}
\begin{proof}
It is a standard fact that if $C$ is a free, finitely-generated chain complex over $\ff[U]$, then $C$ is homotopy equivalent to a complex with $\partial\equiv 0 \bmod{U}$. To see this, note that since $\F[U]$ is a PID, we may fix a basis for $C$ of the form $\{ x, y_i,z_i \}_{i=1}^n$, where
\begin{align*}
&\partial x = 0, \\
&\partial y_i = \smash{U^{\eta_i}} z_i, \text{ and} \\
&\partial z_i = 0,
\end{align*}
for some set of integers $\eta_i \geq 0$. Without loss of generality, assume that $\eta_i \geq 1$ precisely when $1 \leq i \leq m$. Then $C$ is homotopy equivalent to the subcomplex $C'$ generated by $\{x, y_i, z_i\}_{i=1}^m$, which has $\partial\equiv 0 \bmod{U}$. Applying Lemma~\ref{lem:iotatransfer} completes the proof.
\end{proof}

\begin{definition}\label{def:reduced}
Lemma \ref{lem:reduced} shows there is no loss of generality in considering only almost $\iota$-complexes with $\partial \equiv 0 \bmod{U}$. We call these \emph{reduced} almost $\iota$-complexes. We will assume henceforth that all of our almost $\iota$-complexes are reduced.
\end{definition}

\begin{lemma}\label{lem:reducednice}
If $\cC$ is a reduced almost $\inv$-complex, then $\omega^2 \equiv 0 \bmod{U}$. Moreover, if $\cC'$ is another reduced almost $\inv$-complex and $f \co \cC \rightarrow \cC'$ is a chain map between them, then $f$ is an almost $\inv$-morphism if and only if $f \omega + \omega f \equiv 0 \bmod{U}$.
\end{lemma}
\begin{proof}
The lemma follows immediately from the fact that $\im \partial$ is contained in $\im U$ for any reduced almost $\inv$-complex.
\end{proof}

Thus, in practice we will not concern ourselves with the homotopies in Definitions~\ref{def:aic} and \ref{def:aim}, and instead prove that maps are almost $\inv$-morphisms by checking that they commute with $\omega$ modulo $U$.

We now turn to the central result of this section:

\begin{lemma}\label{lem:totalorder}
Let $\cC$ be an almost $\inv$-complex. Suppose that there is no local map from $\cC(0)$ to $\cC$. Then there exists a local map from $\cC$ to $\cC(0)$.
\end{lemma} 
\begin{proof}
Fix a basis $\mathcal{B} = \{x, y_i, z_i\}$ for $\cC$ as in the proof of Lemma~\ref{lem:reduced}. For each $g \in \mathcal{B}$, write $\omega g$ as a linear combination of elements in $\mathcal{B}$ and their $U$-powers. Since any two actions of $\omega$ which are congruent modulo $U$ are locally equivalent, without loss of generality we may assume that all of these $U$-powers are zero. We have displayed this situation schematically in Figure~\ref{fig:centrallemma}; here, the $\omega$-arrows send elements of $\mathcal{B}$ to (linear combinations of) elements of $\mathcal{B}$, without any $U$-powers.

Now consider the $\ff$-vector spaces
\begin{align*}
&Z = \text{span}_{\ff} \{ z_i \}, \text{ and} \\
&W = \text{span}_{\ff} \{\omega g \mid g \in \mathcal{B}\}.
\end{align*}
Note that $Z$ and $W$ are both subspaces of the $\ff$-span of $\mathcal{B}$; i.e., they consist of sums of generators that are not decorated by any powers of $U$. We may of course consider the $\ff[U]$-spans of $Z$ and $W$, which we denote by:
\begin{align*}
&Z' = \text{span}_{\ff[U]} \{ z_i \}, \text{ and} \\
&W' = \im \omega = \text{span}_{\ff[U]} \{\omega g \mid g \in \mathcal{B}\}.
\end{align*}
Note that $Z' + W'$ is a subcomplex of $\cC$ which is preserved by $\omega$. Indeed, $\im \partial$ is a subset of $Z'$, so $\partial$ maps into $Z'$. Similarly, $\omega$ maps into $W'$, by definition. 

We now claim that $x$ is not in $Z + W$. Indeed, suppose it were. Then $x = z + w$ for some $z \in Z$ and $w \in W$. Since $z$ is a $U$-torsion cycle, we then have that $x + z$ is a $U$-nontorsion cycle in the image of $\omega$. Because $\omega^2 \equiv 0 \bmod{U}$, this implies that $\omega(x+z) \equiv 0 \bmod{U}$. But this means we can construct a local map from $\cC(0)$ to $\cC$ by mapping the generator of $\cC(0)$ to $x+z$, a contradiction. Hence $x$ is not in $Z + W$.

Now choose a homogenous $\ff$-basis $\{p_1, \ldots, p_r\}$ for $Z + W$. Extend the linearly independent set $\{x, p_1, \ldots, p_r\}$ to a homogenous $\ff$-basis
\[
\{x, p_1, \ldots, p_r, q_1, \ldots, q_s\}
\]
for all of $\text{span}_{\ff}(\mathcal{B})$. Define
\[
\cC' = \text{span}_{\ff[U]}\{p_1, \ldots, p_r, q_1, \ldots, q_s\}.
\]
This is a subcomplex of $\cC$ which is preserved by $\omega$, following the same argument as for $Z' + W'$. It is then easily checked that the quotient map
\[
\cC \rightarrow \cC/\cC' \cong \ff[U]
\]
is a local map from $\cC$ to $\cC(0)$, completing the proof.
\end{proof}

\begin{figure}[h!]
\center
\includegraphics[scale=0.9]{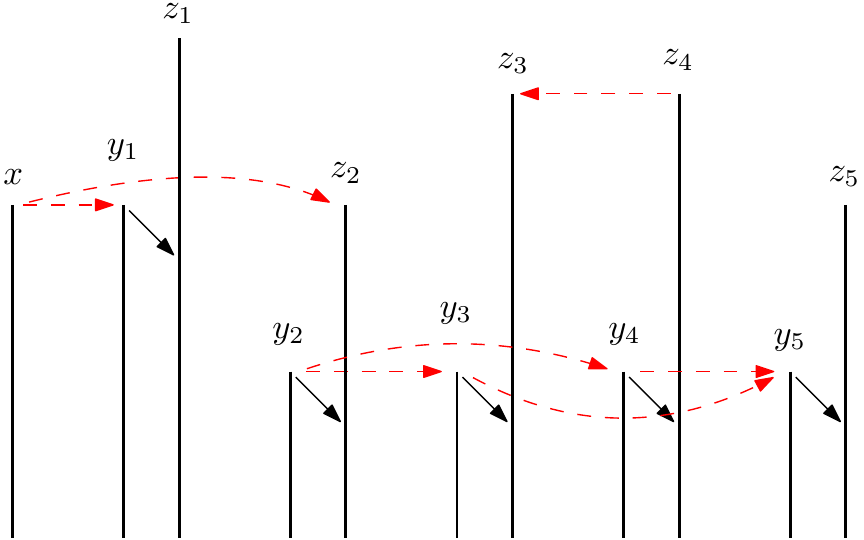}
\caption{A schematic example of an almost $\inv$-complex with a simplified $\omega$-action.}
\label{fig:centrallemma}
\end{figure}

It follows immediately that:

\begin{theorem}\label{thm:totalorder}
The almost local equivalence group $\mfIhat$ is totally ordered.
\end{theorem}
\begin{proof}
It suffices to show that every element is either greater than or equal to $\cC(0)$ or less than or equal to $\cC(0)$. This is the content of Lemma~\ref{lem:totalorder}.
\end{proof}

%%%%%%%%%%%%%%%%%%%%%%%%%%%%%%%%%%%%%%%%%%%%%%%%%%%%%%%%%%%%%%%%%%%%%%%%%%%%%%%%%%%%%%%%%%%%%%%%%%%%%%%%%%%%%%%%%%%%%%%%%%%%%%%%%%%%%%%%%%%%%%%%%%%%%%%%%%%%%%%%%%%%%%%%%%%%%%%%%%%%%%%%%%%%%%%%%%%%%%%%%%%%%%%%%%%%%%%%%%%%%%%%%%%%%%%%%%%%%%%%%%%%%%%%%%%%%%%%%%%%%%%%%%%%%%%%%%%%%%%%%%%%%%%%%%%%%%%%%%%%%%%%%%%%%%%%%%%%%%%%%%%%%%%%%%%%%%%%%%%%%%%%%%%%%%%%%%%%%%%%%%%%%%%%%%%%%%%%%%%%%%%%%%%%%%%%%%%%%%%%%%%%%%%%%%%%%%%%%%%%%%%%%%%%%%%%%%%%%%%%%%%%%%%%%%%%%%%%%%%%%%%%%%%%%%%%%%%%%%%%%%%%%%%%%%%%%%%%%%%%%%%%%%%%%%%%%%%%%%%%%%%%%%%%%%%%%%%%%%%%%%%%%%%%%%%%%%%%%%%%%%%%%%%%%%%%%%%%%%%%%%%%%%%%%%%%%%%%%%%%%%%%%%%%%%%%%%%%

\section{Standard Complexes and Their Properties}
\label{sec:4}
We now introduce an important family of almost $\inv$-complexes, which we call the \textit{standard complexes}. Our main goal for the present section will be to understand the total order on this family afforded by Theorem~\ref{thm:totalorder}. In Section~\ref{sec:6}, we will use this to show that every almost $\inv$-complex is locally equivalent to a standard complex, providing an explicit parameterization of $\smash{\mfIhat}$. 

\subsection{Standard complexes}  
We begin with the definition of a standard complex.

\begin{definition}\label{def:stdcomplexes}
Let $n$ be a nonnegative integer, and fix a sequence
\[
(a_1, b_2, a_3, b_4, \ldots, a_{2n-1}, b_{2n})
\]
with each $a_i \in \{-1, +1\}$ and $b_i \in \mathbb{Z} \setminus \{0\}$. We denote such a sequence using the notation $(a_i, b_i)$.\footnote{Note that the $a_i$ are indexed by odd integers, while the $b_i$ are indexed by even integers.} We define the \textit{standard complex of type $(a_i, b_i)$}, denoted by $\cC(a_i, b_i)$, as follows. As a free module over $\F[U]$, $\cC(a_i, b_i)$ is generated by $2n + 1$ elements $T_0, T_1, \ldots, T_{2n}$. For each symbol $a_i$, we introduce a single $\omega$-relation between $T_{i-1}$ and $T_i$, according to the rule
\begin{itemize}
\item If $a_i = -1$, then $\omega T_{i-1} = T_i$.
\item If $a_i = +1$, then $\omega T_i = T_{i-1}$.
\end{itemize}
Similarly, for each symbol $b_i$, we introduce a single $\partial$-relation between $T_{i-1}$ and $T_i$, according to the rule
\begin{itemize}
\item If $b_i < 0$, then $\partial T_{i-1} = U^{|b_i|} T_i$.
\item If $b_i > 0$, then $\partial T_i = U^{|b_i|} T_{i-1}$. 
\end{itemize}
We will sometimes use the symbol set $a_i \in \{-, +\}$, in order to differentiate the $a_i$ from the $b_i$. We call $2n$ the \textit{length} of the sequence $(a_i, b_i)$ and/or the complex $\cC(a_i, b_i)$. Note that if $n = 0$, then our complex has only one generator and no $\omega$- or $\partial$-relations. See Figure~\ref{fig:stdcomplexes1} for an illustration of the four general types of standard complexes in the case $n = 1$. 

\begin{figure}[h!]
    \centering
    \begin{subfigure}[b]{0.23\textwidth}
    \centering
        \begin{tikzpicture}[scale=.75, yscale=.5]
	\begin{scope}[thin, black]
		\draw [-] (-1, -2) -- (-1, -7); %x
		\draw [-] (0, -2) -- (0, -7); %y
		\draw [-] (1, 3) -- (1, -8); %z
		\node[anchor=south east] (x) at (-1, -2) {\footnotesize $T_0$};
		\node[anchor=south west] (y) at (0, -2) {\footnotesize $T_1$};
		\node[anchor=south west] (z) at (1, 3) {\footnotesize $T_2$};
		\foreach \x in {-2,-4,-6}
    			\draw (-1,\x ) -- (-1,\x ) node (\x x) {$\bullet$};
		\foreach \y in {-2,-4,-6}
    			\draw (0,\y ) -- (0,\y ) node (\y y) {$\bullet$};
		 \foreach \z in {3, 1, -1, -3,-5,-7}
    			\draw (1,\z ) -- (1,\z ) node (\z z) {$\bullet$};
		\draw [thick, ->] (-2y) -- (-3z) ;
		\draw [thick, densely dotted, red, ->] (-2x) -- (-2y);
	\end{scope}	
\end{tikzpicture}
        \caption{$\cC(-, -3)$}
    \end{subfigure}
     %spacing
    \begin{subfigure}[b]{0.23\textwidth}
    \centering
        \begin{tikzpicture}[scale=.75, yscale=.5]
	\begin{scope}[thin, black]
		\draw [-] (-1, 2) -- (-1, -7); %x
		\draw [-] (0, 2) -- (0, -7); %z
		\draw [-] (1, -3) -- (1, -8); %y
		\node[anchor=south east] (x) at (-1, 2) {\footnotesize $T_0$};
		\node[anchor=south west] (z) at (0, 2) {\footnotesize $T_1$};
		\node[anchor=south west] (y) at (1, -3) {\footnotesize $T_2$};
		\foreach \x in {2, 0, -2,-4,-6}
    			\draw (-1,\x ) -- (-1,\x ) node (\x x) {$\bullet$};
		\foreach \z in {2, 0, -2,-4,-6}
    			\draw (0,\z ) -- (0,\z ) node (\z z) {$\bullet$};
		 \foreach \y in {-3,-5,-7}
    			\draw (1,\y ) -- (1,\y ) node (\y y) {$\bullet$};
		\draw [thick, ->] (-3z) -- (-4y);
		\draw [thick, densely dotted, red, ->] (2x) -- (2z);
	\end{scope}	
\end{tikzpicture}
        \caption{$\cC(-,3)$}
    \end{subfigure}
     %spacing
    \begin{subfigure}[b]{0.23\textwidth}
    \centering
        \begin{tikzpicture}[scale=.75, yscale=.5]
	\begin{scope}[thin, black]
		\draw [-] (-1, -2) -- (-1, -7); %x
		\draw [-] (0, -2) -- (0, -7); %y
		\draw [-] (1, 3) -- (1, -8); %z
		\node[anchor=south east] (x) at (-1, -2) {\footnotesize $T_0$};
		\node[anchor=south west] (y) at (0, -2) {\footnotesize $T_1$};
		\node[anchor=south west] (z) at (1, 3) {\footnotesize $T_2$};
		\foreach \x in {-2,-4,-6}
    			\draw (-1,\x ) -- (-1,\x ) node (\x x) {$\bullet$};
		\foreach \y in {-2,-4,-6}
    			\draw (0,\y ) -- (0,\y ) node (\y y) {$\bullet$};
		 \foreach \z in {3, 1, -1, -3,-5,-7}
    			\draw (1,\z ) -- (1,\z ) node (\z z) {$\bullet$};
		\draw [thick, ->] (-2y) -- (-3z) ;
		\draw [thick, densely dotted, red, <-] (-2x) -- (-2y);
	\end{scope}	
\end{tikzpicture}
        \caption{$\cC(+, -3)$}
    \end{subfigure}
         %spacing
 \begin{subfigure}[b]{0.23\textwidth}
    \centering
        \begin{tikzpicture}[scale=.75, yscale=.5]
	\begin{scope}[thin, black]
		\draw [-] (-1, 2) -- (-1, -7); %x
		\draw [-] (0, 2) -- (0, -7); %z
		\draw [-] (1, -3) -- (1, -8); %y
		\node[anchor=south east] (x) at (-1, 2) {\footnotesize $T_0$};
		\node[anchor=south west] (z) at (0, 2) {\footnotesize $T_1$};
		\node[anchor=south west] (y) at (1, -3) {\footnotesize $T_2$};
		\foreach \x in {2, 0, -2,-4,-6}
    			\draw (-1,\x ) -- (-1,\x ) node (\x x) {$\bullet$};
		\foreach \z in {2, 0, -2,-4,-6}
    			\draw (0,\z ) -- (0,\z ) node (\z z) {$\bullet$};
		 \foreach \y in {-3,-5,-7}
    			\draw (1,\y ) -- (1,\y ) node (\y y) {$\bullet$};
		\draw [thick, ->] (-3z) -- (-4y);
		\draw [thick, densely dotted, red, ->] (2z) -- (2x);
	\end{scope}	
	\end{tikzpicture}
        \caption{$\cC(+, 3)$}
       
    \end{subfigure}
    \caption{Standard complexes with $n = 1$. Vertical lines represent multiplication by $U$. Red dotted horizontal arrows represent the action of $\omega = 1+\iota$ and black diagonal arrows represent the differential $\partial$; both are $U$-equivariant. The reader should compare with Figures~\ref{fig:iotacomplexes} and \ref{fig:almostiotacomplexes}; note that under the forgetful homomorphism, $X_i$ maps to $\cC(-, i)$ and $-X_i$ maps to $\cC(+, -i)$.
}\label{fig:stdcomplexes1}
\end{figure}
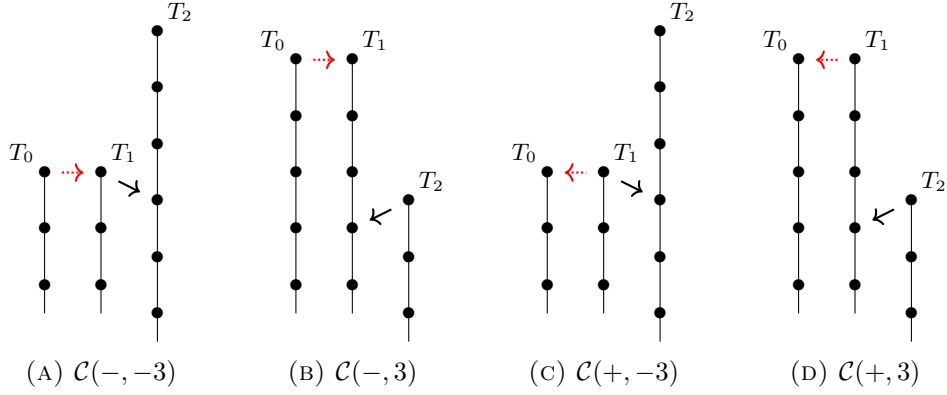

It is helpful to visualize the $T_i$ as a sequence of towers arranged from left-to-right across the page, with each tower connected to the immediately previous one by an arrow representing either an $\omega$- or a $\partial$-relation. A useful heuristic to remember is that if the $i$th symbol is negative, then the arrow between $T_{i-1}$ and $T_i$ goes from left-to-right, with the reverse being true if the $i$th symbol is positive. Note also that if $b_i < 0$, then $T_i$ appears in a higher grading than $T_{i-1}$, and vice-versa if $b_i > 0$. See Figure~\ref{fig:stdcomplexes2} for an illustrative example.

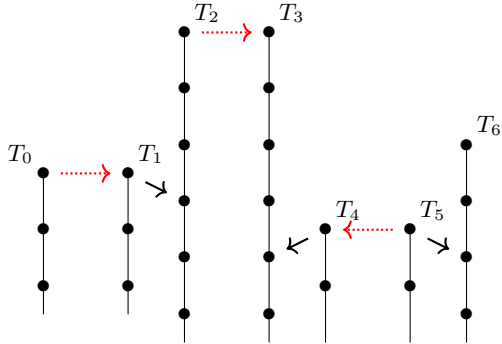
\begin{figure}[h!]

    \centering
        \begin{tikzpicture}[scale=.75, yscale=.5]
	\begin{scope}[thin, black]
		\draw [-] (-1, -2) -- (-1, -7); %x
		\draw [-] (0.5, -2) -- (0.5, -7); %y
		\draw [-] (1.5, 3) -- (1.5, -8); %z
		\draw [-] (3, 3) -- (3, -8); %a
		\draw [-] (4, -4) -- (4, -8); %b
		\draw [-] (5.5, -4) -- (5.5, -8); %c
		\draw [-] (6.5, -1) -- (6.5, -8); %d
		\node[anchor=south east] (x) at (-1, -2) {\footnotesize $T_0$};
		\node[anchor=south west] (y) at (0.5, -2) {\footnotesize $T_1$};
		\node[anchor=south west] (z) at (1.5, 3) {\footnotesize $T_2$};
		\node[anchor=south west] (a) at (3, 3) {\footnotesize $T_3$};
		\node[anchor=south west] (b) at (4, -4) {\footnotesize $T_4$};
		\node[anchor=south west] (c) at (5.5, -4) {\footnotesize $T_5$};
		\node[anchor=south west] (d) at (6.5, -1) {\footnotesize $T_6$};
		\foreach \x in {-2,-4,-6}
    			\draw (-1,\x ) -- (-1,\x ) node (\x x) {$\bullet$};
		\foreach \y in {-2,-4,-6}
    			\draw (0.5,\y ) -- (0.5,\y ) node (\y y) {$\bullet$};
		 \foreach \z in {3, 1, -1, -3,-5,-7}
    			\draw (1.5,\z ) -- (1.5,\z ) node (\z z) {$\bullet$};
		\foreach \a in {3, 1, -1, -3,-5,-7}
    			\draw (3,\a ) -- (3,\a ) node (\a a) {$\bullet$};
		\foreach \b in {-4, -6}
    			\draw (4,\b ) -- (4,\b ) node (\b b) {$\bullet$};
		\foreach \c in {-4, -6}
    			\draw (5.5,\c ) -- (5.5,\c ) node (\c c) {$\bullet$};
		\foreach \d in {-1, -3, -5, -7}
    			\draw (6.5,\d ) -- (6.5,\d ) node (\d d) {$\bullet$};	
		\draw [thick, ->] (-2y) -- (-3z) ;
		\draw [thick, ->] (-4b) -- (-5a) ;
		\draw [thick, ->] (-4c) -- (-5d) ;
		\draw [thick, densely dotted, red, ->] (-2x) -- (-2y);
		\draw [thick, densely dotted, red, ->] (3z) -- (3a);
		\draw [thick, densely dotted, red, ->] (-4c) -- (-4b);
	\end{scope}	
\end{tikzpicture}
        \caption{The standard complex $\cC(-, -3, -, 4, +, -2)$.}
        \label{fig:stdcomplexes2}
\end{figure}

It is clear that fixing the grading of $T_0$ determines the gradings of all of the other $T_i$, via the condition that $\omega$ and $\partial$ be of degrees 0 and $-1$, respectively. We thus normalize our complex by declaring $T_0$ to have grading zero. Note that the homology of $\cC(a_i, b_i)$ has a single nontorsion tower, which is generated by the class of $T_0$. There are also $n$ torsion towers of height $|b_i|$, each generated by the class of either $T_{i-1}$ or $T_i$, depending on the sign of $b_i$. It is clear from the definition that $\cC(a_i, b_i)$ is an almost $\inv$-complex.

We define the \textit{$i$th symbol of $\cC(a_i, b_i)$} to be the symbol with index $i$ in the associated sequence $(a_i, b_i)$. When we need to refer to the $i$th symbol without specifying whether it is an $a_i$ or a $b_i$, we will use the notation $t_i$. 
\end{definition}

%%%%%%%%%%%%%%%%%%%%%%%%%%%%%%%%%%%%%%%%%%%%%%%%%%

\subsection{The lexicographic order}
We now introduce a lexicographic order on the set of standard complexes. Let $\Zbang = (\Z, \leqbang)$ denote the integers with the order:
\[ -1 \lebang -2 \lebang -3 \lebang \cdots \lebang  0 \lebang \cdots \lebang 3 \lebang 2 \lebang 1. \]
Note that for  $t \in \Zbang$, we have $t \lebang 0$ if and only if $t < 0$ in the usual sense. 

We obtain a lexicographic order (which we denote by $\leq$) on the set of sequences $(a_i, b_i)$ as in Definition~\ref{def:stdcomplexes} by treating the entries as elements of $\Zbang$ and using the lexicographic order induced by $\leqbang$. We use the convention that in order to compare two sequences of different lengths, we append sufficiently many trailing zeros to the shorter sequence so that the sequences have the same length. Explicitly, let $(a_i, b_i)$ and $(a_i', b_i')$ be two sequences of length $2m$ and $2n$, respectively. If $(a_i, b_i) < (a_i', b_i')$, then either:
\begin{enumerate}
\item There is some index $k$ such that $t_i = t_i'$ for all $i < k$, and $t_k \lebang t_k'$;
\item The sequence $(a_i, b_i)$ appears as a prefix of $(a_i', b_i')$, and $a_{2m+1}' > 0$; or,
\item The sequence $(a_i', b_i')$ appears as a prefix of $(a_i, b_i)$, and $a_{2n+1} < 0$.
\end{enumerate}

\begin{remark}
It is worth noting that the $\lebang$ order on $\Z \setminus \{0\}$ can also be defined by setting $x \lebang y$ whenever $1/x < 1/y$. This is extended to all of $\Z$ by comparing elements to zero via their sign.
\end{remark}

The central result of this section will be to show that the lexicographic order on sequences coincides with the total order defined by Theorem~\ref{thm:totalorder}. The proof of this consists of a series of straightforward but technical verifications utilizing Definitions~\ref{def:stdcomplexes} and \ref{def:aim}. We have been fairly explicit with the casework throughout, in the hopes that the reader will become comfortable with various manipulations involving standard complexes.

\begin{lemma}\label{lem:lexico1}
If $(a_i, b_i) \leq (a_i', b_i')$ in the lexicographic order, then $\cC(a_i, b_i) \leq \cC(a_i', b_i')$ in $\smash{\mfIhat}$. 
\end{lemma}

\begin{proof}
Let $(a_i, b_i)$ and $(a_i', b_i')$ be of length $2m$ and $2n$, respectively, and suppose that $(a_i, b_i) \leq (a_i', b_i')$. We construct a local map $f$ from $\cC = \cC(a_i, b_i)$ to $\cC' = \cC(a_i', b_i')$ as follows. If $(a_i, b_i) = (a_i', b_i')$, then we can take $f$ to be the identity map, so we assume $(a_i, b_i) < (a_i', b_i')$. Denote the generators of $\cC$ by $T_i$ and the generators of $\cC'$ by $T_i'$.

First, suppose that $(a_i, b_i)$ and $(a_i', b_i')$ agree up to index $k-1$, and that their $k$th symbols differ. We define $f$ on all generators except $T_{k}$ by
\[
f(T_i) = 
\begin{cases}
T_i' &\text{ if } i < k\\
0 & \text{ if } i > k.
\end{cases}
\]
Note that $f(T_0) = T_0'$. We define $f(T_{k})$ according to the following casework:
\begin{enumerate}
\item Suppose $k$ is odd, so that $a_k \lebang a_k'$. This means that $a_k = -$ and $a_k' = +$. Then we define $f(T_{k}) = 0$. Note that $f$ provides an isomorphism between the generators and arrows of $\cC$ lying to the left of $T_{k-1}$ (inclusive) and the analogous generators and arrows of $\cC'$; moreover, all of the generators to the right of $T_{k-1}$ are sent to zero by $f$. It is thus straightforward to check that $f$ commutes with $\partial$, since every $\partial$-arrow in $\cC$ either lies to the left of $T_{k-1}$, or has both endpoints mapped to zero. To check that $f\omega + \omega f \equiv 0$ mod $U$, the only nontrivial cases to verify are $T_{k-1}$ and $T_k$. For these, we have:
\[
\begin{cases}
\omega f (T_{k-1}) = \omega T_{k-1}' = 0 \\
f (\omega T_{k-1}) = f (T_k) = 0
\end{cases}
\text{ and }\ \ \ 
\begin{cases}
\omega f (T_{k}) = \omega(0) = 0 \\
f (\omega T_{k}) = f (0) = 0.
\end{cases}
\]
\item Suppose $k$ is even, so that $b_k \lebang b_k'$. There are now two subcases: 
\begin{enumerate}
\item Suppose that $b_k$ and $b_k'$ have the same sign. Then we define
\[
f(T_{k}) = U^{|b_k - b_k'|}T_{k}'. 
\]
It is straightforward to check that $f$ commutes with $\partial$, so we are left with verifying that $f\omega + \omega f \equiv 0$ mod $U$. For all generators except possibly $T_k$ and $T_{k+1}$, this is immediate; while for $T_k$ and $T_{k+1}$, the desired equality holds since $f(T_k) \equiv 0$ mod $U$ and $f(T_{k+1}) = 0$.
\item Suppose that $b_k < 0$ and $b_k' > 0$. Then we define $f(T_k) = 0$. Checking that $f\omega + \omega f \equiv 0$ mod $U$ is straightforward, since every $\omega$-arrow in $\cC$ either lies to the left of $T_{k-1}$, or has both endpoints mapped to zero. To check that $f$ commutes with $\partial$, the only nontrivial cases to verify are $T_{k-1}$ and $T_k$. We have:
\[
\ \ \ \ \ \ \ \ \ \ \ \ \ 
\begin{cases}
\partial f (T_{k-1}) = \partial T_{k-1}' = 0 \\
f (\partial T_{k-1}) = f (U^{|b_k|}T_{k}) = 0
\end{cases}
\text{ and }\ \ \ 
\begin{cases}
\partial f (T_{k}) = \partial(0) = 0 \\
f (\partial T_{k}) = f (0) = 0.
\end{cases}
\]
\end{enumerate}
\end{enumerate}

Now suppose that one of $(a_i, b_i)$ and $(a_i', b_i')$ appears as a prefix of the other. First consider the case when $m < n$ and $a_{2m+1}' > 0$. Then we define $f(T_i) = T_i'$ for all $0 \leq i \leq 2m$. The only nontrivial verification in this case is:
\[
\begin{cases}
\omega f (T_{2m}) = \omega T_{2m}' = 0 \\
f (\omega T_{2m}) = f (0) = 0.
\end{cases}
\]
Now suppose $m > n$ and $a_{2n+1} < 0$. Then we define $f(T_i) = T_i'$ for all $0 \leq i \leq 2n$, and zero otherwise. The only nontrivial verifications in this case are:
\[
\begin{cases}
\omega f (T_{2n}) = \omega T_{2n}' = 0 \\
f (\omega T_{2n}) = f (T_{2n+1}) = 0
\end{cases}
\text{ and }\ \ \ 
\begin{cases}
\omega f (T_{2n+1}) = \partial(0) = 0 \\
f (\omega T_{2n+1}) = f (0) = 0.
\end{cases}
\]
This defines $f$ in all cases. Since $f(T_0) = T_0'$, we see that $f$ is evidently a local map. This completes the proof.
\end{proof}

In order to prove the converse of Lemma~\ref{lem:lexico1}, it will be helpful to begin with a lemma which constrains local maps between standard complexes. Let $\cC$ and $\cC'$ be two standard complexes whose associated sequences agree up to index $k$. As in the proof of Lemma~\ref{lem:lexico1}, there is an obvious identification between the generators $T_i$ of $\cC$ and the generators $T_i'$ of $\cC'$ for $0 \leq i \leq k$. However, it is easy to check that a local map does not necessarily have to take $T_i$ to $T_i'$, even if $\cC = \cC'$. Instead, we have the following:

\begin{lemma}\label{lem:lexico2}
Let $\cC = \cC(a_i, b_i)$ and $\cC' = \cC(a_i', b_i')$ be two standard complexes, and suppose that $(a_i, b_i)$ and $(a_i', b_i')$ agree up to index $k$. Let $f\co \cC \rightarrow \cC'$ be a local map. Then $f(T_i)$ is supported by $T_i'$ for all $0 \leq i \leq k$.
\end{lemma}
\begin{proof}
For $i=0$, the claim is an immediate consequence of the definition of a local map. We proceed by induction. Say we have established the claim for $i < k$. We prove the claim for index $i + 1$. If $i+1$ is odd, we have the casework:
\begin{enumerate}
\item Assume $a_{i+1} = a_{i+1}' > 0$. Then $\omega T_{i+1} = T_i$, so $\omega f(T_{i+1})\equiv f(T_i) \bmod{U}$. By inductive hypothesis, it follows that $\omega f(T_{i+1})$ is supported by $T_i'$. This is only possible if $f(T_{i+1})$ is supported by $T_{i+1}'$, since $T_{i+1}'$ is the unique generator $g$ of $\cC'$ for which $T_i'$ appears in $\omega g$.
\item Assume $a_{i+1} = a_{i+1}' < 0$. Then $\omega T_i = T_{i+1}$, so $\omega f(T_i)\equiv f(T_{i+1})\bmod{U}$. By inductive hypothesis, $f(T_i)$ is supported by $T_i'$. Since $T_i'$ is the only generator $g$ of $\cC'$ for which $T_{i+1}'$ appears in $\omega g$, we see that $\omega f(T_i)$ is supported by $\omega T_i' = T_{i+1}'$. It follows that $f(T_{i+1})$ is supported by $T_{i+1}'$, as desired.
\end{enumerate}
\noindent
If $i + 1$ is even, let $\eta = |b_{i+1}| = |b_{i+1}'|$. Consider the casework:
\begin{enumerate}
\item Assume $b_{i+1} = b_{i+1}' > 0$. Then $\partial T_{i+1} =U^\eta T_i$, so $\partial f(T_{i+1})=U^\eta f(T_i)$. By inductive hypothesis, it follows that $\partial f(T_{i+1})$ is supported by $U^\eta T_i'$. This is only possible if $f(T_{i+1})$ is supported by $T_{i+1}'$, since $T_{i+1}'$ is the unique generator of $\cC'$ whose differential contains any $U$-power of $T_i'$.  
\item Assume $b_{i+1} = b_{i+1}' < 0$. Then $\partial T_i = U^\eta T_{i+1}$, so $\partial f(T_i)=U^\eta f(T_{i+1})$. By inductive hypothesis, $f(T_i)$ is supported by $T_i'$. Since $T_i'$ is the only generator of $\cC$ whose differential contains a $U$-power of $T_{i+1}'$, we see that $\partial f(T_i)$ is supported by $\partial T_i' = U^\eta T_{i+1}'$. It follows that $f(T_{i+1})$ is supported by $T_{i+1}'$, as desired.
\end{enumerate}
\noindent
This completes the induction.
\end{proof}

We now establish the converse of Lemma~\ref{lem:lexico1}:

\begin{lemma}\label{lem:lexico3}
If $(a_i, b_i) > (a_i', b_i')$ in the lexicographic order, then $\cC(a_i, b_i) > \cC(a_i', b_i')$ in $\smash{\mfIhat}$.
\end{lemma}
\begin{proof}
Let $(a_i, b_i)$ and $(a_i', b_i')$ be of length $2m$ and $2n$, respectively, and suppose that $(a_i, b_i) > (a_i', b_i')$. Write $\cC = \cC(a_i, b_i)$ and $\cC' = \cC(a_i', b_i')$, and denote the generators of $\cC$ and $\cC'$ by $T_i$ and $T_i'$. To establish the lemma, we proceed by contradiction. Thus, suppose that there is a local map $f \co \cC \rightarrow \cC'$. 

First, suppose that $(a_i, b_i)$ and $(a_i', b_i')$ agree up to index $k-1$, and that their $k$th symbols differ. By Lemma~\ref{lem:lexico2}, we have that $f(T_{k-1})$ is supported by $T_{k-1}'$. There are two cases depending on the parity of $k$. Suppose $k$ is odd, so that $a_k = +$ and $a_k' = -$. Then $\omega f(T_{k-1})$ is supported by $\omega T_{k-1}' = T_k'$, but $f (\omega T_k) = f(0) = 0$, a contradiction. Thus, we may assume that $k$ is even and $b_k \gebang b_k'$. There are three further subcases:
\begin{enumerate}
\item Suppose $b_k \gebang b_k' > 0$, so that $\partial T_k = U^{|b_k|}T_{k-1}$. Then $\partial f(T_{k})$ is supported by $U^{|b_k|} T_{k-1}'$. This is impossible, since $T_{k-1}'$ only appears in the image of $\partial$ with at least a $U$-power of $|b_k'|$, and $|b_k| < |b_k'|$. 
\item Suppose $b_k' \lebang b_k < 0$, so that $\partial T_{k-1} = U^{|b_k|}T_{k}$. This implies that $f (\partial T_{k-1})$ is in the image of multiplication by $U^{|b_k|}$. On the other hand, $f (\partial T_{k-1}) = \partial f(T_{k-1})$, which is supported by $\partial T_{k-1}' = U^{|b_k'|}T_k'$. This contradicts the fact that $|b_k'| < |b_k|$. 
\item Suppose $b_k > 0$ and $b_k' < 0$. Then $\partial f(T_{k-1})$ is supported by $\partial T_{k-1}' = U^{|b_k'|} T_k'$, but $f (\partial T_{k-1}) = f(0) = 0$, a contradiction.
\end{enumerate}

Now consider the case when one of $(a_i, b_i)$ and $(a_i', b_i')$ appears as a prefix of the other. First, suppose that $m < n$ and $a_{2m+1}' < 0$. Then $\omega f(T_{2m})$ is supported by $\omega T_{2m}' = T_{2m+1}'$, but $f(\omega T_{2m}) = f(0) = 0$, a contradiction. Thus, suppose that $m > n$ and $a_{2n+1} > 0$. Then $f(\omega T_{2n+1}) = f(T_{2n})$, which is supported by $T_{2n}'$. On the other hand, $f(\omega T_{2n+1}) \equiv \omega f(T_{2n+1}) \bmod U$. Since $T_{2n}'$ does not appear as a term in any image of $\omega$, this is a contradiction. This completes the proof.
\end{proof}

Putting Lemmas~\ref{lem:lexico1} and \ref{lem:lexico3} together immediately yields:

\begin{theorem}\label{thm:lexico}
Let $\cC = \cC(a_i, b_i)$ and $\cC' = \cC(a_i', b_i')$ be two standard complexes. Then $\cC \leq \cC'$ if and only if $(a_i, b_i) \leq (a_i', b_i')$ in the lexicographic order.
\end{theorem}

%%%%%%%%%%%%%%%%%%%%%%%%%%%%%%%%%%%%%%%%%%%%%%%%%%%%%%%%%%%%%%%%%%%%%%%%%%%%%%%%%%%%%%%%%%%%%%%%%%%%%%%%%%%%%%%%%%%%%%%%%%%%%%%%%%%%%%%%%%%%%%%%%%%%%%%%%%%%%%%%%%%%%%%%%%%%%%%%%%%%%%%%%%%%%%%%%%%%%%%%%%%%%%%%%%%%%%%%%%%%%%%%%%%%%%%%%%%%%%%%%%%%%%%%%%%%%%%%%%%%%%%%%%%%%%%%%%%%%%%%%%%%%%%%%%%%%%%%%%%%%%%%%%%%%%%%%%%%%%%%%%%%%%%%%%%%%%%%%%%%%%%%%%%%%%%%%%%%%%%%%%%

\section{Technical Notions}
\label{sec:5}
We now collect together several (seemingly unrelated) technical lemmas and definitions that will prove useful in Section~\ref{sec:6}. The main purpose of our work here will be to establish a more flexible language for dealing with standard complexes and almost $\inv$-maps.

\subsection{Augmented complexes}
We begin with slight enlargement of the class of standard complexes.
\begin{definition}\label{def:augcomplexes} 
Let
\[
(a_1, b_2, \ldots, a_{2n-1}, b_{2n}, a_{2n+1})
\]
be a sequence as in Definition~\ref{def:stdcomplexes}, but now ending in a symbol of odd (rather than even) index. We define a complex $\cC(a_1, b_2, \dots, a_{2n+1})$ in the same way as before, except that now we have $2n + 2$ generators $T_0, \ldots, T_{2n+1}$, and the final arrow is an $\omega$-arrow rather than a $\partial$-arrow. We call such a complex an \textit{augmented complex} and denote it using the same notation $\cC(t_i)$. Note that the homology of an augmented complex has two nontorsion towers, generated by the classes of $T_0$ and $T_{2n+1}$, so an augmented complex is not quite an almost $\inv$-complex in the sense of Definition~\ref{def:aic}.
\end{definition}

Definition~\ref{def:aim} still provides a notion of an almost $\inv$-morphism between two augmented complexes, or between an augmented complex and a standard complex. However, since augmented complexes have two nontorsion towers, we need to modify our definition of a local map:

\begin{definition}\label{def:augmorphism}
Let $\cC$ be an augmented complex and $\cA$ be any almost $\inv$-complex. Let $f \co \cC \rightarrow \cA$ be an almost $\inv$-morphism. If $f$ takes the class of $T_0$ to a nontorsion element in the homology of $\cA$, we say that $f$ is \textit{forwards local}; while if $f$ takes the class of $T_{2n+1}$ to a nontorsion element, we say that $f$ is \textit{backwards local}. If both of these hold, then we say that $f$ is \textit{totally local}.
\end{definition}

Now let $\cC$ be an augmented complex with generators $T_0, \ldots, T_{2n+1}$. If we reflect $\cC$ across a vertical line lying between $T_n$ and $T_{n+1}$, then we obtain a new augmented complex $\rC$, whose generators are the same as those of $\cC$ but listed in the reverse order. There is a slight subtlety in the case that the final generator of $\cC$ does not lie in grading zero, since we then have to shift the grading of our reversed complex so as to satisfy our normalization conventions. However, in this paper, we will only ever consider $\rC$ in the the case where the final generator of $\cC$ has grading zero. We formalize this in the following definition:

\begin{definition}\label{def:reversal}
Let $k = 2n +1$, and let $\cC = (t_1, \ldots, t_k)$ be an augmented complex whose final generator $T_k$ lies in grading zero. Then the \textit{reversed complex} $\rC = (\overline{t}_1, \ldots, \overline{t}_k)$ is the augmented complex defined by
\[
\overline{t}_i = - t_{k + 1 - i}
\]
for all $1 \leq i \leq k$. If $f \co \cC \rightarrow \cA$ is an almost $\inv$-map, then we also obtain a \textit{reversed map} $\overline{f} \co \rC \rightarrow \cA$. Denoting the generators of $\rC$ by $\overline{T}_0, \dots, \overline{T}_k$, this is defined by
\[
\overline{f}(\overline{T}_i) = f(T_{k - i})
\]
for all $0 \leq i \leq k$. It is easily checked that $\overline{f}$ is an almost $\inv$-map, which is backwards local if $f$ is forwards local (and vice-versa).
\end{definition}

\begin{remark}\label{rmk:reversal}
Note that since $k = 2n + 1$ is odd, $\rC$ is never equal to $\cC$. In particular, we have $\overline{t}_{n+1} = - t_{n+1}$. This cannot be equal to $t_{n+1}$, since $t_{n+1}$ is nonzero.
\end{remark}

%%%%%%%%%%%%%%%%%%%%%%%%%%%%%%%%%%%%%%%%%%%%%%%%%%%%%%%%%%%%%%%%%%%%%%%%%%%%%%%%%%%%

\subsection{The extension lemma and the merge lemma}
We now define a slight modification of the notion of an almost $\inv$-map. 

\begin{definition}
\label{def:shortmaps}

Let $\cA$ be an almost $\inv$-complex. If $\cC$ is a standard complex of length $2n$, then we say that a map $f$ from $\cC$ to $\cA$ is a \textit{short map} if 
\begin{enumerate}
\item $\partial f + f \partial = 0$; and
\item $\omega f(T_i) + f(\omega T_i) \equiv 0 \bmod U$ for all $0 \leq i \leq 2n - 1$.
\end{enumerate}
If $\cC$ is an augmented complex of length $2n+1$, then we say $f$ is a short map if
\begin{enumerate}
\item $\omega f + f \omega \equiv 0 \bmod U$; and
\item $\partial f(T_i) + f(\partial T_i) = 0$ for all $0 \leq i \leq 2n$.
\end{enumerate}
Thus, a short map is just an almost $\inv$-map where we waive the $\omega$- or $\partial$-condition on the final generator, depending on whether $\cC$ is a standard or augmented complex, respectively. Note that if $\cC$ is an augmented complex, then a short map is \textit{not} necessarily a chain map. In both cases, we denote the presence of a short map using the notation 
\[
f \co \cC \leadsto \cA.
\] 
If $f(T_0)$ is $U$-nontorsion class in $\cA$, then we say that $f$ is a \textit{local short map}. (This definition is the same for both standard and augmented complexes.)
\end{definition}

It will be useful for us to have the following basic terminology:

\begin{definition}\label{def:extension}
Let $\cC = \cC(t_i)$ be a standard or augmented complex. If $\cC' = \cC(t_i')$ is another standard or augmented complex, then we say that $\cC'$ \textit{extends} $\cC$ if $(t_i)$ is a prefix of $(t_i')$. 
\end{definition}

\begin{definition}\label{def:truncation}
Let $\cC = \cC(t_i)$ be a standard or augmented complex of length $k > 0$. For convenience, we define the \textit{truncation} of $\cC$ to be the complex corresponding to the prefix of $(t_i)$ with length $k - 1$.
\end{definition}

Note that if $\cC'$ extends $\cC$, then the obvious inclusion of $\cC$ into $\cC'$ is a local short map. Moreover, if we have a short map $f \co \cC' \leadsto \cA$, then precomposing with this inclusion gives a short map from $\cC$ into $\cA$ also. 

The following lemma says that if $\cC$ is a standard complex and $\cA$ is an almost $\inv$-complex, then any short map $f \co \cC \leadsto \cA$ can be extended to a genuine almost $\inv$-map into $\cA$, possibly from a larger domain. We leave it to the reader to formulate the analogous lemma for augmented complexes.

\begin{lemma}[Extension Lemma]\label{lem:extension}
Let $\cC$ be a standard complex, and let
\[
f\co \cC \leadsto \cA
\]
be a short map from $\cC$ to some almost $\inv$-complex $\cA$. Then there is a standard complex $\cC'$ extending $\cC$, together with a genuine almost $\inv$-map
\[
f'\co \cC' \rightarrow \cA
\] 
which agrees with $f$ on the generators of $\cC$.\footnote{Here, the generators of $\cC$ are viewed as generators of $\cC'$ in the obvious way.} Moreover, if $f$ is local, then $f'$ is local.
\end{lemma}
\begin{proof}
Let $\cC = \cC(a_i, b_i)$ be of length $2n$, and denote its generators by $T_i$ for $0 \leq i \leq 2n$. Consider $\omega f(T_{2n})$. If this is zero mod $U$, then $f$ is already an almost $\inv$-map. Thus, suppose that $\omega f(T_{2n})$ is nonzero mod $U$. The simplest situation will be when $\omega f(T_{2n})$ is in fact a cycle, in which case we set
\[
\cC' = \cC(a_1, b_2, \dots, a_{2n-1}, b_{2n}, -, -1).
\]
We define $f'$ to be equal to $f$ on the generators of $\cC$, and define $f'$ on the new generators $T_{2n+1}$ and $T_{2n+2}$ by $f'(T_{2n+1}) = \omega f(T_{2n})$ and $f'(T_{2n+2}) = 0$. To check that $f'$ is an almost $\inv$-map, the nontrivial cases are:
\[
\begin{cases}
\partial f'(T_{2n+1}) = \partial \omega f(T_{2n}) = 0 \\
f'(\partial T_{2n+1}) = f'(UT_{2n+2}) = 0
\end{cases}
\text{and} \ \ \
\begin{cases}
\omega f'(T_{2n+1}) = \omega^2 f(T_{2n}) \equiv 0 \\
f'(\omega T_{2n+1}) = f'(0) = 0.
\end{cases}
\]
This provides the extension in the case that $\partial \omega f(T_{2n}) = 0$.

Now suppose that $\partial \omega f(T_{2n}) \neq 0$. In this case, we (suggestively) denote $\tau_{2n+1} = \omega f(T_{2n})$. Since $\partial \tau_{2n+1} \neq 0$, the assumption that $\cA$ is reduced implies
\[
\partial \tau_{2n+1} = U \tau_{2n+2}
\]
for some cycle $\tau_{2n+2}$. Let $\cC'$ be as before, and define $f'(T_{2n+1}) = \tau_{2n+1}$ and $f'(T_{2n+2}) = \tau_{2n+2}$. A similar argument as above easily establishes that $f'$ satisfies all the conditions for being an almost $\inv$-map, except for the $\omega$-condition on the final generator $T_{2n+2}$. Indeed, we have $f'(\omega T_{2n+2}) = f'(0) = 0$, while $\omega f'(T_{2n+2}) = \omega \tau_{2n+2}$, which may be nonzero mod $U$. But this precisely means that $f'$ is a short map from $\cC'$ to $\cA$. Iterating the argument, we obtain a sequence of short maps
\[
f^{(m)}\co \cC(a_1, b_2, \dots, a_{2n-1}, b_{2n}, -, -1, -, -1, \dots, -, -1) \leadsto \cA.
\]
Let the final generator in each complex of the above sequence be denoted by $T_k$, where $k = 2n + 2m$. Note that the grading of $T_k$ increases with $m$. Since $\cC$ is bounded above, it follows that for $m$ sufficiently large, $\partial \omega f(T_k)$ must be zero. (If not, we would have $\partial \omega f(T_k) = U\tau_{k+2}$ for some nonzero $\tau_{k+2}$ in a higher grading than $T_{k}$.) Hence this process terminates and produces the desired extension. 
\end{proof}

\begin{remark}
Lemma~\ref{lem:extension} does not necessarily yield the maximal extension of a given short map (in the sense of the total order). Though this will not be used, it is perhaps instructive for the reader to think about the casework for finding the maximal extension.
\end{remark}

\begin{definition}\label{def:suffix}
Let $\cC$ and $\cC'$ be two standard or two augmented complexes of length $k$ and $k'$, respectively. We say that $\cC$ and $\cC'$ \textit{share a suffix} if there exist indices $0 \leq p \leq k$ and $0 \leq q \leq k'$ such that:
\begin{enumerate}
\item $p$ and $q$ are of the same parity; and, 
\item The generators and arrows of $\cC$ to the right of $T_p$ (inclusive) are isomorphic to the generators and arrows of $\cC'$ to the right of $T_q'$ (inclusive).
\end{enumerate}
See Figure~\ref{fig:suffix} for an illustration. Note that we require the isomorphism in (2) to be grading-preserving. In general, when discussing two complexes that share a suffix, we will assume that the suffix is maximal (that is, $p$ and $q$ are the smallest possible indices for which the above conditions hold).
\end{definition}

\begin{figure}[h!]
\center
\includegraphics[scale=0.8]{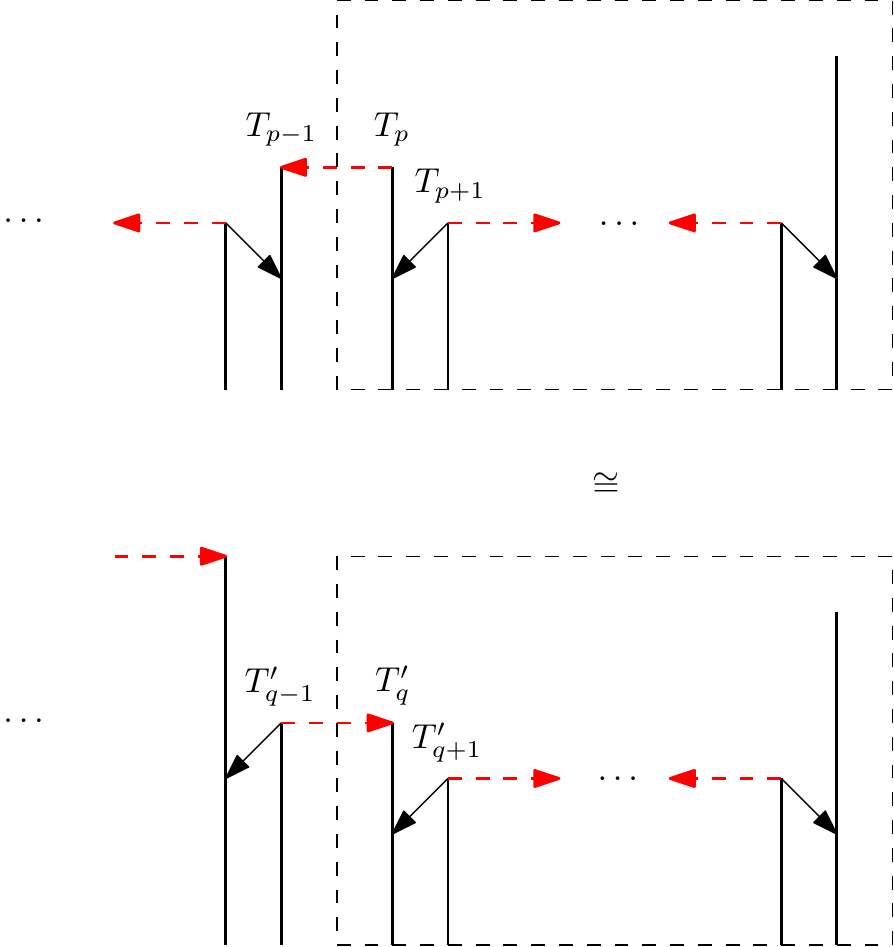}
\caption{Two complexes which share a suffix.}
\label{fig:suffix}
\end{figure}

It is clear that Definition~\ref{def:suffix} is equivalent to requiring that (a) the final generators of $\cC$ and $\cC'$ lie in the same grading; and (b) the symbol sequences $(t_i)$ and $(t_i')$ share a suffix in the lexicographic sense, starting at the indices $p+1$ and $q+1$, respectively. We stress that $p$ and $q$ are allowed to be zero and/or the final indices of their respective complexes. In the former situation, this means that one complex appears entirely at the tail-end of the other; while in the latter, the definition reduces to the requirement that the final generators of $\cC$ and $\cC'$ lie in the same grading.

Now suppose that $\cC$ and $\cC'$ share a suffix, and assume that we have short maps $f$ and $f'$ from $\cC$ and $\cC'$ into some almost $\inv$-complex $\cA$. Given the isomorphism afforded by (2) above, it is natural to attempt to construct a \textit{merged map} $g\co \cC \leadsto \cA$, which we define to be the sum of $f$ and $f'$ on the shared suffix of $\cC$ and $\cC'$. Since the $\partial$- and $\omega$-conditions are linear, this certainly still has the desired behavior with respect to the $\partial$- and $\omega$-arrows in the shared suffix. The following lemma explains that to complete this map to a short map from $\cC$ to $\cA$, it suffices to compare the symbols $t_p$ and $t_q'$.

For convenience, we define $t_0 = t_0' = 0$. We refer to the symbol zero as having sign zero, which is distinct from being either positive or negative.

\begin{lemma}[Merge Lemma]\label{lem:merge}
Let $\cC$ and $\cC'$ be two standard or two augmented complexes which share a suffix (as above), and let $\mathcal{A}$ be any almost $\inv$-complex. Suppose that we have short maps $f\co \cC \leadsto \mathcal{A}$ and $f'\co \cC' \leadsto \mathcal{A}$. If $t_p > t_q'$, then we may define a new short map $g \co \cC \leadsto \mathcal{A}$ with the properties
\[
g(T_i) = 
\begin{cases}
f(T_i) &\text{ if } i \leq p - 2 \\
f(T_i) + f'(T_{i - p + q}') &\text{ if } i \geq p.
\end{cases}
\]
We define $g$ on $T_{p-1}$ according to the rule:
\[
g(T_{p-1}) = 
\begin{cases}
f(T_{p-1}) &\text{ if } t_p \text{ and } t_q' \text{ have different signs } \\
f(T_{p-1}) + U^{|t_p - t_q'|}f'(T_{q-1}') &\text{ if } t_p \text{ and } t_q' \text{ have the same sign }
\end{cases}
\]
If $f$ and $f'$ are genuine almost $\inv$-maps, then so is $g$. Moreover, if $f$ and $f'$ are local, then $g$ is local.
\end{lemma}
\begin{proof}
We first consider the case where $p$ and $q$ are odd, so that $t_p = +$ and $t_q' = -$. (For example, this is the case in Figure~\ref{fig:suffix}.) Then we have defined $g(T_{p-1}) = f(T_{p-1})$. It is clear that $g$ satisfies the $\partial$-condition, since every $\partial$-arrow in $\cC$ either lies to the left of $T_{p-1}$ (in which case $g$ is equal to $f$) or to the right of $T_p$ (in which case $g = f + f'$). Thus we must verify that $g$ satisfies the $\omega$-condition. The nontrivial cases are:
\[
\begin{cases}
\omega g (T_{p-1}) = \omega f (T_{p-1}) \equiv f( \omega T_{p-1}) = f(0) = 0 \bmod U \\
g (\omega T_{p-1}) = g (0) = 0
\end{cases}
\]
and
\[
\begin{cases}
\omega g (T_{p}) =\omega f(T_p) + \omega f'(T_q') \equiv f(\omega T_p) + f' (\omega T_q') = f(T_{p-1}) \bmod U \\
g (\omega T_{p}) = g( T_{p-1} ) = f(T_{p-1}).
\end{cases}
\]

Now suppose that $p$ and $q$ are even. Consider the case when $p = 0$. Then $b_q' < 0$, and the only nontrivial verification is:
\[
\begin{cases}
\partial g (T_0) = \partial f(T_0) + \partial f'(T_q') = f(\partial T_0) + f'(\partial T_q') = 0\\
g (\partial T_0) = g (0) = 0.
\end{cases}
\]
Similarly, if $q=0$, then $b_p > 0$, and the only nontrivial verifications are:
\[
\begin{cases}
\partial g (T_{p-1}) = \partial f(T_{p-1}) = f(\partial T_{p-1}) = 0\\
g (\partial T_{p-1}) = g (0) = 0
\end{cases}
\]
and
\[
\begin{cases}
\partial g (T_{p}) = \partial f(T_p) + \partial f'(T_0') = f(\partial T_p) + f'(\partial T_0') = U^{b_p}f(T_{p-1})\\
g(\partial T_{p}) = g (U^{b_p} T_{p-1}) = U^{b_p} f(T_{p-1}).
\end{cases}
\]
The other cases are similar in flavor; we leave them as an exercise for the reader.

Finally, note that if $f$ and $f'$ are local, then $g$ is also local. Indeed, if $p \geq 2$, then we immediately have $g(T_0) = f(T_0)$ by the first equality in the definition of $g$. If $p = 1$, then $t_p = +$ and $t_q' = -$, so $g(T_0) = f(T_0)$ by the second equality. Finally, if $p = 0$, then $g(T_0) = f(T_0) + f'(T_q')$. Since $t_q' < 0$, we have that $f'(T_q')$ is a cycle in $\cA$, which is $U$-torsion since $q \neq 0$. Hence $g(T_0)$ is a $U$-nontorsion class in $\cA$.
\end{proof}

%%%%%%%%%%%%%%%%%%%%%%%%%%%%%%%%%%%%%%%%%%%%%%%%%%%%%%%%%%%%%%%%%%%%%%%%%%%%%%%%%%%%%%%%%%%%%%%%%%%%%%%%%%%%%%%%%%%%%%%%%%%%%%%%%%%%%%%%%%%%%%%%%%%%%%%%%%%%%%%%%%%%%%%%%%%%%%%%%%%%%%%%%%%%%%%%%%%%%%%%%%%%%%%%%%%%%%%%%%%%%%%%%%%%%%%%%%%%%%%%%%%%%%%%%%%%%%%%%%%%%%%%%%%%%%%%%%%%%%%%%%%%%%%%%%%%%%%%%%%%%%%%%%%%%%%%%%%%%%%%%%%%%%%%%%%%%%%%%%%%%%%%%%%%%%%%%%%%%%%%%%%%%%%%%%%%%%%%%%%%%%%%%%%%%%%%%%%%%%%%%%%%%%%%%%%%%%%%%%%%%%%%%%%%%%%%%%%%%%%%%%%%%%%%%%%%%%%%%%%%%%%%%%%%%%%%%%%%%%%%%%%%%%%%%%%%%%%%%%%%%%%%%%%%%%%%%%%%%%%%%%%%%%%%%%%%%%%%%%%%%%%%%%

\section{Parameterization of $\mfIhat$}
\label{sec:6}
In this section, we prove that every almost $\inv$-complex is locally equivalent to a standard complex, providing an explicit parameterization of $\smash{\mfIhat}$. This claim will follow from the fact that for any almost $\inv$-complex $\cA$, the set of standard complexes less than or equal to $\cA$ has a maximal element. Indeed, given this assertion, we immediately obtain the desired theorem from the following simple lemma: 

\begin{lemma}\label{lem:intermediate}
Let $\cC$ and $\cC'$ be two standard complexes with $\cC < \cC'$. Then there is a standard complex that lies strictly between $\cC$ and $\cC'$.
\end{lemma}
\begin{proof}
Let $\cC = \cC(a_i, b_i)$ and $\cC' = \cC(a_i', b_i')$ be of length $2m$ and $2n$, respectively. Assume that $\cC < \cC'$. If $(a_i, b_i)$ is a prefix of $(a_i', b_i')$, then any standard complex $\cC(a_1', b_2' \ldots, a_{2n-1}', b_{2n}', -, x)$ (with $x$ arbitrary) lies strictly between $\cC$ and $\cC'$. Otherwise, we may take any standard complex $\cC(a_1, b_2, \ldots, a_{2m-1}, b_{2m}, +, x)$. Note that we are using the understanding of the total order afforded by Theorem~\ref{thm:lexico}.
\end{proof}

\begin{theorem}\label{thm:parameterization}
Every almost $\iota$-complex is locally equivalent to a standard complex.
\end{theorem}
\begin{proof}
Let $\cA$ be an almost $\inv$-complex. Then we may consider the maximal standard complex $\cC$ with $\cC \leq \cA$, which exists by Theorem~\ref{thm:maximal}. Let $\cC'$ likewise be the minimal standard complex with $\cA \leq \cC'$. (This exists by dualizing and applying Theorem~\ref{thm:maximal}.) If $\cC$ and $\cC'$ are not equal, then by Lemma~\ref{lem:intermediate} there is a standard complex lying strictly between them. This is either less than or equal to $\cA$ or greater than or equal to $\cA$, contradicting the maximality/minimality of $\cC$ and $\cC'$.
\end{proof}

\noindent
In fact, we have the following slightly stronger result:

\begin{theorem}\label{thm:splits}
Let $\cA$ be an almost $\inv$-complex. Then $\cA$ is locally equivalent to a direct sum $\cC \oplus \cA'$, where $\cC$ is the standard complex locally equivalent to $\cA$.\footnote{Here, $\cA'$ is a complex satisfying all the requirements of Definition~~\ref{def:aic}, except that the homology of $\cA'$ is $U$-torsion.}
\end{theorem}

\noindent
Although Theorem~\ref{thm:splits} will not be used in the rest of this paper, we include it here for completeness. The proof of Theorem~\ref{thm:splits} will be given at end of the section.

The majority of this section is thus devoted to establishing the fact that the set of standard complexes $\cC \leq \cA$ has a maximal element. Instead of proving the desired theorem directly, it will be helpful for us to formulate a sequential version of the claim by introducing a set of auxiliary invariants, roughly analogous to the invariants $(a_i)$ considered by the second author in \cite{Hominfiniterank}. We begin with a slight notational modification to the definition of a standard complex. Let $(a_i, b_i)$ be a sequence of symbols as in Definition~\ref{def:stdcomplexes}, but with an even number of trailing zeros. We define the standard complex associated to such a sequence by disregarding the trailing zeros, so that 
\[
\cC(a_1, b_2, \ldots, a_{2n-1}, b_{2n}, 0, 0, \ldots, 0, 0) = \cC(a_1, b_2, \ldots, a_{2n-1}, b_{2n}).
\]
Similarly, we define the $i$th symbol of $\cC(a_1, b_2, \ldots, a_{2n-1}, b_{2n})$ to be zero for all $i > 2n$. Note that this is the only way for a complex to have a zero symbol.

\begin{definition}\label{def:auxiliaryinvariants}
Let $\cA$ be an almost $\inv$-complex. We inductively define a sequence of invariants $s_i(\cA) \in \Zbang$, as follows. Let
\[
s_1(\cA) = \max \phantom{}^{!} \{ t_1 \in \Zbang \mid \cC(t_1, \ldots, t_k) \leq \cA \}.
\]
That is, let $\mathfrak{S}_1$ be the set of standard complexes (of any length) which are less than or equal to $\cA$, and take the maximum (with respect to the order on $\Zbang$) over the set of first symbols appearing in $\mathfrak{S}_1$. Note that if the identity is less than or equal to $\cA$, then this set of symbols contains the element zero. 

Now suppose that $s_1 = s_1(\cA), \dots, s_i = s_i(\cA)$ have all been defined. Then we set
\[
s_{i+1}(\cA) = \max \phantom{}^{!} \{ t_{i+1} \in \Zbang \mid \cC(s_1, \ldots, s_i, t_{i+1}, \ldots, t_k) \leq \cA \}.
\]
That is, let $\mathfrak{S}_{i+1}$ be the set of standard complexes less than or equal to $\cA$, whose first $i$ symbols agree with the previously defined invariants $s_1(\cA), \dots, s_i(\cA)$. Then $s_{i+1}(\cA)$ is the maximum over the set of $(i + 1)$st symbols appearing in $\mathfrak{S}_{i+1}$.
\end{definition}

We stress that in the above definition, only standard complexes appear as elements of $\mathfrak{S}_i$, and not augmented complexes. Note that for $i$ odd, $s_i \in \{ -1, 0, +1\}$ while for $i$ even, $s_i \in \Zbang$.

\begin{remark}
It is not immediately clear that the $s_i(\cA)$ are well-defined. In each case, we must prove that the set of $i$th symbols appearing in $\mathfrak{S}_i$ has a maximal element. Note, however, that if $s_i(\cA) = 0$ for some $i$, then $s_j(\cA)$ is defined and equal to zero for all further $j \geq i$. 
\end{remark}

Presently, we will show that the $s_i(\cA)$ compute successive symbols in the desired maximal standard complex $\cC \leq \cA$. Roughly speaking, there are two ways that a family of complexes $\cC_n \leq \cA$ might fail to attain its supremum. Firstly, the sequence of symbols associated to a fixed coordinate might diverge in $\Zbang$; and secondly, the family $\{\cC_n\}$ might consist of complexes of successively greater and greater length. We deal with the first kind of divergence by establishing that the $s_i(\cA)$ are well-defined, while in Lemma~\ref{lem:stab}, we show that $s_i(\cA) = 0 $ for $i$ sufficiently large. Given these facts, it is not hard to check that the maximal standard complex $\cC \leq \cA$ exists and has symbol sequence given precisely by the $s_i(\cA)$.

\begin{lemma}\label{lem:s1}
The invariant $s_1(\cA)$ is well-defined.
\end{lemma}
\begin{proof}
To see that $\mathfrak{S}_1$ is nonempty, take any local short map from $\cC(0)$ into $\cA$. (This exists by sending the generator of $\cC(0)$ to any generator of the $U$-nontorsion tower of $\cA$.) By the extension lemma, this extends to a standard complex $\cC \leq \cA$. The fact that the set of first symbols appearing in $\mathfrak{S}_1$ has a maximal element is trivially true, since each $t_1 \in \{-1, 0, + 1\}$.
\end{proof}

We now prove that the $s_i(\cA)$ are well-defined. We proceed by induction on $i$. It is clear that if $s_i(\cA)$ is defined, then $\mathfrak{S}_{i+1}$ is nonempty. Hence our attention will be focused on proving that the set of $(i+1)$st symbols appearing in $\mathfrak{S}_{i+1}$ has a maximal element. This is trivial if $s_i(\cA) = 0$, so we will assume throughout that $\cA$ is an almost $\inv$-complex for which the invariants $s_1(\cA), \dots, s_i(\cA)$ are all defined and nonzero. For convenience, write $s_j = s_j(\cA)$ for all $1 \leq j \leq i$.

We begin with the following technical lemma:

\begin{lemma}\label{lem:notimU}
Let $s_j = s_j(\cA)$ for all $1 \leq j \leq i$. Suppose that we have a local short map from a standard or an augmented standard complex
\[
f \co \cC(s_1, \ldots, s_i) \leadsto \cA.
\]
Then $f(T_j)$ is not in $\im U$ for any $0 \leq j \leq i$. Note that $0$ is considered to be an element of $\im U$.
\end{lemma}
\begin{proof}
We proceed by contradiction. Let $j$ be the minimal index for which $f(T_j) \in \im U$, and let $f(T_j) = U\eta_j$. (Note here that $\eta_j$ is allowed to be zero.) Since $\cA$ does not have any $U$-nontorsion cycles of positive grading, it is clear that $j \neq 0$. Suppose that $j$ is even. If $b_j \neq 1$, it is easy to check that we then have a local short map 
\[
f': \cC(s_1, \ldots, s_j - 1) \leadsto \cA,
\] 
defined by setting $f'(T_k) = f(T_k)$ for all $k < j$ and $f'(T_j) = \eta_j$. By the extension lemma, this extends to a local map into $\cA$, contradicting the maximality of $s_j$. In the special case that $b_j = 1$, we instead argue as follows. Since $\partial f(T_j) = f(\partial T_j) = U f(T_{j-1})$, we have $\partial \eta_j = f(T_{j-1})$. Because $\cA$ is reduced, this shows $f(T_{j-1}) \in \im U$, contradicting the minimality of $j$.

Thus, we may assume $j$ is odd. Suppose $a_j = -$. Since $f(T_j) \equiv 0 \bmod U$, it is easily checked that the restriction of $f$ then gives a local map
\[
f' \co \cC(s_1, \ldots, s_{j-1}) \rightarrow \cA.
\]
This contradicts the maximality of $s_j$. Thus, suppose $s_j = a_j = +$. Then $f(T_{j-1}) \equiv \omega f(T_j) \bmod U$, so $f(T_{j-1}) \in \im U$. This contradicts the minimality of $j$.
\end{proof}

We now prove that the set of $(i+1)$st symbols appearing in $\mathfrak{S}_{i+1}$ has a maximal element. As in Lemma~\ref{lem:s1}, this is trivial in the case that $i + 1$ is odd, so we assume that $i + 1$ is even. A quick examination of the order on $\Zbang$ shows that the only subsets of $\Zbang$ which fail to achieve their supremum are subsets of $- \N$ which are unbounded below (in the usual sense). Thus, let $(x_n)$ be a sequence of $(i+1)$st symbols appearing in $\mathfrak{S}_{i+1}$ for which $x_n \rightarrow -\infty$. It suffices to prove that this implies the existence of a $(i+1)$st symbol in $\mathfrak{S}_{i+1}$ which is positive. We establish this over the course of the next few lemmas:

\begin{lemma}\label{lem:pos1}
Let $x_n$ be a sequence of negative integers with $x_n \rightarrow - \infty$, and suppose that we have a sequence of local short maps from the family of standard complexes $\cC_n = \cC(s_1, \ldots, s_i, x_n)$ into $\cA$:
\[
f_n \co \cC(s_1, \ldots, s_i, x_n) \leadsto \cA.
\]
Then there exists a forwards local map from the augmented complex 
\[
f_{\infty} \co \cC(s_1, \ldots, s_i) \rightarrow \cA.
\]
\end{lemma}
\begin{proof}
Since $\cA$ is bounded above, it is clear that for $n$ sufficiently large, we must have $f_n(T_{i+1}) = 0$. For such $n$, it is easily checked that we obtain a forwards local map $f_{\infty} \co \cC(s_1, \ldots, s_i) \rightarrow \cA$ by restricting $f_n$ to the generators $T_0, \ldots, T_{i}$ of $\cC(s_1, \ldots, s_i, x_n)$.
\end{proof}

\begin{lemma}\label{lem:pos2}
Suppose that we have a forwards local map from the augmented complex $\cC(s_1, \ldots, s_i)$ into $\cA$:
\[
f_{\infty} \co \cC(s_1, \ldots, s_i) \rightarrow \cA.
\]
Then there exists a local short map
\[
f \co \cC(s_1, \ldots, s_i, x) \leadsto \cA
\]
for some $x > 0$. 
\end{lemma}
\begin{proof}
Let $\eta_i = f_{\infty}(T_i)$. By Lemma~\ref{lem:notimU}, $\eta_i$ is not a boundary (as otherwise it would lie in $\im U$). There are now two cases, depending on whether the class of $\eta_i$ is $U$-torsion in the homology of $\cA$. First, suppose that $\eta_i$ is $U$-torsion, so that
\[
U^x \eta_i = \partial \eta_{i + 1}
\]
for some $\eta_{i+1}$ and $x > 0$. Then the desired local short map $f \co \cC(s_1, \ldots, s_i, x) \leadsto \cA$ is defined by setting $f = f_{\infty}$ on $T_0, \ldots, T_i$ and sending $f(T_{i+1}) = \eta_{i+1}$.

Thus, assume that $\eta_i$ is not $U$-torsion. Then $f_{\infty}$ is a totally local map from $\cC(s_1, \ldots, s_i)$ into $\cA$. Our strategy in this case will be to show that we can replace $f_{\infty}$ with a different local map into $\cA$, for which we can apply the argument of the previous paragraph. There are two possibilities. First, suppose that the grading of $T_i$ is less than zero. Then $\eta_i$ is a $U$-nontorsion cycle in (even) grading less than zero, so if $g$ is any cycle in grading zero which generates the $U$-nontorsion tower of $H_*(\cA)$, the element
\[
\eta_i' = \eta_i + U^{|\gr(\eta_i)/2|} g
\]
is a $U$-torsion cycle. Since $\eta_i' \equiv \eta_i \bmod U$, it is easily checked that changing $f_{\infty}(T_i)$ to $\eta_i'$ yields a forwards local map to which we can apply the argument of the previous paragraph. 

The more difficult case occurs when $T_i$ has grading zero. In this situation, we use the merge lemma. Let $\cC = \cC(s_1, \ldots, s_i)$, and consider the reversed complex
\[
\rC = \cC(\overline{s}_1, \cdots, \overline{s}_i) = \cC(-s_i, \ldots, -s_1).
\]
This admits a totally local map $\overline{f}_{\infty}$ into $\cA$. The truncation of $\rC$ admits a local short map into $\cA$, which extends to a genuine local map into $\cA$ by the extension lemma.  By maximality of $\cC$, we have that $\rC\leq \cC$. Assume for the moment that $i > 1$. Then by Remark~\ref{rmk:reversal}, the inequality $\rC \leq \cC$ must be strict. Let $k$ be the minimal such index for which $s_k\neq \bar{s}_k$. This implies that the symbols of $\cC$ and $\rC$ also agree for all indices greater than $i + 1 - k$, since
\[
s_{i + 1 - j} = - \overline{s}_j = - s_j = \overline{s}_{i + 1 - j}
\]
for all $1 \leq j < k$. Moreover, 
\[
s_{i+1-k} = - \overline{s}_k > - s_k = \overline{s}_{i+1-k}.
\]
By the merge lemma, we thus conclude that there is a short map $g \co \cC \leadsto \cA$ which is equal to $f_{\infty} + \overline{f}_{\infty}$ on all generators $T_j$ of index $j \geq i + 1 - k$. In particular, since both $f_{\infty}$ and $\overline{f}_{\infty}$ are backwards local, $g(T_i)$ is the sum of two $U$-nontorsion cycles and is hence a $U$-torsion cycle. Since $k \leq i - 1$, we also have $g(T_0) = f(T_0)$, showing $g$ is forwards local. We can now apply the argument in the first paragraph to $g$.

In the special case that $i = 1$, we argue as follows. Since $\cC$ is either $(+)$ or $(-)$, after taking the reversal, we have totally local maps from both $(+)$ and $(-)$ into $\cA$. Applying the merge lemma, we obtain a forwards local map from $(+)$ into $\cA$ which takes $T_1$ to a $U$-torsion cycle. Applying the argument of the first paragraph, we thus have a local short map from $(+, t)$ into $\cA$ for some $t > 0$. Since this extends to a local map into $\cA$, we (retroactively) have $s_1 = +$. This completes the proof. 
\end{proof}

Putting everything together, we finally have:

\begin{theorem}\label{thm:welldefined}
The invariants $s_i(\cA)$ are well-defined. 
\end{theorem}
\begin{proof}
We proceed by induction on $i$. The base case $i = 1$ is Lemma~\ref{lem:s1}. If $s_i(\cA) = 0$, then $s_{i+1}(\cA)$ is clearly defined and equal to zero. Hence we may assume that $s_1 = s_1(\cA), \dots, s_i = s_i(\cA)$ are all defined and nonzero. Let $\cC_n \in \mathfrak{S}_{i+1}$ be any family of standard complexes for which
\[
\cC_n = \cC(s_1, \dots, s_i, x_n, \dots)
\]
with $x_n \rightarrow - \infty$. Restricting each $\cC_n$ to the first $i + 1$ symbols produces a sequence of local short maps into $\cA$, satisfying the hypotheses of Lemma~\ref{lem:pos1}. By Lemma~\ref{lem:pos1} and Lemma~\ref{lem:pos2}, we then obtain a local short map 
\[
f \co \cC(s_1, \ldots, s_i, x) \leadsto \cA
\]
with $x > 0$. By the extension lemma, this extends to a local map into $\cA$. An examination of the order on $\Zbang$ then easily shows that the set of $(i+1)$st symbols in $\mathfrak{S}_{i+1}$ always contains a maximal element.
\end{proof}

We now show that the invariants $s_i(\cA)$ are eventually zero for $i$ sufficiently large.

\begin{lemma}\label{lem:stab}
Let $\cA$ be an almost $\inv$-complex. Then $s_i(\cA) = 0$ for all $i$ sufficiently large.
\end{lemma}

\begin{proof}
We proceed by contradiction. Fix $i$ large, and suppose that the invariants $s_1 = s_1(\cA), \ldots, s_i = s_i(\cA)$ are all nonzero. Then we have a local map from some standard complex
\[
f\co \cC(s_1, \ldots, s_i, t_{i+1}, \ldots, t_n) \rightarrow \cA.
\]
Since $\cA$ is finitely generated, Lemma~\ref{lem:notimU} implies that the gradings of the generators $T_j$ for $0 \leq j \leq i$ must lie in a bounded interval, as otherwise their images $f(T_j)$ would be in $\im U$. It follows that for $i$ sufficiently large, we must have $f(T_m) = f(T_n)$ for some $m \neq n$ with $0 \leq m, n \leq i$. Moreover, by considering only generators of the same parity, we may assume that $m \equiv n \bmod 2$.

Let $\cC_m = \cC(s_1, \ldots, s_m)$ and $\cC_n= \cC(s_1, \ldots, s_n)$. These admit local short maps into $\cA$ by restricting $f$. Moreover, since $T_m$ and $T_n$ lie in the same grading, $\cC_m$ and $\cC_n$ share a suffix. Because $m \neq n$, it is clear that we can always apply the merge lemma, either with $f = f|_{\cC_m}$ and $f' = f|_{\cC_n}$, or vice-versa. Without loss of generality, we assume the former. Then we have a local short map
\[
g\co \cC_m \leadsto \cA
\]
with $g(T_m) = f(T_m) + f(T_n) = 0$. This contradicts Lemma~\ref{lem:notimU}.
\end{proof}

We thus finally obtain the desired theorem:

\begin{theorem}\label{thm:maximal}
Let $\cA$ be any almost $\inv$-complex. Then the set of standard complexes less than or equal to $\cA$ is nonempty and has a maximal element.
\end{theorem}
\begin{proof}
It is clear that the standard complex with symbol sequence given by the $s_i(\cA)$ provides the desired $\cC$. Indeed, this is less than or equal to $\cA$, since it lies in some $\mathfrak{S}_j$ for $j$ large; it is maximal by maximality of the $s_i(\cA)$. 
\end{proof}

We now turn to the proof of the more refined Theorem~\ref{thm:splits}. This will depend on the following sequence of lemmas. By Lemma~\ref{lem:lexico2}, recall that if $f$ is a local map from a standard complex $\cC$ to itself, then $f(T_i)$ is supported by $T_i$ for all $i$. Our first lemma investigates what we can say when $f(T_i)$ is supported by some generator $T_j$ with $j \neq i$:

\begin{lemma}\label{lem:othersupp}
Let $\cC = \cC(t_1, \ldots, t_n)$ be a standard complex, and let $f \co \cC \rightarrow \cC$ be a local map from $\cC$ to itself. Let $T_i$ be any generator of $\cC$, and suppose that $f(T_i)$ is supported by $T_j$ for $j \neq i$. Then the following hold:
\begin{enumerate}
\item If $i$ and $j$ are of the same parity, then we have the inequalities of sequences:
\begin{enumerate}
\item $(t_{i+1}, t_{i+2}, \ldots, t_n) \lebang (t_{j+1}, t_{j+2}, \ldots, t_n)$, and
\item $(-t_i, -t_{i-1}, \ldots, -t_1) \lebang (-t_{j}, -t_{j-1}, \ldots, -t_1)$.
\end{enumerate}
\item If $i$ and $j$ are of different parities, then we have the inequalities of sequences:
\begin{enumerate}
\item $(t_{i+1}, t_{i+2}, \ldots, t_n) \lebang (-t_j, -t_{j-1}, \ldots, -t_1)$, and
\item $(-t_i, -t_{i-1}, \ldots, -t_1) \lebang (t_{j+1}, t_{j+2}, \ldots, t_n)$.
\end{enumerate}
\end{enumerate}
In the case that $i$ or $j$ is equal to $0$ or $n$, we take the above sequences to be $(0)$ wherever necessary.
\end{lemma}

\begin{proof}
The proof is similar to that of Lemma~\ref{lem:lexico2}. We first consider the situation when $i$ and $j$ are of the same parity. We prove inequality (a). Suppose that $i$ and $j$ are odd. Then:
\begin{enumerate}
\item If $t_{i+1} < 0$, then $\partial f(T_i) =  U^{|t_{i+1}|} f(T_{i+1})$. Since $f(T_i)$ is supported by $T_j$, this implies $\partial T_j \in \im U^{|t_{i+1}|}$. (Here, we use the structure of $\partial$ on $\cC$.) It follows that $t_{i+1} \leqbang t_{j+1}$.
\item If $t_{i+1} > 0$, then $\partial f(T_{i+1}) = U^{t_{i+1}}f(T_i)$. Since $f(T_i)$ is supported by $T_j$, this implies $U^{t_{i+1}} T_j \in \im \partial$. (Here, we use the structure of $\partial$ on $\cC$.) It follows that $t_{i+1} \leqbang t_{j+1}$. 
\end{enumerate}
If strict inequality holds in either of the above cases, then the desired inequality is established. Otherwise, it is easily checked that $f(T_{i+1})$ is supported by $T_{j+1}$. We may thus proceed by induction and consider the case when $i$ and $j$ are even. Then:
\begin{enumerate}
\item If $t_{i+1} < 0 $, then trivially $t_{i+1} \leqbang t_{j+1}$, since $t_{i+1} = -$.
\item If $t_{i+1} = 0$ (that is, $i = n$), then $\omega f(T_i) \equiv 0 \bmod U$. Since $f(T_i)$ is supported by $T_j$, this implies that $\omega T_j = 0$. (Here, we use the structure of $\omega$ on $\cC$.) Hence $t_{j+1} > 0$.
\item If $t_{i+1} > 0$, then $\omega f(T_{i+1}) \equiv f(T_i)$. Since $f(T_i)$ is supported by $T_j$, this implies $T_j \in \im \omega$. (Here, we use the structure of $\omega$ on $\cC$.) It follows that $t_{j+1} = t_{i+1} = +$. 
\end{enumerate}
Again, if strict inequality holds, then the desired inequality is established; otherwise, it is easily checked that $f(T_{i+1})$ is supported by $T_{j+1}$. Proceeding by induction establishes inequality (a). (Note that the two relevant sequences are of different lengths, since $j \neq i$. Thus they cannot be equal.) The proof of inequality (b) proceeds analogously, by inducting on decreasing index.\footnote{One can also think of this as considering the reversal of $\cC$.} More precisely, suppose that $i$ and $j$ are odd. Then:
\begin{enumerate}
\item If $t_i < 0$, then $\omega f(T_{i-1}) \equiv f(T_i) \bmod U$. Since $f(T_i)$ is supported by $T_j$, this implies $T_j \in \im \omega$. Hence $t_j = t_i = -$. 
\item If $t_i > 0$, then trivially $- t_i \leqbang -t_j$, since $-t_i = -$.
\end{enumerate} 
If $i$ and $j$ are even, we have:
\begin{enumerate}
\item If $t_{i} < 0$, then $\partial f(T_{i-1}) =  U^{|t_{i}|} f(T_{i})$. Since $f(T_i)$ is supported by $T_j$, this implies $U^{|t_{i}|} T_j \in \im \partial$. It follows that $t_{i} \geqbang t_{j}$, so $- t_i \leqbang - t_j$.
\item If $t_i = 0$ (that is, $i = 0$), then $\partial f(T_i) = 0$. Since $f(T_i)$ is supported by $T_j$, this implies $\partial T_j = 0$. Hence $t_j < 0$, so $-t_j > 0$.
\item If $t_{i} > 0$, then $\partial f(T_{i}) = U^{t_{i}}f(T_{i-1})$. Since $f(T_i)$ is supported by $T_j$, this implies $\partial T_j \in \im U^{t_i}$. It follows that $t_{i} \geqbang t_{j}$, so $-t_i \leqbang - t_j$.
\end{enumerate}
We thus proceed as before, except that the inductive step lowers the values of the indices $i$ and $j$ by one.

Now suppose that $i$ and $j$ are of different parities. Our strategy is again to compare the $\partial$-arrow adjacent to $T_i$ with the $\partial$-arrow adjacent to $T_j$. Now, however, one of these arrows lies to the right of its associated generator, while the other lies to the left. (Similarly for the $\omega$-arrows.) The analysis is similar to the previous case, and we leave the details to the reader. Note that since $n$ is even, a parity argument shows that the relevant symbol sequences have different lengths, and hence cannot be equal.
\end{proof}

\begin{lemma}\label{lem:selflocalmapinjective}
Let $\cC = \cC(t_1, \ldots, t_n)$ be a standard complex, and let $f$ be a local map from $\cC$ to itself. Then $f$ is injective.
\end{lemma}
\begin{proof}
We proceed by contradiction. Suppose that we have a linear combination
\[
f\left( \sum_i U^{n_i} T_i \right) = 0.
\]
Without loss of generality, we may factor out powers of $U$ from the above linear combination so that at least one generator appears with no $U$-power. Let $I$ be the set of indices corresponding to such ``naked" generators. For each $i \in I$, we associate to $i$ a symbol sequence as follows. If $i$ is even, we associate to $i$ the symbol sequence $(t_{i+1}, t_{i+2}, \ldots, t_n)$. If $i$ is odd, we associate to $i$ the symbol sequence $(-t_i, -t_{i-1}, \ldots, -t_1)$. Note that the sequences associated to different elements of $I$ have different lengths (possibly due to parity), and are thus distinct.

Let $j$ be the index in $I$ whose associated sequence is minimal with respect to the lexicographic order. By Lemma~\ref{lem:lexico2}, $f(T_{j})$ is supported by $T_{j}$. Since $f$ maps the above linear combination to zero, there must be some other index $i \in I$ for which $f(T_i)$ is supported by $T_{j}$. Lemma~\ref{lem:othersupp} then contradicts the minimality of $j$.
\end{proof}

\begin{lemma}\label{lem:selflocalmapisomorphism}
Let $\cC = \cC(t_1, \ldots, t_n)$ be a standard complex, and let $f$ be a local map from $\cC$ to itself. Then $f$ is an isomorphism.
\end{lemma}
\begin{proof}
By Lemma~\ref{lem:selflocalmapinjective}, $f$ is injective. Restricting $f$ to each grading, we obtain a linear map from a finite-dimensional $\F$-vector space to itself, which is surjective (since it is injective).
\end{proof}
\noindent
This finally yields the proof of Theorem~\ref{thm:splits}:

\begin{proof}[Theorem~\ref{thm:splits}]
Let $\cA$ be an almost $\inv$-complex, and let $\cC$ be the standard complex in its local equivalence class. Then we have local maps 
\[ f \co \cC \to \cA \quad  \text{ and } \quad g \co \cA \to \cC. \]
Then $g \circ f$ is a local map from $\cC$ to itself, which is an isomorphism by Lemma \ref{lem:selflocalmapisomorphism}. It follows that the short exact sequence
\[ 0 \to \cC \xrightarrow{f} \cA \to \cA / \Im f \to 0 \]
splits.
\end{proof}

%%%%%%%%%%%%%%%%%%%%%%%%%%%%%%%%%%%%%%%%%%%%%%%%%%%%%%%%%%%%%%%%%%%%%%%%%%%%%%%%%%%%%%%%%%%%%%%%%%%%%%%%%%%%%%%%%%%%%%%%%%%%%%%%%%%%%%%%%%%%%%%%%%%%%%%%%%%%%%%%%%%%%%%%%%%%%%%%%%%%%%%%%%%%%%%%%%%%%%%%%%%%%%%%%%%%%%%%%%%%%%%%%%%%%%%%%%%%%%%%%%%%%%%%%%%%%%%%%%%%%%%%%%%%%%%%%%%%%%%%%%%%%%%%%%%%%%%%%%%%%%%%%%%%%%%%%%%%%%%%%%%%%%%%%%%%%%%%%%%%%%%%%%%%%%%%%%%%%%%%%%%%%%%%%%%%%%%%%%%%%%%%%%%%%%%%%%%%%%%%%%%%%%%%%%%%%%%%%%%%%%%%%%%%%%%%%%%%%%%%%%%%%%%%%%%%%%%%%%%%%%%%%%%%%%%%%%%%%%%%%%%%%%%%%%%%%%%%%%%%%%%%%%%%%%%%%%%%%%%%%%%%%%%%%%%%%%%%%%%%%%%%%%%%%%%%%%%%%%%%%%%%%%%%%%%%%%%%%%%%%%%%%%%%%%%%%%%%%%%%%%%%%%%%%%%%%%%%%%%%%%%%%%%%%%

\section{Homomorphisms from $\mfIhat$ to $\Z$}\label{sec:7}
We have now shown that $\smash{\mfIhat}$ is in bijection with the set of standard complexes up to local equivalence. In principle, this gives an explicit identification of all of the elements of $\smash{\mfIhat}$. However, as we discuss in Section~\ref{sec:8}, we do not have a complete description of the group structure on $\smash{\mfIhat}$ in terms of the standard complex parameters. We thus settle for constructing a family of homomorphisms from $\smash{\mfIhat}$ into $\Z$, as follows. For any standard complex $\cC = \cC(a_i, b_i)$ and integer $n \geq 1$, define
\[
\phi_n(\cC) = (\# \text{ parameters } b_i = n) - (\# \text{ parameters } b_i = -n). 
\]
Thus, $\phi_n$ records the number of towers of height $n$ in $H_*(\cC)$, counted with sign. We define $\phi_n$ on all of $\smash{\mfIhat}$ by passing to the representative standard complex in each local equivalence class.

The majority of this section will be devoted to proving that the $\phi_n$ are homomorphisms. Given that we do not fully understand the group structure on $\smash{\mfIhat}$, it will be necessary to take an extremely roundabout approach to understanding the $\phi_n$. Our strategy will be to express $\phi_n$ in terms of other homomorphisms for which (fortuitously) a complete understanding of the group structure on $\smash{\mfIhat}$ is not needed.

\subsection{Paired bases and shift maps}

\begin{definition}\label{def:pairedbasis}
Let $\cA$ be an almost $\inv$-complex. A \textit{paired basis} for $\cA$ is an $\mathbb{F}[U]$-module basis consisting of (homogenous) generators $\{x, y_i, z_i\}_{i = 1}^n$, together with positive integers $\{\eta_i\}_{i = 1}^n$, such that
\begin{align*}
&\partial x = 0, \\
&\partial y_i = \smash{U^{\eta_i}} z_i, \text{ and} \\
&\partial z_i = 0.
\end{align*}
Note that $[x]$ generates the $U$-nontorsion tower of $H_*(\cC)$, while each $[z_i]$ generates a $U$-torsion tower of height $\eta_i$. Accordingly, we refer to $x$ as the \textit{$U$-nontorsion generator}, and the $z_i$ as the \textit{$U$-torsion generators}. We call the $y_i$ the \textit{non-cycle generators}.
\end{definition}

\begin{example}\label{ex:standardpaired}
If $\cC$ is a standard complex, then the $T_i$ form a paired basis for $\cC$. More precisely, let $x = T_0$ and set
\[
\begin{cases}
y_i = T_{2i-1} \\
z_i = T_{2i}
\end{cases}
\text{ or } \ \ \ \ \
\begin{cases}
y_i = T_{2i} \\
z_i = T_{2i-1},
\end{cases}
\]
depending on whether $b_{2i} < 0$ or $b_{2i} > 0$, respectively. This defines a paired basis for $\cC$ with $\eta_i = |b_{2i}|$.
\end{example}

\begin{example}\label{ex:tensorpaired}
If $\cC_1$ and $\cC_2$ are two standard complexes, then the usual tensor product basis for $\cC_1 \otimes \cC_2$ is \textit{not} a paired basis. To construct a paired basis for $\cC_1 \otimes \cC_2$, we perform the change-of-basis depicted in Figure~\ref{fig:paired}. More precisely, let $\{x, y_i, z_i; \eta_i\}_{i=1}^m$ be a paired basis for $\cC_1$, and (abusing notation) let $\{x, y_j, z_j; \eta_j\}_{j=1}^n$ be a paired basis for $\cC_2$. (It will be clear from context which generators lie in $\cC_1$ and which lie in $\cC_2$.) Then the $U$-nontorsion generator of $\cC_1 \otimes \cC_2$ is given by $x_{00}=x\otimes x$. The $U$-torsion generators fall into two types, which we denote by $z_{ij}$ for $(i, j) \neq (0,0)$ and $\zeta_{ij}$ for $i, j > 0$. These are defined by
\[
z_{ij} = 
\begin{cases}
z_i\otimes x & \text{if } j = 0 \\
x\otimes z_j & \text{if } i = 0 \\
z_i \otimes z_j & \text{else} \\
\end{cases}
\ \ \ \text{ and } \ \ \
\zeta_{ij} = 
\begin{cases}
U^{\eta_i-\eta_j}z_{i}\otimes y_j + y_{i}\otimes z_j & \text{if } \eta_i \geq \eta_j\\
z_{i}\otimes y_j +U^{\eta_j-\eta_i} y_{i}\otimes z_j & \text{if }  \eta_i<\eta_j.
\end{cases}
\]
It is easily checked that the $z_{ij}$ and $\zeta_{ij}$ are cycles. Correspondingly, there are two types of non-cycle generators, which we denote by $y_{ij}$ for $(i, j) \neq (0,0)$ and $\upsilon_{ij}$ for $i, j >0$. These are defined by
\[
y_{ij} =
\begin{cases}
y_i\otimes x &\text{if } j = 0\\
x\otimes y_j&\text{if } i = 0\\
z_i\otimes y_j &\text{if } i, j > 0 \text{ and } \eta_i\geq\eta_j\\
y_i\otimes z_j &\text{if } i, j > 0 \text{ and } \eta_i<\eta_j
\end{cases}
\ \ \ \text{ and } \ \ \
\upsilon_{ij} =y_i\otimes y_j.
\]
It is straightforward to check that
\begin{align*}
&\partial y_{ij} = U^{\min(\eta_i, \eta_j)} z_{ij} \text{ and} \\
&\partial \upsilon_{ij} =U^{\min(\eta_i, \eta_j)} \zeta_{ij}
\end{align*}
with the understanding that in the first line, $\partial y_{i0} = U^{\eta_i} z_{i0}$ and $\partial y_{0j} = U^{\eta_j} z_{0j}$. Let
\[
\eta_{ij} = 
\begin{cases}
\eta_i &\text{if } j = 0 \\
\eta_j &\text{if } i = 0 \\
\min(\eta_i, \eta_j) &\text{else}.
\end{cases}
\]
Then the collection $\{x_{00}, y_{ij}, \upsilon_{ij}, z_{ij}, \zeta_{ij}; \eta_{ij}\}$ forms a paired basis for $\cC_1 \otimes \cC_2$, where each $\eta_{ij}$ (with $i, j > 0$)  appears in both the differential $\partial y_{ij} = U^{\eta_{ij}} z_{ij}$ and the differential $\partial \upsilon_{ij} = U^{\eta_{ij}} \zeta_{ij}$.
\end{example}

\begin{figure}[h!]
\center
\includegraphics[scale=0.85]{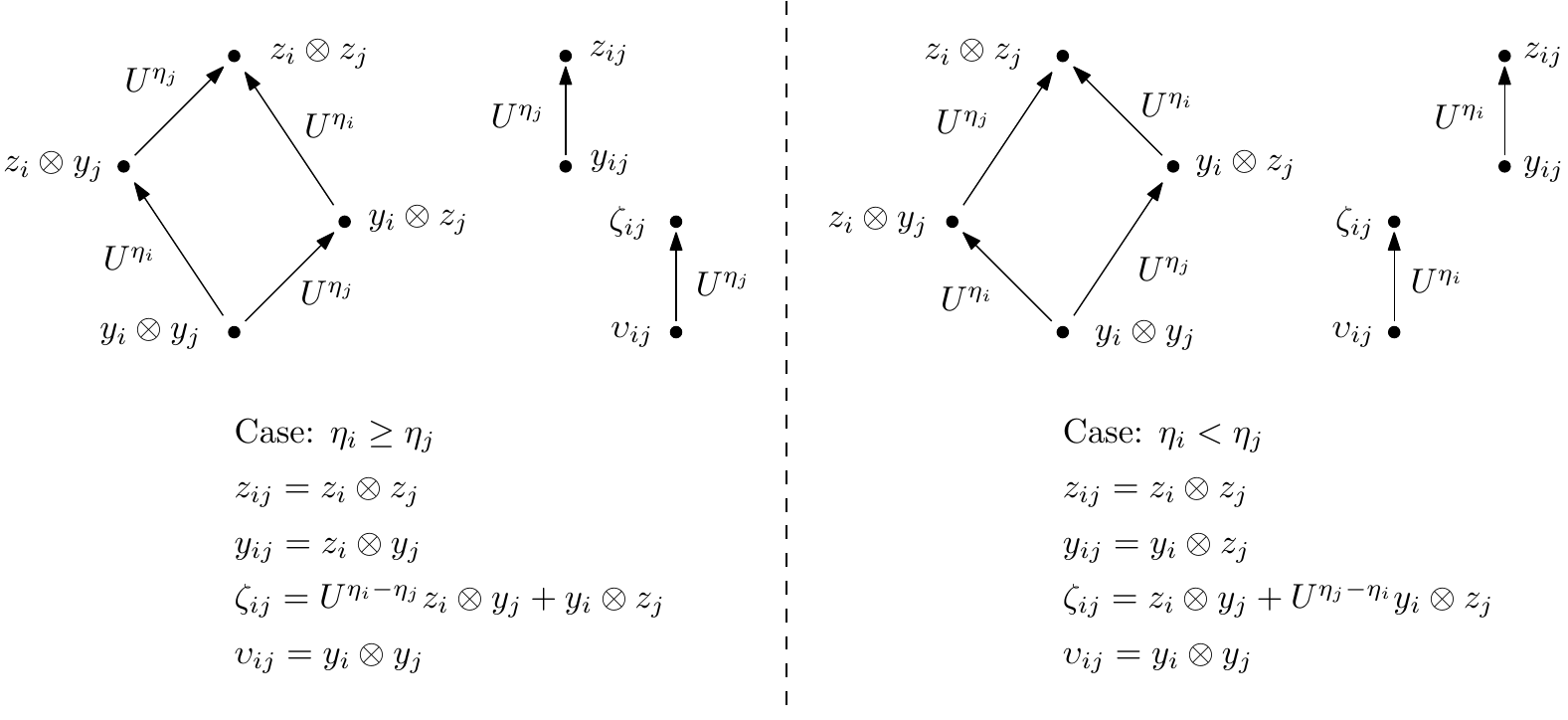}
\caption{Schematic depiction of a paired basis for $\cC_1 \otimes \cC_2$. Generators are placed vertically according to their gradings. Arrows record the action of $\partial$, with labels representing multiplication by some power of $U$. For example, $\partial (y_i \otimes y_j) = U^{\eta_i}z_i \otimes y_j + U^{\eta_j}y_i \otimes z_j$.}
\label{fig:paired}
\end{figure}

\begin{definition}\label{def:shiftmap}
For each $n \geq 1$, define the \textit{shift map} $\shift_n \from \mfIhat \to \mfIhat$ as follows. Let $\cC = \cC(a_i, b_i)$ be a standard complex. Then $\shift_n(\cC) = (a_i', b_i')$ is defined to be the standard complex (of the same length) with $a_i' = a_i$ and
\[
b_i' =
\begin{cases}
b_i &\text{if } |b_i| < n\\
b_i + 1 &\text{if } b_i \geq n\\
b_i - 1 &\text{if } b_i \leq -n.
\end{cases}
\]
Note that the (ungraded) homology of $\shift_n(\cC)$ is obtained from the homology of $\cC$ by replacing each tower of height $\geq n$ by a tower of one greater height. If $\cC$ has a paired basis given by $\{x, y_i, z_i; \eta_i\}$, then $\shift_n(\cC)$ has a corresponding paired basis given by $\{x', y_i', z_i'; \eta_i'\}$, where
\begin{equation}\label{eqn:7.1}
\eta_i' = 
\begin{cases}
\eta_i & \text{if } \eta_i < n \\
\eta_i + 1 & \text{if } \eta_i \geq n.
\end{cases}
\end{equation}
We extend $\shift_n$ to all of $\mfIhat$ by passing to the representative standard complex in each equivalence class.
\end{definition}

\begin{definition}\label{def:shiftcor}
Let $\cC$ be a standard complex and let $\cC' = \shift_n(\cC)$. We refer to the correspondence 
\[ 
\{x, y_i, z_i\} \mapsto \{x', y_i', z_i'\}
\]
between the unprimed generators of $\cC$ and the primed generators of $\cC'$ as the \textit{shift correspondence}. We extend this to all of $\cC$ by imposing linearity and $U$-equivariance. Note that the shift correspondence is not a chain map, nor is it grading-preserving. However, it does commute with $\omega$.

Similarly, let $\cC_1$ and $\cC_2$ be standard complexes. Then $\shift_n(\cC_1) \otimes \shift_n(\cC_2)$ has a paired basis defined by viewing $\shift_n(\cC_1)$ and $\shift_n(\cC_2)$ as standard complexes and applying the basis change of Example~\ref{ex:tensorpaired}. We denote the paired basis generators of $\shift_n(\cC_1) \otimes \shift_n(\cC_2)$ constructed in this manner by 
\[
\{x_{00}', y_{ij}', \upsilon_{ij}', z_{ij}', \zeta_{ij}'; \eta_{ij}'\}.
\]
An examination of the definition of $\eta_{ij}$ shows that
\begin{equation}\label{eqn:7.2}
\eta_{ij}' =
\begin{cases}
\eta_{ij} & \text{if } \eta_{ij} < n\\
\eta_{ij} + 1 & \text{if } \eta_{ij} \geq n.
\end{cases}
\end{equation}
In this situation, we similarly refer to the correspondence between the unprimed generators of $\cC_1 \otimes \cC_2$ and the primed generators of $\shift_n(\cC_1) \otimes \shift_n(\cC_2)$ as the shift correspondence. Note, however, that this is \textit{not} the tensor product of the shift correspondences $\cC_1 \mapsto \shift_n(\cC_1)$ and $\cC_2 \mapsto \shift_n(\cC_2)$. (We refer to this latter correspondence as the \textit{tensor product correspondence}.) Instead, we have:
\end{definition}

\begin{lemma}\label{lem:shiftcor}
Let $\cC_1$ and $\cC_2$ be standard complexes. Then the shift correspondence $\cC_1 \otimes \cC_2 \mapsto \shift_n(\cC_1) \otimes \shift_n(\cC_2)$ is congruent to the tensor product correspondence modulo~$U$.
\end{lemma}
\begin{proof}
Denote the paired basis for $\cC_1$ by $\{x, y_i, z_i; \eta_i\}$ and the paired basis for $\cC_2$ by $\{x, y_j, z_j; \eta_j\}$. We likewise use the primed notation for $\cC_1' = \shift_n(\cC_1)$ and $\cC_2' = \shift_n(\cC_2)$. An examination of (\ref{eqn:7.1}) shows that $\eta_i \geq \eta_j$ if and only if $\eta_i' \geq \eta_j'$. It is then straightforward to check that the shift correspondence is equal to the tensor product correspondence on all generators except possibly the $\zeta_{ij}$. To see why the two might differ in this case, suppose that $\eta_i \geq \eta_j$. Then $\zeta_{ij}'$ is equal to
\[
U^{\eta_i'-\eta_j'}z_{i}'\otimes y_j' + y_{i}'\otimes z_j',
\]
while the tensor product correspondence instead sends $\zeta_{ij}$ to 
\[
U^{\eta_i-\eta_j}z_{i}'\otimes y_j' + y_{i}'\otimes z_j'.
\]
Now, if $\eta_i - \eta_j = 0$, then $\eta_i' - \eta_j' = 0$. However, if $\eta_i - \eta_j > 0$, then either $\eta_i' - \eta_j' = \eta_i - \eta_j$, or $\eta_i' - \eta_j' = \eta_i - \eta_j + 1$. This shows that the two expressions above do not have to be equal, but are congruent mod $U$. A similar argument holds in the case that $\eta_i < \eta_j$.
\end{proof}

The main goal of this subsection will be to show that $\shift_n$ is a homomorphism. We begin with the following auxiliary definition:

\begin{definition}\label{def:almostchain}
Let $\cC$ be a standard complex with paired basis $\{x, y_i, z_i; \eta_i\}$ and let $\cA$ be any almost $\inv$-complex. An \emph{almost chain correspondence} $f\co \cC \rightarrow \cA$ is an ungraded $\ff[U]$-module map for which:
\begin{enumerate}
\item $\partial f(x) = 0$ and $\partial f(z_i) = 0$ for all $i$,
\item $\partial f(y_i)\equiv U^{\eta_i}f(z_i)\bmod{U^{\eta_i+1}}$ for all $i$; and,
\item $f\omega + \omega f \equiv 0 \bmod U$.
\end{enumerate}
Thus $f$ sends cycles to cycles, but $f(\partial y_i) + \partial f(y_i)$ is only zero modulo $U^{\eta_i + 1}$. We stress that $f$ is not required to be graded (or even homogeneous), although it is linear and $U$-equivariant.
\end{definition}

\begin{remark}
It turns out that for our application, the condition $\partial f(z_i) = 0$ in Definition~\ref{def:almostchain} is unnecessary. However, we have included it for completeness.
\end{remark}

\noindent
The main import of Definition~\ref{def:almostchain} will be the following lemma, which explains how to obtain a genuine almost $\inv$-map from an almost chain correspondence. In our context, we will need to construct almost $\inv$-maps between various complexes, but it will often be more convenient to construct almost chain correspondences instead (which is why we have introduced Definition~\ref{def:almostchain}).

\begin{lemma}\label{lem:almostchaingood}
Let $f \co \cC \rightarrow \cA$ be an almost chain correspondence. Then there exists an almost $\inv$-map $g \co \cC \rightarrow \cA$. Suppose moreover that the homogenous part of $f(x)$ in grading zero is a $U$-nontorsion class in $H_*(\cA)$. Then $g$ is local.
\end{lemma}
\begin{proof}
For any element $e$ of $\cC$, let $[e]_d$ denote the homogeneous part of $e$ lying in grading $d$. Define
\begin{align*}
g(x) &= [f(x)]_0, \\
g(y_i) &= [f(y_i)]_{\gr(y_i)}, \text{and} \\
g(z_i) &= [\partial f(y_i)/U^{\eta_i}]_{\gr(z_i)},
\end{align*}
extending linearly and $U$-equivariantly. Clearly, $g$ is homogeneous and grading-preserving. To check that $g$ is a chain map, we use the fact that $\partial$ is graded:
\begin{align*}
\partial g(y_i) &= [\partial f(y_i)]_{\gr(y_i) - 1} \\
&= U^{\eta_i} [\partial f(y_i)/U^{\eta_i}]_{\gr(y_i) - 1 + 2\eta_i} \\
&= U^{\eta_i} g(z_i)
\end{align*}
and
\[
\partial g(z_i) = [\partial^2 f(y_i)/U^{\eta_i}]_{\gr(z_i)- 1} = 0.
\]
This shows that $g$ is a chain map. To see that $g$ satisfies the $\omega$-condition, note that $g(z_i) \equiv [f(z_i)]_{\gr(z_i)} \bmod U$ (using (2) of Definition~\ref{def:almostchain}). Since this congruence is an equality for the other basis generators, we in fact have
\[
g(e) \equiv [f(e)]_{\gr(e)} \bmod U
\]
for any homogenous element $e$ of $\cC$. Using the fact that $\omega$ is graded, it follows that
\[
\omega g(T_i) \equiv \omega [f(T_i)]_{\gr(T_i)} = [ \omega f(T_i) ]_{\gr(T_i)} \equiv [ f(\omega T_i) ]_{\gr(T_i)} \equiv g(\omega T_i) \bmod{U}
\]
for any generator $T_i$ of $\cC$. This completes the proof.
\end{proof}

We now come to the main technical lemma of this section:

\begin{lemma}\label{lem:localshiftmap}
Let $\cC$, $\cC_1$, and $\cC_2$ be standard complexes, and suppose we have a local map $f\co \cC \rightarrow \cC_1 \otimes \cC_2$. Then there is a local map $f' \co \shift_n(\cC) \rightarrow \shift_n(\cC_1) \otimes \shift_n(\cC_2)$.
\end{lemma}
\begin{proof}
By Lemma~\ref{lem:almostchaingood}, it suffices to construct an almost chain correspondence between $\shift_n(\cC)$ and $\shift_n(\cC_1) \otimes \shift_n(\cC_2)$. Let $\{x, y_k, z_k; \eta_k\}$ be the usual paired basis for $\cC$, and let $\{x_{00}, y_{ij}, \upsilon_{ij}, z_{ij}, \zeta_{ij}; \eta_{ij}\}$ be the basis for $\cC_1 \otimes \cC_2$ constructed in Example~\ref{ex:tensorpaired}. We similarly use the primed notation of Definition~\ref{def:shiftcor} for the paired bases of $\shift_n(\cC)$ and $\shift_n(\cC_1) \otimes \shift_n(\cC_2)$. Before we begin, it will be helpful to explicitly examine $f(y_k)$ and $f(z_k)$. Write $f(y_k)$ as the sum of ($U$-powers of) the non-cycle generators $y_{ij}$ and $\upsilon_{ij}$, together with possibly some cycle generators:
\begin{equation}\label{eqn:7.3}
f(y_k) = \text{cycles}\phantom{.} + \sum U^{*} y_{ij} + \sum U^{*} \upsilon_{ij}.
\end{equation}
Since $f$ is a chain map, we have $\partial f(y_k) = U^{\eta_k} f(z_k)$. It follows that $f(z_k)$ must be the sum of ($U$-powers of) the $U$-torsion generators with the same index:
\begin{equation}\label{eqn:7.4}
f(z_k) = \sum U^{*+ \eta_{ij} - \eta_k} z_{ij} + \sum U^{*+ \eta_{ij} - \eta_k} \zeta_{ij}.
\end{equation}
Here, the $U$-exponents $U^*$ are the same in (\ref{eqn:7.3}) as they are in (\ref{eqn:7.4}), and we have the obvious equality of index sets between (\ref{eqn:7.3}) and (\ref{eqn:7.4}).

We now define $f'$. On the cycle generators of $\cC$, let
\begin{align*}
&f'(x') = f(x)' \text{ and} \\
&f'(z_k') = f(z_k)'.
\end{align*}
Defining $f'(y_k')$ is more complicated. There are two possibilities. If $\eta_k < n$, let
\[
f'(y_k') = f(y_k)'
\]
as before. If $\eta_k \geq n$, we first separate the non-cycle generators appearing in $f(y_k)$ into those whose indices have $\eta_{ij} < n$ and those with $\eta_{ij} \geq n$: 
\begin{align}\label{eqn:7.5}
f(y_k) = \text{cycles}\phantom{.} &+ \sum_{\eta_{ij} < n} U^{*} y_{ij} + \sum_{\eta_{ij} \geq n} U^{*} y_{ij} \\
&+  \sum_{\eta_{ij} < n} U^{*} \upsilon_{ij} + \sum_{\eta_{ij} \geq n} U^{*} \upsilon_{ij}. \nonumber
\end{align}
We then define
\begin{align}\label{eqn:7.6}
f'(y_k') = \text{cycles}'\phantom{.} &+ \sum_{\eta_{ij} < n} U^{*+1} y_{ij}' + \sum_{\eta_{ij} \geq n} U^{*} y_{ij}' \\
&+  \sum_{\eta_{ij} < n} U^{*+1} \upsilon_{ij}' + \sum_{\eta_{ij} \geq n} U^{*} \upsilon_{ij}'. \nonumber 
\end{align}
That is, $f'(y_k')$ is defined as before, except that whenever a non-cycle generator with $\eta_{ij} < n$ appears, we multiply it by an extra power of $U$. 

We now show that $f'$ is an almost chain correspondence. Since the shift correspondence sends cycles to cycles, it is clear that $\partial f'(x') = 0$ and $\partial f'(z_k') = 0$. To check the $\partial$-condition on $y_k$, first suppose that $\eta_k < n$. Then $\eta_k' = \eta_k$. Applying the shift correspondence to (\ref{eqn:7.3}) and taking the differential yields
\begin{align*}
\partial f'(y_k') = \sum U^{* + \eta_{ij}'} z_{ij}' + \sum U^{* + \eta_{ij}'} \zeta_{ij}'.
\end{align*}
Applying the shift correspondence to (\ref{eqn:7.4}) and multiplying through by $U^{\eta_k'}$ gives
\begin{align*}
U^{\eta_k'} f'(z_k') = \sum U^{*+ \eta_{ij}} z_{ij}' + \sum U^{*+ \eta_{ij}} \zeta_{ij}'.
\end{align*}
We thus see these two expressions do not have to be equal, since there may be some index pair $(i,j)$ for which $\eta_{ij}' = \eta_{ij} + 1$. However, this only occurs when $\eta_{ij} \geq n$. Since $\eta_k < n$, it follows that reducing the two expressions above modulo $\smash{U^{\eta_k + 1}}$ sends both of the offending terms to zero.

Now suppose that $\eta_k \geq n$. Then $\eta_k' = \eta_k + 1$. Taking the differential of (\ref{eqn:7.6}) yields 
\begin{align}\label{eqn:7.7}
\partial f'(y_k') = &\sum_{\eta_{ij} < n} U^{*+1 + \eta_{ij}'} z_{ij}' + \sum_{\eta_{ij} \geq n} U^{* + \eta_{ij}'} z_{ij}'  + \\
&\sum_{\eta_{ij} < n} U^{*+1 + \eta_{ij}'} \zeta_{ij}' + \sum_{\eta_{ij} \geq n} U^{* + \eta_{ij}'} \zeta_{ij}'. \nonumber
\end{align}
Applying the shift correspondence to (\ref{eqn:7.4}) and multiplying through by $U^{\eta_k'}$ gives
\begin{align*}
U^{\eta_k'} f'(z_k') = \sum U^{*+ \eta_{ij} + 1} z_{ij}' + \sum U^{*+ \eta_{ij} + 1} \zeta_{ij}'.
\end{align*}
It is easily checked that these two expressions are equal. Indeed, if $\eta_{ij} \geq n$, then $\eta_{ij}' = \eta_{ij} + 1$. On the other hand, if $\eta_{ij} < n$, then $\eta_{ij}' = \eta_{ij}$, and the corresponding term in (\ref{eqn:7.7}) already occurs with the needed extra power of $U$. This shows that $f$ satisfies (1) and (2) of Definition~\ref{def:almostchain}.

It thus remains to show that $f'$ satisfies the $\omega$-condition. For this, we first observe that $f'(y_k') \equiv f(y_k)' \bmod U$. Indeed, if $\eta_k < n$, then this congruence is an equality. If $\eta_k \geq n$, then a comparison of (\ref{eqn:7.5}) and (\ref{eqn:7.6}) shows that $f(y_k)'$ and $f'(y_k')$ differ precisely in those terms for which $\eta_{ij} < n$. For such an index pair, the corresponding exponent $U^*$ must satisfy $* + \eta_{ij} - \eta_k \geq 0$, by (\ref{eqn:7.4}). This means that $* > 0$, since $\eta_k \geq n > \eta_{ij}$. Reduction modulo $U$ thus gives the desired congruence. This shows that $f'(e') \equiv f(e)' \bmod U$ for all $e$ in $\cC$.

Now, the shift correspondence $\cC \mapsto \shift_n(\cC)$ commutes with $\omega$. By Lemma~\ref{lem:shiftcor}, the shift correspondence $\cC_1 \otimes \cC_2 \mapsto \shift_n(\cC_1) \otimes \shift_n(\cC_2)$ is congruent to the tensor product correspondence $\bmod$ $U$. It is easy to check that the tensor product correspondence commutes with $\omega$; hence the shift correspondence $\cC_1 \otimes \cC_2 \mapsto \shift_n(\cC_1) \otimes \shift_n(\cC_2)$ commutes with $\omega \bmod{U}$. The $\omega$-condition for $f'$ then follows from the previous paragraph, together with the fact that $f$ commutes with $\omega \bmod{U}$. Observing that $f'(x')$ is a $U$-nontorsion cycle of degree zero, applying Lemma~\ref{lem:almostchaingood} completes the proof.
\end{proof}

We thus finally obtain the desired theorem:

\begin{theorem}\label{thm:shifthomomorphism}
For any $n\geq 1$, the shift map $\shift_n$ is a homomorphism from $\smash{\mfIhat}$ to $\smash{\mfIhat}$.
\end{theorem}
\begin{proof}
Let $X$ and $Y$ be any two almost $\inv$-complexes. We wish to show that
\[
\shift_n(X \otimes Y) \sim \shift_n(X) \otimes \shift_n(Y).
\]
By Theorem~\ref{thm:parameterization}, up to local equivalence we may replace $X$ and $Y$ with standard complexes $\cC_1$ and $\cC_2$. Similarly, we may replace $X \otimes Y$ with some standard complex $\cC$. Let $f\co \cC \rightarrow \cC_1 \otimes \cC_2$ be a local map. By Lemma~\ref{lem:localshiftmap}, we then have a local map 
\[
f' \co \shift_n(\cC) \rightarrow \shift_n(\cC_1) \otimes \shift_n(\cC_2),
\]
showing that $\shift_n(X \otimes Y) \leq \shift_n(X) \otimes \shift_n(Y)$ for any $X$ and $Y$. To prove the reverse inequality, we apply this to show
\[
\shift_n(Y) = \shift_n(X \otimes Y \otimes X^\vee) \leq \shift_n(X \otimes Y) \otimes \shift_n(X^\vee).
\]
It is easily checked that $\shift_n(X^\vee) = \shift_n(X)^\vee$ by an explicit consideration of how the standard complex parameters change under $\shift_n$ and dualization. Tensoring both sides of the above inequality with $\shift_n(X)$ thus yields $\shift_n(X) \otimes \shift_n(Y) \leq \shift_n(X \otimes Y)$. This completes the proof.
\end{proof}

\begin{comment}
\begin{remark}
In fact, for any (injective) order-preserving map $\sigma \from \Z_{>0}\to \Z_{>0}$, there is a shift homomorphism $\shift_\sigma \from \mfIhat\to \mfIhat$.  Allowing possibly order-reversing maps defines homomorphisms from certain subgroups of $\mfIhat$ to $\mfIhat$. \todo{JH: The second sentence in this remark is a little confusing. I think it's meant to refer to inverses (restricted to the image) of the maps from the first sentence, but if so, aren't those still order-preserving?}
\end{remark}
\end{comment}

%%%%%%%%%%%%%%%%%%%%%%%%%%%%%%%%%%%%%%%%%%%%%%%%%%%%%%%%%%%%%%%%%%%%%%%%%%%%%%%%%%%%%%%%%%%%%%%%%%%%%%%%%%%%%%%%%%%%%%%%%%%%%%%%%%%%%%%%%%%%%%%%%%%%%%%%%%%%%%%%%%%%%%%%%%%%%%%%%%%%%%%%%%%%%%%%%%%%%%%%%%%%%%%%%%%%%%%%%%%%%%%%%%%%%%%%%%%%%%%%%%%%%%%%%%%%%%%%%%%%%%%%%%%%%%%%%

\subsection{The pivotal homomorphism}\label{sec:pivot-homomorphism}

We now define the second important auxiliary homomorphism of this section. 

\begin{definition}\label{def:pivot}
Let $\cC = \cC(a_i, b_i)$ be a standard complex of length $2n$. We define $P(\cC) \in \Z$ to be the grading of the final generator $T_{2n}$, and extend $P$ to a map from all of $\smash{\mfIhat}$ to $\Z$ by passing to the standard complex representative in each local equivalence class. We refer to $P$ as the \textit{pivotal homomorphism}.
\end{definition}

In order to prove that $P$ is in fact a homomorphism, it will be convenient to immediately recast the definition of $P$, as follows:

\begin{definition}\label{defn:omega-homology}
Let $\cA$ be any (reduced) almost $\inv$-complex. Denote the action of $\omega$ on $\cA/U$ by $\smash{\omegahat}$. Noting that $\smash{\omegahat}^2 = 0$, we define the \textit{$\omega$-homology} of $\cA$ to be
\[
\omegahom(\cA)=\ker \omegahat/\im \omegahat.
\]
This is a graded $\ff$-vector space. Note that if $\cA_1$ and $\cA_2$ are two almost $\inv$-complexes and $f$ is an almost $\inv$-map between them, then $f$ induces a map from $\omegahom(\cA_1)$ to $\omegahom(\cA_2)$. This follows from the fact that $\omega f + f \omega \equiv 0 \bmod U$.
\end{definition}

\begin{example}\label{ex:standardpivot}
It is clear that if $\cC$ is a standard complex, then $\omegahom(\cC)$ is isomorphic to $\ff$, supported precisely in degree $P(\cC)$. Indeed, we will use this as a characterization of $P(\cC)$ (when $\cC$ is a standard complex).
\end{example}

It is easy to see through some simple examples that $\omegahom$ is not an invariant of local equivalence. (Consider any almost $\inv$-complex, and introduce a $U$-torsion tower with no $\omega$-arrows going in or out.) However, we do have the following lemma:

\begin{lemma}\label{lem:omegamapnonzero}
Let $\cA$ be any almost $\inv$-complex. Let $\cC$ be the standard complex in the local equivalence class of $\cA$, and let $f \co \cC \rightarrow \cA$ be a local map. Then the induced map
\[
\tilde{f}: \omegahom(\cC) \cong \ff \rightarrow \omegahom(\cA)
\]
is nonzero.
\end{lemma}
\begin{proof}
Let $\cC = \cC(a_i, b_i)$ be of length $2n$. By Lemma~\ref{lem:notimU}, $f(T_{2n})$ cannot be in $\im U$. Hence $f(T_{2n})$ is nonzero in $\cA/U$. To show that this class remains nonzero in $\omegahom(\cA)$, it suffices to prove that $f(T_{2n})$ is not in the image of $\omega \bmod U$. This follows from a similar argument as in the extension lemma. Suppose that we did have some $\tau_{2n+1} \in \cA$ for which $\omega \tau_{2n+1} \equiv f(T_{2n})$. Write $\partial \tau_{2n+1} = U\tau_{2n+2}$ for some $\tau_{2n+2} \in \cA$ (possibly zero). Then we have a local short map
\[
\cC(a_1, b_1, \ldots, a_n, b_n, +, -1) \leadsto \cA
\]
sending $T_{2n+1}$ to $\tau_{2n+1}$ and $T_{2n+2}$ to $\tau_{2n+2}$. By the extension lemma, this extends to a genuine local map into $\cA$, contradicting the maximality of $\cC$.
\end{proof}

In order to prove that $P$ is a homomorphism, we will explicitly compute the $\omega$-homology of $\cC_1$ and $\cC_2$ and show that it is one-dimensional. By Lemma~\ref{lem:omegamapnonzero}, this suffices to give a computation of $P(\cC_1 \otimes \cC_2)$.

\begin{lemma}\label{lem:omegatensor}
Let $\cC$ and $\cC'$ be standard complexes of length $2m$ and $2n$, respectively. Then $\omegahom(\cC \otimes \cC')$ is isomorphic to $\ff$ and is generated by $T_{2m} \otimes T_{2n}'$.
\end{lemma}
\begin{proof}
The $\omega$-complex for $\cC/U$ is generated by $T_0, \ldots, T_{2m}$, with an $\omega$-relation linking each consecutive pair of generators $T_{2i}$ and $T_{2i+1}$ for $0 \leq i \leq m - 1$. This, together with the analogous picture for $\cC'$, is displayed along the horizontal and vertical axes in Figure~\ref{fig:omegatensor}. It is straightforward to explicitly compute the $\omega$-action on $(\cC \otimes \cC')/U$, keeping in mind that the action of $\omega$ on the tensor product is given by
\[
1 \otimes \omega + \omega \otimes 1 + \omega \otimes \omega.
\]
Indeed, we have the following sample computation. Suppose $\omega T_p = T_q$ and $\omega T_r' = T_s'$. Then the $\omega$-relations among the four tensor product generators are:
\begin{align*}
&\omega(T_p \otimes T_r') = T_p \otimes T_s' + T_q \otimes T_r' + T_q \otimes T_s' \\
&\omega(T_p \otimes T_s') = T_q \otimes T_s' \\
&\omega(T_q \otimes T_r') = T_q \otimes T_s', \text{ and} \\
&\omega(T_q \otimes T_s') = 0.
\end{align*}
It is then clear from Figure~\ref{fig:omegatensor} that the only nonzero class in $\omegahom(\cC \otimes \cC')$ is represented by $T_{2m} \otimes T_{2n}'$.
\end{proof}

\begin{figure}[h!]
\center
\includegraphics[scale=0.85]{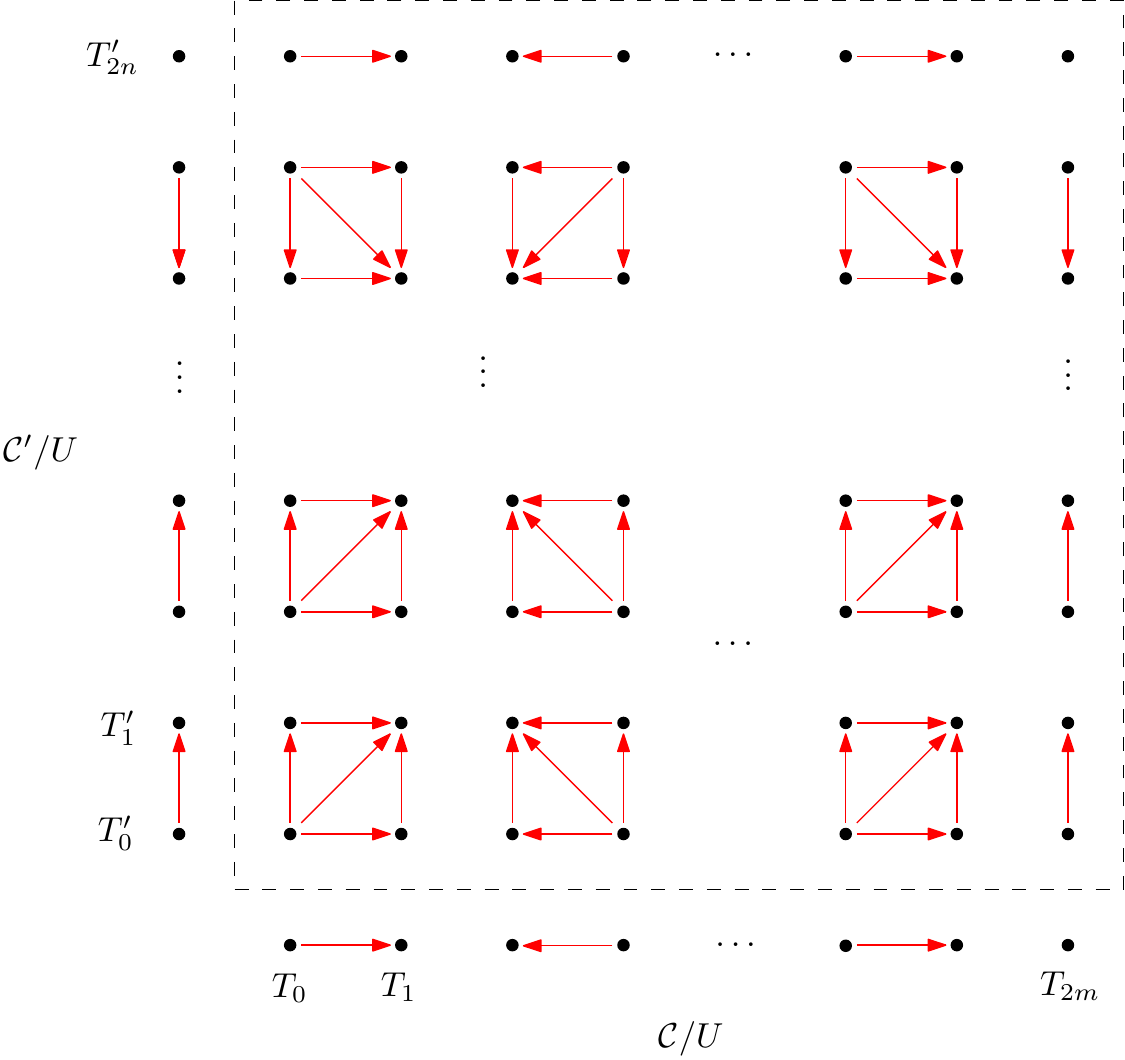}
\caption{Schematic representation of the $\omega$-action on $(\cC \otimes \cC')/U$. Red arrows depict the action of $\omega$. The complex for $\cC/U$ is displayed along the horizontal axis; $\cC'/U$ is displayed along the vertical axis. The complex for $(\cC \otimes \cC')/U$ is displayed inside the dotted box, with the tensor product generators placed in the obvious way.}
\label{fig:omegatensor}
\end{figure}

We thus have:

\begin{theorem}
The pivotal map $P$ is a homomorphism from $\smash{\mfIhat}$ to $\Z$.
\end{theorem}
\begin{proof}
Let $\cC_1$ and $\cC_2$ be standard complexes. To compute $P(\cC_1 \otimes \cC_2)$, let $\cC$ be the standard complex in the local equivalence class of $\cC_1 \otimes \cC_2$. Then $P(\cC_1 \otimes \cC_2)$ is equal to the grading of the nonzero generator of $\omegahom(\cC)$, as explained in Example~\ref{ex:standardpivot}. Now, by Lemma~\ref{lem:omegatensor}, $\omegahom(\cC_1 \otimes \cC_2)$ is one-dimensional and is supported by an element in grading $P(\cC_1) + P(\cC_2)$. By Lemma~\ref{lem:omegamapnonzero}, we have a grading-preserving isomorphism between $\omegahom(\cC)$ and $\omegahom(\cC_1 \otimes \cC_2)$. Hence $P(\cC_1 \otimes \cC_2) = P(\cC) = P(\cC_1) + P(\cC_2)$. This completes the proof.
\end{proof}

The importance of the pivotal homomorphism lies in the following observation. Let $\cC = \cC(a_i, b_i)$ be any standard complex. Since $P(\cC)$ is just the grading of the final generator of $\cC$, it is easily checked that
\[
P(\cC)=\sum^\infty_{i=1} (-2i+1)\phi_i(\cC).
\]
As the left-hand side is a homomorphism, it is not too unreasonable to expect each of the $\phi_i$ to be a homomorphism also. We formalize this intuition using the following useful relation between $P$ and the shift maps $\shift_n$, which follows easily from Definition~\ref{def:shiftmap}:
\begin{align}\label{eqn:7.8}
P(\shift_n(\cC))&=\sum_{i=1}^{n-1} (-2i+1)\phi_i(\cC)+\sum_{i=n}^\infty (-2(i+1)+1)\phi_i(\cC)\\
&=P(\cC)-2\sum_{i=n}^\infty \phi_i(\cC). \nonumber
\end{align}

\begin{theorem}\label{thm:submain}
Each $\phi_n$ is a homomorphism from $\smash{\mfIhat}$ to $\Z$.
\end{theorem}
\begin{proof}
Let $X$ and $Y$ be any two elements of $\smash{\mfIhat}$. Without loss of generality we assume that $X$ and $Y$ are standard complexes. We prove the claim by downward induction on $n$. Clearly, for all $n$ sufficiently large, we have $\phi_n(X) =\phi_n(Y) = \phi_n(X\otimes Y) = 0$. Thus, suppose that we have shown $\phi_n(X \otimes Y) = \phi_n(X) + \phi_n(Y)$ for all $n \geq N$. By (\ref{eqn:7.8}), we have the string of equalities:
\begin{align*}
-\sum_{i=N-1}^\infty 2\phi_i(X)-\sum_{i=N-1}^\infty  2\phi_i(Y) &= P(\shift_{N-1}(X)) - P(X) + P(\shift_{N-1}(Y)) - P(Y)\\
&=P(\shift_{N-1}(X \otimes Y)) -P(X\otimes Y)\\
&= - \sum^\infty_{i=N-1} 2\phi_i(X\otimes Y),
\end{align*}
where in the second line we have used the fact that $P$ and $\shift_{N-1}$ are homomorphisms. Using the induction hypothesis, we obtain the equality
\[
-2\phi_{N-1}(X)-2\phi_{N-1}(Y)=-2\phi_{N-1}(X\otimes Y).
\]
This completes the proof.
\end{proof}

The fact that the $\phi_n$ are homomorphisms immediately gives the following computation:

\begin{theorem}\label{thm:hhatzinfty}
For any $i, j > 0$, we have
\[
\phi_j(\cC(-, i)) = \delta_{ij},
\]
where $\delta_{ij}$ is the Kronecker delta. In particular, the forgetful homomorphism $\Inv \rightarrow \smash{\mfIhat}$ restricted to the subgroup $h(\Theta_{\text{SF}}) = h(\Theta_{\text{AR}})$ is injective, showing that
\[
\widehat{h}(\Theta_{\text{SF}}) = \widehat{h}(\Theta_{\text{AR}}) \cong \Z^{\infty}.
\]
\end{theorem}
\begin{proof}
The claim follows immediately from noting that the complexes $X_i$ of Example~\ref{ex:xicomplex} form a basis for $h(\Theta_{\text{SF}}) = h(\Theta_{\text{AR}}) \cong \Z^{\infty}$, together with the fact that the forgetful homomorphism maps $X_i$ to $\cC(-, i)$.
\end{proof}

%%%%%%%%%%%%%%%%%%%%%%%%%%%%%%%%%%%%%%%%%%%%%%%%%%%%%%%%%%%%%%%%%%%%%%%%%%%%%%%%%%%%%%%%%%%%%%%%%%%%%%%%%%%%%%%%%%%%%%%%%%%%%%%%%%%%%%%%%%%%%%%%%%%%%%%%%%%%%%%%%%%%%%%%%%%%%%%%%%%%%%%%%%%%%%%%%%%%%%%%%%%%%%%%%%%%%%%%%%%%%%%%%%%%%%%%%%%%%%%%%%%%%%%%%%%%%%%%%%%%%%%%%%%%%%%%%%%%%%%%%%%%%%%%%%%%%%%%%%%%%%%%%%%%%%%%%%%%%%%%%%%%%%%%%%%%%%%%%%%%%%%%%%%%%%%%%%%%%%%%%%%%%%%%%%%%%%%%%%%%%%%%%%%%%%%%%%%%%%%%%%%%%%%%%%%%%%%%%%%%%%%%%%%%%%%%%%%%%%%%%%%%%%%%%%%%%%%%%%%%%%%%%%%%%%%%%%%%%%%%%%%%%%%%%%%%%%%%%%%%%%%%%%%%%%%%%%%%%%%%%%%%%%%%%%%%%%%%%%%%%%%%%%%%%%%%%%%%%%%%%%%%%%%%%%%%%%%%%%%%%%%%%%%%%%%%%%%%%%%%%%%%%%%%%%%%%%%%%%%%%%%%%%%%%%%%%%%%%%%%%%%%%%%%%%%%%%%%%%%%%%%%%%%%%%%%%%%%%%%%%%%%%%%%%%%%%%%%%%%%%%%%%%%%%%%%%%%%%%%%

\section{Examples and further discussion}\label{sec:8}
We conclude this paper by discussing the group structure on $\smash{\mfIhat}$ and giving some possible applications to questions involving the span of Seifert fibered spaces in $\Theta^3_{\Z}$.

\subsection{Group structure}
In this subsection, we investigate the group structure on $\smash{\mfIhat}$ in terms of the standard complex parameters. Our main result will be to explicitly describe the group operation for all elements lying in $\smash{\widehat{h}}(\Theta_{\text{SF}})$.\footnote{Again, note that there are currently no known examples of homology spheres with image lying outside of $\smash{\widehat{h}}(\Theta_{\text{SF}})$.} As discussed in Theorem~\ref{thm:hhatzinfty}, this subgroup is abstractly isomorphic to $\Z^{\infty}$, with a basis given by the complexes $\cC(-, i)$ for $i \geq 1$. Our aim here will be to determine exactly which standard complexes lie in $\smash{\widehat{h}}(\Theta_{\text{SF}})$, and show that the group operation on $\smash{\widehat{h}}(\Theta_{\text{SF}})$ has a simple description in terms of the standard complex parameters.

Consider any sum
\[
\cC = \cC(\mp, \pm i_1) + \cC(\mp, \pm i_2) + \cdots + \cC(\mp, \pm i_m),
\]
where each $i_k > 0$, and every term is either of the form $\cC(-, i_k)$ or $\cC(+, -i_k)$. Without loss of generality, we assume that the $i_k$ are nonincreasing; i.e., $i_1 \geq i_2 \geq \cdots \geq i_m$. Moreover, we will further assume that $\cC$ is \textit{fully simplified}, meaning that we cancel all pairs of the form $\cC(-, i) + \cC(+, -i)$. The following theorem should be compared with Theorem 1.1 of \cite{Dai}, which computes the connected Heegaard Floer homology of elements of $\Theta_{\text{SF}}$. Indeed, Theorem~\ref{thm:seiferts} can be established via a careful analysis of the proof of \cite[Theorem 1.1]{Dai}, but here we present a self-contained argument in the language of almost $\inv$-complexes.

\begin{theorem}\label{thm:seiferts}
Let 
\[
\cC = \cC(\mp, \pm i_1) + \cC(\mp, \pm i_2) + \cdots + \cC(\mp, \pm i_m),
\]
be fully simplified with $i_1 \geq i_2 \geq \cdots \geq i_m$. Then the standard parameters of $\cC$ are given by concatenating the parameters of the above terms in the order they appear:
\[
\cC = \cC(\mp, \pm i_1, \mp, \pm i_2, \ldots, \mp, \pm i_m).
\]
\end{theorem}

\begin{proof}
To communicate the idea of the proof, it will be helpful to warm up with the simplest case when $m = 2$. Let $n$ and $\Delta$ be positive integers. We consider the following two cases:
\begin{align*}
&\cC_1 = \cC(-, n) + \cC(-, \Delta) \text{ with } n \geq \Delta, \text{ and}\\
&\cC_2 = \cC(+, -n) + \cC(-, \Delta) \text{ with } n > \Delta.
\end{align*}
The other two cases in which $\cC(-, \Delta)$ is replaced by $\cC(+, -\Delta)$ follow by dualizing. For both $\cC_1$ and $\cC_2$, the obvious tensor product basis consists of nine generators. These are displayed on the left in Figure~\ref{fig:seifertsfig1}, where they are labeled $a$ through $i$. We have drawn red arrows to represent the action of $\omega$, and black arrows for the action of $\partial$. Each black arrow is labeled by either $n$ or $\Delta$ to represent multiplication by a power of $U$; hence (for example) for $\cC_1$, we have $\partial i = U^n f + U^{\Delta} h$. 

On the right-hand side of Figure~\ref{fig:seifertsfig1}, we have depicted some convenient basis changes for $\cC_1$ and $\cC_2$. In the first row (corresponding to $\cC_1$), we carry out the basis change
\begin{align*}
&d' = d + b + e, \\
&g' = g + U^{n-\Delta}c + U^{n-\Delta}f, \text{ and} \\
&h' = h + U^{n - \Delta}f,
\end{align*}
keeping the other basis elements fixed (e.g., $a' = a$, and so on). The reader should (carefully) verify that this results in the diagram on the right-hand side of the first row. In the second row (corresponding to $\cC_2$), we carry out the basis change
\begin{align*}
&a' = a + b + e + U^{n-\Delta}i \text{ and} \\
&e' = e + U^{n-\Delta}i.
\end{align*}
There is a slight subtlety in this case, as the reader can check that $\omega e'$ is not actually equal to $b'$, and likewise $\omega d'$ is not equal to $a'$. However, since $n > \Delta$, it is easily established that these hold modulo $U$. (We have marked some of the red arrows with congruence symbols to represent this.) It is then evident from Figure~\ref{fig:seifertsfig1} that we have the local equivalences
\begin{align*}
&\cC_1 = \cC(-, n, -, \Delta) \text{ and} \\
&\cC_2 = \cC(+, -n, -, \Delta),
\end{align*}
as desired.

\begin{figure}[h!]
\center
\includegraphics[scale=0.7]{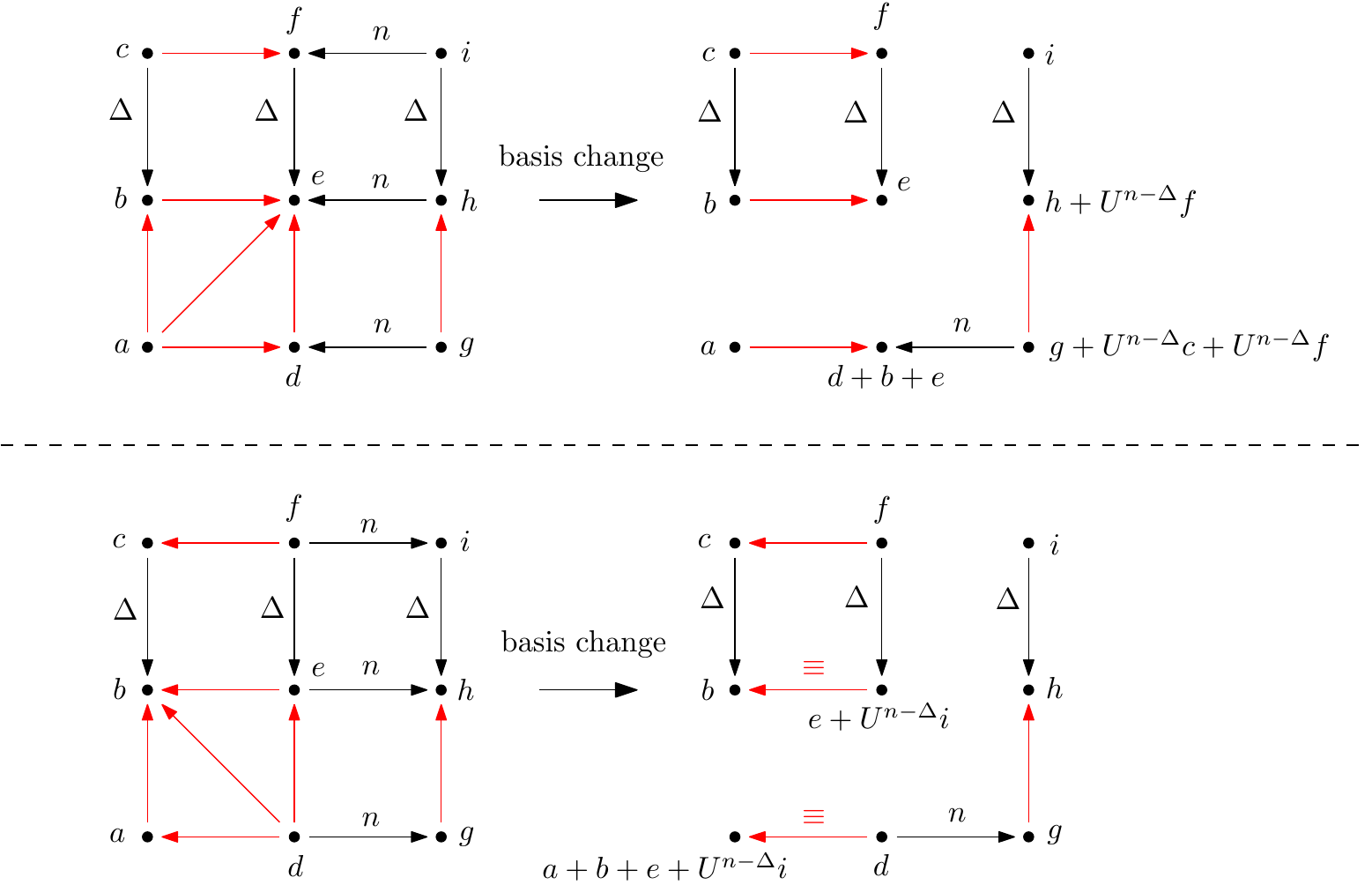}
\caption{Tensor product complexes for $\cC_1 = \cC(-, n) + \cC(-, \Delta)$ (top) and $\cC_2 = \cC(+, -n) + (- ,\Delta)$ (bottom). The tensor product basis is shown on the left, and an appropriate basis change is shown on the right.}
\label{fig:seifertsfig1}
\end{figure}

To establish the general case, we proceed by induction on $m$. Suppose that we have established the claim for 
\[
\cC = \cC(\mp, \pm i_1) + \cC(\mp, \pm i_2) + \cdots + \cC(\mp, \pm i_m),
\]
as in the statement of the theorem. Let $\Delta$ be a positive integer with $\Delta \leq i_k$ for all $1 \leq k \leq m$, and consider
\[
\cC' = \cC + \cC(-, \Delta).
\]
The case where we add $\cC(+, - \Delta)$ instead of $\cC(-, \Delta)$ again follows by dualizing. The obvious tensor product complex for $\cC + \cC(-, \Delta)$ is schematically depicted in the first row of Figure~\ref{fig:seifertsfig2}. Using the inductive hypothesis applied to $\cC$, this has $3(2m+1)$ generators. We have suppressed explicitly labeling the black arrows; implicitly, all of the vertical black arrows come with a label of $\Delta$, and the $m$ columns of horizontal black arrows are labeled by $i_1$ through $i_m$ (from left to right).

\begin{figure}[h!]
\center
\includegraphics[scale=0.7]{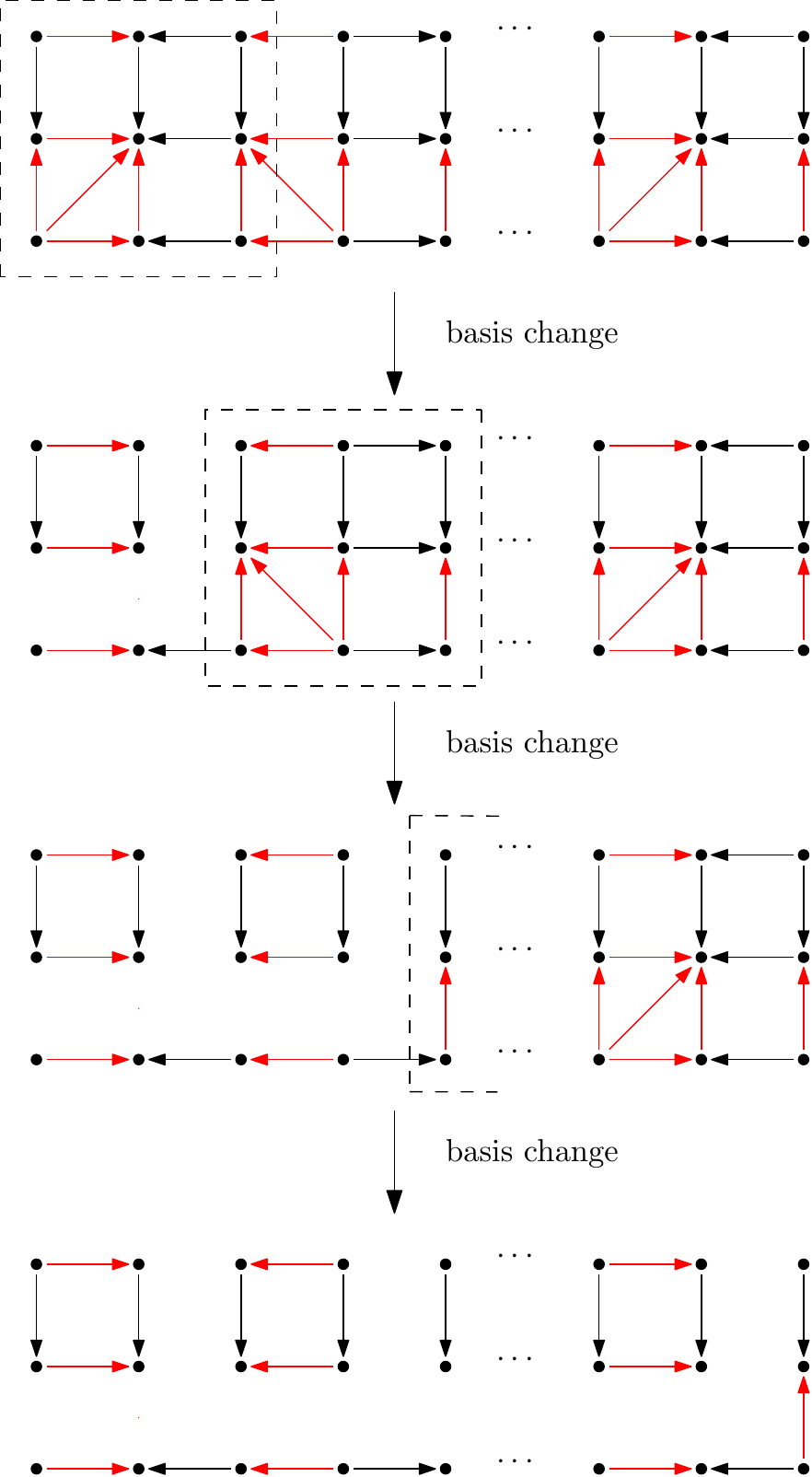}
\caption{Successive basis changes for $\cC + \cC(-, \Delta)$.}
\label{fig:seifertsfig2}
\end{figure}

Our strategy will be to attempt to split off subcomplexes by change-of-basis moves paralleling those defined for $\cC_1$ and $\cC_2$. We begin by considering the nine leftmost generators of $\cC'$, as indicated by the dashed box in the first row of Figure~\ref{fig:seifertsfig2}. Label these $a$ through $i$, as in Figure~\ref{fig:seifertsfig1}. Applying the appropriate basis change from either the first or the second row of Figure~\ref{fig:seifertsfig1} leads to the diagram in the second row of Figure~\ref{fig:seifertsfig2}. More precisely, if the first term of $\cC$ is $\cC(-, i_1)$, then we apply the basis change defined for $\cC_1$, while if the first term of $\cC$ is $\cC(+, -i_1)$, we apply the basis change defined for $\cC_2$. Note that in the first case, there is an additional subtlety: since we replace the generators $g$ and $h$ by $g'$ and $h'$ (respectively), we are in danger of changing the red arrows entering/exiting $g$ and $h$ on the right. To check that this does not happen, we consider two cases:
\begin{enumerate}
\item Suppose that there are red arrows entering $g$ and $h$ from the right. We claim that in this situation, we must have $i_1 > \Delta$. Indeed, because $\cC$ is maximally simplified, if $i_1$ were equal to $\Delta$, then all subsequent terms in our sum are identically equal to $\cC(-, \Delta)$. This implies that $g$ and $h$ have red arrows exiting them, rather than entering. Hence $i_1 > \Delta$. But this shows $g' \equiv g \bmod U$ and $h' \equiv h \bmod U$, which means that the original red arrows to the right of $g$ and $h$ hold modulo $U$.
\item Suppose that there are red arrows exiting $g$ and $h$ towards the right. Then we can explicitly check that the red arrows to the right of $g$ and $h$ are unchanged, using the fact that $\omega g' = \omega g + U^{n-\Delta}f = h' + (\omega g + h)$ and $\omega h' = \omega h$.
\end{enumerate}
We thus see that in either case, our basis move does not change the form of the diagram lying to the right of $g$, $h$, and $i$.

We now consider the nine generators lying inside the dashed box in the second row of Figure~\ref{fig:seifertsfig2}, re-labeling these $a$ through $i$ as usual. Again, we attempt to perform a basis change as in Figure~\ref{fig:seifertsfig1}. If the second term of $\cC$ is of the form $\cC(-, i_2)$, then we use the basis change defined for $\cC_1$, as before. Here, it is important to note that there are no arrows exiting $b$ and $c$ to the left, so that the changes $d' = d + b + e$ and $g' = g + U^{n - \Delta}c + U^{n - \Delta}f$ do not alter or create any additional arrows beyond those depicted in Figure~\ref{fig:seifertsfig1}. If the second term in $\cC$ is $\cC(+, i_2)$, then we attempt to use the basis change defined for $\cC_2$. However, there is now an additional subtlety, as depicted in Figure~\ref{fig:seifertsfig3}. The problem here is that we have a black arrow entering/exiting $a$ from the left, so when we set
\[
a' = a + b + e + U^{n-\Delta}i,
\]
we must ensure that we do not change the diagram to the left of the dashed box. If the black arrow to the left of $a$ is exiting $a$, then this follows from the fact that $\partial a' = \partial a$. However, if the arrow is instead entering $a$ (representing some relation $\partial p = U^ka$), then the diagram is no longer accurate, since evidently $\partial p \neq U^k a'$. In this situation, we thus carry out the additional (retroactive) basis change
\[
p' = p + U^{k - \Delta}c + U^{k - \Delta}f
\]
as in Figure~\ref{fig:seifertsfig3}, so that $\partial p' = U^k a'$. Note that here $\Delta < k$, by using the fact that $\cC$ is fully simplified. Hence $p' \equiv p \bmod U$, so (modulo $U$) our basis change does not change the red arrow exiting $p$. 

In any case, we see that performing the appropriate change-of-basis splits off another subcomplex and leads to a diagram as in the third row of Figure~\ref{fig:seifertsfig2}. Iterating this procedure results in the complex depicted in the bottom row of Figure~\ref{fig:seifertsfig2}. This is locally equivalent to 
\[
\cC(\mp, \pm i_1, \mp, \pm i_2, \ldots, \mp, \pm i_m, -, \Delta),
\]
as desired.
\begin{figure}[h!]
\center
\includegraphics[scale=0.7]{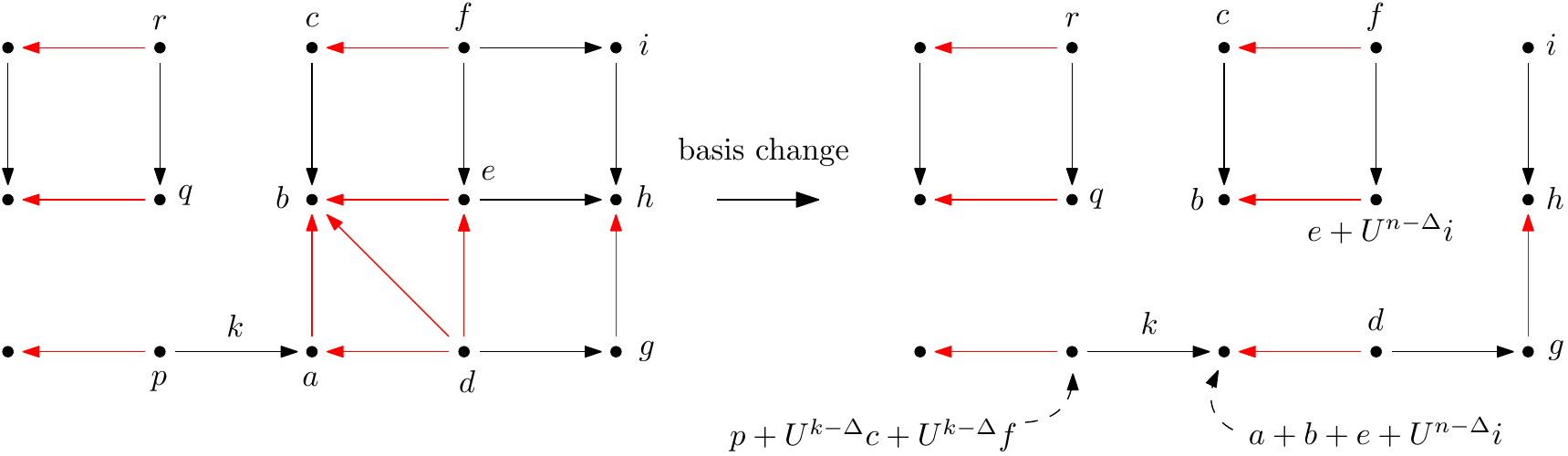}
\caption{A special case involving an additional (retroactive) basis change.}
\label{fig:seifertsfig3}
\end{figure}
\end{proof}

Theorem~\ref{thm:seiferts} gives an explicit description of the set of standard complexes realized by linear combinations of Seifert fibered spaces (or, more generally, almost-rational plumbed homology spheres). These consist of concatenations of nonincreasing, maximally simplified lists of basis elements $\cC(-, i)$ and $\cC(+, -i)$. Conversely, the sum of any two such standard complexes can be computed by decomposing them into their constituent terms $\cC(\mp, \pm i)$ and re-applying Theorem~\ref{thm:seiferts}.

However, in general the group operation on $\mfIhat$ is not so simple. We now give an example which does not obviously correspond to any sort of concatenation operation:

\begin{example}\label{ex:complicatedtensor}
Let $m$ and $n$ be positive integers, and assume without loss of generality that $m \leq n$. The tensor product complex corresponding to
\[
\cC = \cC(-, -m) + \cC(-, -n)
\]
is depicted on the left in Figure~\ref{fig:complicatedtensor}. On the right, we have displayed a basis change which shows that $\cC$ is locally equivalent to the standard complex
\[
\cC(-, -m, -, m, +, -n, -, -m).
\]
Note that $\phi_m(\cC(-,-m, -, m, +, -n, -, -m)) = -1 + 1 -1 = -1$, which is consistent with the fact that $\phi_m$ is a homomorphism. Example~\ref{ex:complicatedtensor} shows that the sum of two standard complexes of length $p$ and $q$ may be a standard complex of length greater than $p + q$.

\begin{figure}[h!]
\center
\includegraphics[scale=0.7]{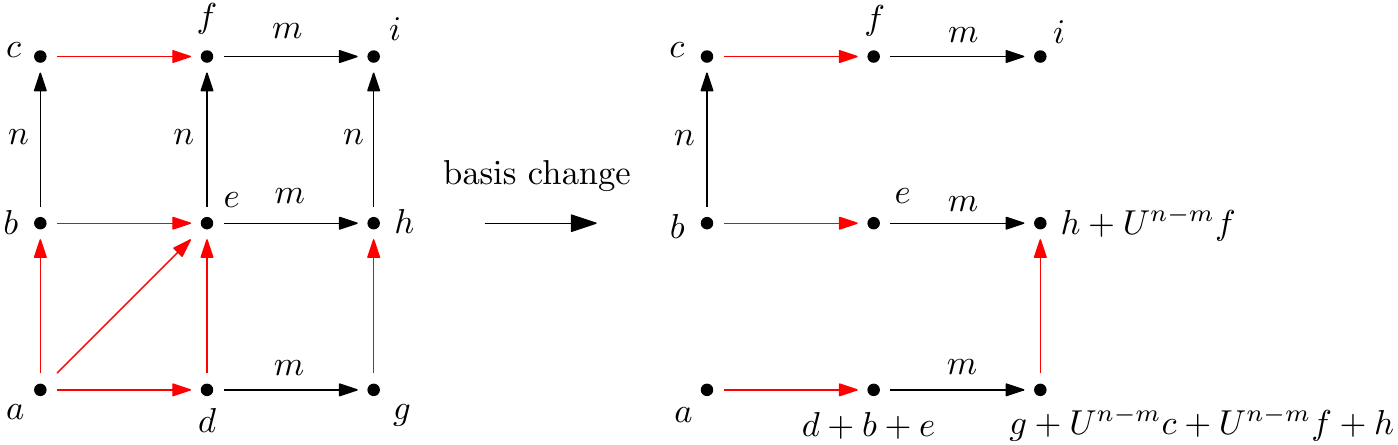}
\caption{The tensor product $\cC(-, -m) + \cC(-, -n)$, with $m \leq n$.}
\label{fig:complicatedtensor}
\end{figure}
\end{example}

There are certainly other cases in which one can compute the group law on $\mfIhat$, but the authors are not currently aware of any general description in terms of the standard complex parameters.

%%%%%%%%%%%%%%%%%%%%%%%%%%%%%%%%%%%%%%%%%%%%%%%%%%%%%%%%%%%%%%%%%%%%%%%%%%%%%%%%%%%%%%%%%%%%%%%%%%%%%%%%%%%%%%%%%%%%%%%%%%%%%%%%%%%%%%%%%%%%%%%%%%%%%%%%%%%%%%%%%%%%%%%%%%%%%%%%%%%%%%%%%%%%%%%%%%%%%%%%%%%%%%%%%%%%%%%%%%%%%%%%%%%%%%%%%%%%%%%%%%%%%%%%%%%%%%%%%%%%%%%%

\subsection{Concluding remarks}
We close this paper with some remarks regarding questions about $\Theta^3_\Z$ and $\mfIhat$. We begin by showing that the forgetful homomorphism from $\Inv$ to $\smash{\mfIhat}$ is not surjective:

\begin{theorem}\label{thm:forgetfulnotsurjective}
Let $\cA$ be an almost $\inv$-complex whose standard complex parameters begin with the pair $(-, -i)$ or $(+, i)$ for $i > 0$. Then $\cA$ is not in the image of the forgetful homomorphism $\Inv \rightarrow \smash{\mfIhat}$.
\end{theorem}

\begin{proof}
Assume $\cA$ is locally equivalent to a standard complex $\cC$ whose parameters begin with $(-, -i)$. Let $\mathcal{X}$ be a genuine $\inv$-complex, and suppose that $\mathcal{X}$ maps to $\cA$ under the forgetful homomorphism $\Inv \rightarrow \smash{\mfIhat}$. (More precisely, suppose that the local equivalence class of $\mathcal{X}$ maps to the local equivalence class of $\cA$.) Putting various maps together, this implies the existence of a map
\[
f\co \mathcal{X} \rightarrow \cC,
\]
which is a local map in the sense of almost $\inv$-complexes (viewing $\mathcal{X}$ as an almost $\inv$-complex). Let $x$ be any maximally graded, $U$-nontorsion cycle in $\mathcal{X}$. Then the fact that $f$ is local implies $f(x)$ must be supported by the generator $T_0$ of $\cC$. It easily follows that $f(\omega x)$ must be supported by $T_1 = \omega T_0$, and hence that $f(\partial \omega x) = \partial f(\omega x)$ must be supported by $\partial T_1 = U^i T_2$. In particular, this shows that $\partial \omega x \neq 0$. But since $\mathcal{X}$ is a genuine $\inv$-complex, we have $\partial \omega x = \omega \partial x = 0$, a contradiction. The case for $(+, i)$ follows by dualizing.
\end{proof}

While Theorem~\ref{thm:forgetfulnotsurjective} shows that there are many almost $\inv$-complexes which do not come from $\inv$-complexes, we do not currently have a complete description of the image of $\Inv$ in $\smash{\mfIhat}$. On the topological side, the extent of our knowledge is summarized in Figure~\ref{fig:summary}. Here, the only computations known to be realized by actual 3-manifolds are those for $\Theta_{\text{SF}}$ and $\Theta_{\text{AR}}$. One can thus attempt to use Theorem~\ref{thm:seiferts} to find a 3-manifold which is not cobordant to a linear combination of Seifert fibered spaces.

\begin{figure}[h!]
\center
\includegraphics[scale=1]{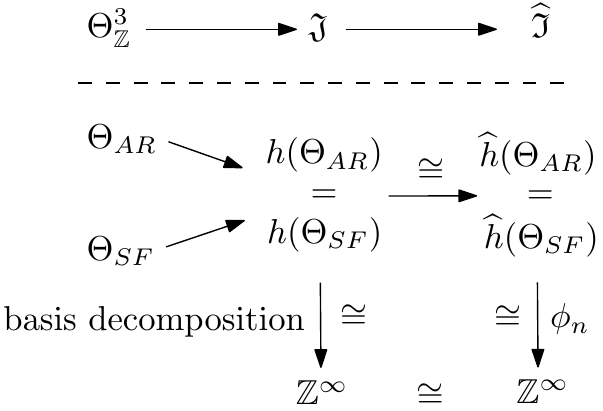}
\caption{Our current understanding of $\Inv$ and $\mfIhat$.}
\label{fig:summary}
\end{figure}

In a different direction, one can also consider the relationship between $\mfIhat$ and the different flavors of Heegaard Floer homology, in particular $\HFhat$ and $\HFm$. Indeed, an almost $\iota$-complex is equivalent to the following data:
\begin{enumerate}
	\item \label{it:iotahat} the involution $\hat{\iota}_* \co \HFhat(Y) \to \HFhat(Y)$ induced by $\iota$, and
	\item \label{it:les} the long exact sequence
\[
\begin{tikzcd}[column sep=small]
\HFm(Y)\arrow{rr}{U} & & \HFm(Y) \arrow{dl} \\
& \HFhat(Y). \arrow{ul} & 
\end{tikzcd}
\]
\end{enumerate}
Here, we require an explicit identification between the copies of $\HFhat(Y)$ appearing in items (1) and (2). Note that Hendricks and Lipshitz \cite{HendricksLipshitz} give an algorithm for computing the induced map
\[
\hat{\iota}_* \co \HFhat(Y) \to \HFhat(Y),
\]
whereas the authors are not aware of any general algorithm for computing either $\iota \co \CFm(Y) \to \CFm(Y)$ or $\iota_* \co \HFm(Y) \to \HFm(Y)$. Thus, it may be that $\smash{\widehat{h}}$ is easier to calculate than $h$. In particular, both the theoretical and computational tools appear to be in place to find a homology sphere that is not homology cobordant to a linear combination of Seifert fibered spaces, although specific examples remains elusive.

%%%%%

\bibliographystyle{amsalpha}
\bibliography{bib}

\end{document}